\documentclass[letterpaper,11pt]{amsart}
\usepackage[utf8]{inputenc}
\usepackage[english]{babel}
\usepackage{concrete}
\usepackage[T1]{fontenc}
\usepackage{csquotes}
\usepackage[colorlinks=true]{hyperref}
\usepackage{bookmark}
\usepackage{thmtools}
\usepackage{cleveref}
\usepackage{xspace}
\usepackage{dutchcal}
\usepackage{mathrsfs}
\usepackage{graphicx}
\usepackage{upgreek}
\usepackage[shortlabels]{enumitem}
\usepackage{ dsfont }
\usepackage{mathtools}
\usepackage{mathbbol}
\usepackage{stmaryrd}
\usepackage{float}
\usepackage{tikz}
\usetikzlibrary{cd}
\usepackage{moreenum}
\usepackage{quiver}
\usepackage{animate}
\usepackage{array}
\usepackage{phoenician}
\usepackage{xstring}
\usepackage{adjustbox}
\usepackage{appendix}
\usepackage{setspace}
\usepackage{svg}
\usepackage[obeyFinal]{todonotes}
\usetikzlibrary{babel}

\usepackage[letterpaper, total={8.5in, 11in},top=0.8in, inner=0.8in, outer=0.8in,bottom=0.8in]{geometry}
\newcommand{\field}[1]{\mathbb{#1}}
\newcommand{\mf}[1]{\mathfrak{#1}}

\newcommand{\mc}[1]{\ensuremath{\mathcal{#1}}}
\newcommand{\ul}[1]{\underline{#1}}
\newcommand{\wt}[1]{\widetilde{#1}}
\addto\captionsenglish{}
\theoremstyle{definition}
\newtheorem{definition}{\ul{Definition}}[section]
\newtheorem{proposition}[definition]{\ul{Proposition}}
\newtheorem{lemma}[definition]{\ul{Lemma}}
\newtheorem*{definition*}{\ul{Definition}}
\newtheorem*{informal_definition}{\ul{Informal definition}}
\newtheorem*{corollary*}{\ul{Corollary}}
\newtheorem{corollary}[definition]{\ul{Corollary}}
\newtheorem{principle}[definition]{\ul{Principle}}
\newtheorem{example}[definition]{\ul{Example}}
\newtheorem{question}[definition]{\ul{Question}}
\newtheorem{construction}[definition]{\ul{Construction}}
\crefname{construction}{Construction}{Constructions}
\crefname{notation}{Notation}{Notations}

\newtheorem{theorem}[definition]{\ul{Theorem}}

\newtheorem{notation}[definition]{\ul{Notation}}

\newtheorem*{theorem*}{\ul{Theorem}}
\theoremstyle{remark}
\newtheorem{remark}[definition]{\ul{Remark}}
\newtheorem*{remark*}{\ul{Remark}}
\newtheorem*{notation*}{\ul{Notation}}

\def\ve{\varepsilon}

\def\Alg{\mc{Alg}}
\def\Loc{\mc{Loc}}

\def\Cc{\mc{C}}

\def\C{\field{C}}

\def\RJ{\R\textrm{J}}
\def\N{\field{N}}

\def\ss{\subset}
\def\sse{\subseteq}

\def\Tc{\mc{T}}

\def\cd{\cdot}

\def\inft{{\operatorname{inf}}}

\def\Uc{\mc{U}}
\def\Z{\mathbb{Z}}
\def\R{\field{R}}
\def\Rc{\mc{R}}
\def\la{\langle}
\def\ra{\rangle}

\def\wcom{{\wedge_{\Com}}}

\def\red{{\operatorname{red}}}

\newcommand{\ol}[1]{\overline{#1}}
\newcommand{\wh}[1]{\widehat{#1}}

\newcommand{\PreStk}{\mc{PreStk}}
\newcommand{\Stk}{\mc{Stk}}

\DeclareMathOperator{\dist}{dist}
\DeclareMathOperator{\Carr}{Carr}
\DeclareMathOperator{\JT}{J}

\DeclareMathOperator{\Zar}{Zar}
\DeclareMathOperator{\Supp}{Supp}
\DeclareMathOperator{\Comp}{Comp}
\DeclareMathOperator{\Shape}{Shape}
\DeclareMathOperator{\Shv}{Shv}
\DeclareMathOperator{\Betti}{Betti}
\DeclareMathOperator*{\Tot}{Tot}

\DeclareMathOperator{\GM}{GM}

\DeclareMathOperator{\Spec}{Spec}
\DeclareMathOperator{\Specm}{Specm}

\DeclareMathOperator{\Map}{Map}

\DeclareMathOperator{\gr}{gr}

\def\bt{\bullet}
\def\m{\mf{m}}

\def\deg{\operatorname{deg}}

\DeclareMathOperator{\pd}{\partial}

\DeclareMathOperator{\Nat}{Nat}

\DeclareMathOperator{\id}{id}

\def\Com{\mc{Com}}

\DeclareMathOperator*{\colim}{colim}

\DeclareMathOperator{\Hom}{Hom}

\DeclareMathOperator{\Mod}{Mod}

\DeclareMathOperator{\oblv}{oblv}
\DeclareMathOperator{\free}{free}

\def\op{{\operatorname{op}}}
\def\K{\ensuremath{\mathcal{K}}}

\def\T{\ensuremath{\mathcal{T}}}

\DeclareMathOperator{\Fun}{Fun}

\def\O{\ensuremath{\mathcal{O}}}

\def\LL{\mathbb{L}}

\def\I{\ensuremath{\mathcal{I}}}
\def\dMfld{\mc{dMfld}}
\def\Mfld{\mc{Mfld}}

\def\vp{\varphi}

\def\M{\ensuremath{\mathcal{M}}}

\def\Spc{\mc{Spc}}
\newcolumntype{L}{>{$}l<{$}}
\def\temp{&}
\catcode`&=\active
\let&=\temp
\def\AA{\mathbb{A}}

\def\inf1{(\infty,1)}

\newcommand*{\CI}{{\mc{C}^\infty}}

\def\Lc{\mc{L}}

\def\X{\mc{X}}
\def\Xc{\mc{X}}
\def\Y{\mc{Y}}
\def\Yc{\mc{Y}}

\def\dR{{\operatorname{dR}}}
\def\dRcom{{\dR_{\Com}}}
\def\dRcinf{{\dR_{\CI}}}
\def\infcom{{\inft_{\Com}}}

\def\redcom{{\red_{\Com}}}
\def\redcinf{{\red_{\CI}}}

\def\cdR{\wh{\dR}}
\def\Fil{\operatorname{Fil}}

\DeclareRobustCommand{\phalphnum}[1]{
	\IfEqCase{#1}{%
		{1}{{\textphnc{a}}}%
		{2}{{\textphnc{b}}}%
		{3}{\textphnc{g}}%
		{4}{\textphnc{d}}%
		{5}{\textphnc{h}}
		{6}{\textphnc{f}}
		{7}{\textphnc{w}}
		{8}{\textphnc{z}}
		{9}{\textphnc{H}}
		{10}{\textphnc{T}}
		{11}{\textphnc{y}}
		{12}{\textphnc{k}}
		{13}{\textphnc{l}}
		{14}{\textphnc{m}}
		{15}{\textphnc{n}}
		{16}{\textphnc{s}}
		{17}{\textphnc{o}}
		{18}{\textphnc{p}}
		{19}{\textphnc{x}}
		{20}{\textphnc{q}}
		{21}{\textphnc{r}}
		{22}{\textphnc{S}}
		{23}{\textphnc{t}}
	}[\PackageError{\phalphnum}{Undefined option to phalphnum: #1}{}]%
}
\setlist[enumerate,1]{label={(\roman*)}}
\def\doi#1{\href{https://doi.org/#1}{doi:#1}}
\def\arXiv#1{\href{https://arxiv.org/abs/#1}{arXiv:#1}}
\def\web#1{\href{#1}{web}}

\def\ldRa{\texttt{ala}\xspace}

\begin{document}
	\title[de Rham Theory]{de Rham theory in derived differential geometry}
	\author{Gregory Taroyan}
	\thanks{The author's work was supported by Vanier Canada Graduate Scholarship, funding number CGV --- 192668.}
	\begin{abstract}
		This paper addresses the question: What is the de Rham theory for general differentiable spaces? We identify two potential answers and study them. In the first part, we show that the de Rham cohomology calculated using (the completion of) the exterior algebra of the cotangent complex yields non-trivial local invariants for singular differentiable spaces. In particular, in some cases, it differs from the constant sheaf cohomology, which provides an obstruction for the de Rham comparison map to be an equivalence. Moreover, we provide conditions under which this local invariant trivializes, yielding a de Rham-type isomorphism. In the second part, we show that for a suitably defined de Rham stack, there is always an isomorphism between functions on it and constant sheaf cohomology of the underlying topological space. Consequently, there exists a version of the de Rham theorem for singular differentiable spaces which holds with almost no restrictions. Finally, we sketch a generalization of this result to other theories of smooth functions, such as holomorphic or analytic functions. The last part is thus related to analytic de Rham stacks of Rodríguez Camargo used by Scholze to geometrize the local Langlands correspondence. 
	\end{abstract}
	\maketitle
	\setcounter{tocdepth}{1}
	\tableofcontents
	\setcounter{tocdepth}{3}
	\section{Introduction}
		\subsection{What is this paper about?} Our main goal is to develop a version of de Rham theory for derived \(\CI\)-manifolds and derived \(\CI\)-stacks. These, in particular, include arbitrary zero loci of smooth functions, diffeological spaces, quasifolds, etc. The basic premise of our approach is to import methods used in derived algebraic geometry to treat singular algebraic varieties and apply them to the study of singular differentiable spaces. 
		Below, we quickly review the essence of this algebraic approach to the de Rham theorem and then discuss the paper's main results.
		\subsubsection{}
		In \cite{Grothendieck_deRham_of_AV} Grothendieck gave a proof of the following remarkable result proposed by Atiyah--Hodge in \cite{Atiyah_Hodge}.
		
		{\it Let \(X\) be a \ul{smooth} affine algebraic variety over \(\C\). Then there is an isomorphism between the \ul{algebraic} de Rham cohomology of \(X\) and the \emph{analytic} de Rham cohomology of the associated complex manifold \(X(\C)\).}
		
		Combined with the usual de Rham theorem for complex manifolds, the above result is equivalent to the following statement.

		{\it 
			Let \(X\) be a \ul{smooth} affine algebraic variety over \(\C\). Then there is an isomorphism between the \emph{algebraic} de Rham cohomology of \(X\) and the cohomology of the constant sheaf \(\C\) of the associated complex manifold \(X(\C)\).
		}

		The key ingredient in this result is that the variety \(X\) is smooth. The reason is that we have access to easily controllable local models in the smooth setting. An additional input of the first theorem is Hironaka's resolution of singularities, which allows one to reduce the general case to a calculation in the proper case. The use of the resolution of singularities makes the proof of the above theorem extremely dependent on the context of algebraic geometry.

		Of course, for ordinary smooth manifolds, there is a much easier way to see that the de Rham complex calculates cohomology of the constant sheaf. One just needs checks that such an equivalence is present on stalks, and this reduces to the usual Poincar\'{e} lemma in \(\R^n\). There is, however, a trade-off to this simplicity of the result in the smooth case. If one wants to consider non-smooth spaces with differentiable structures, such as zero loci of general smooth functions, we lose access to the local models (\(\R^n\)'s) that work so well for manifolds.

 		On the other hand, in the world of algebraic geometry, despite the original result being much harder to prove, the non-smooth case is also accessible. The basic idea due to Hartshorne \cite{Hartshorne} is to replace a non-smooth variety \(X\) with its \enquote{tubular neighbourhood}. Of course, there is no notion of a tubular neighbourhood in algebraic geometry, so the thing that plays its role is the formal completion of an ambient, smooth variety along the singular variety \(X\). This approach yields the following result \cite[Corollary IV.1.2]{Hartshorne}.

		{\it
			Let \(X\) be an affine algebraic variety over \(\C\) immersed into a smooth variety \(Y\). Then there is an isomorphism between the algebraic de Rham cohomology of \(Y\) \emph{completed at} \(X\) and the cohomology of the constant sheaf \(\C\) on the associated analytic space \(X(\C)\).
		}

		Based on the above observations, it is thus interesting to see if there is a way to combine these two approaches to get a version of de Rham theory suitable for general differentiable spaces. In this work, we show that this question has different answers based on the interpretation. We show that a derived manifold's derived de Rham cohomology yields a non-trivial local invariant. Hence, while it is not isomorphic to the cohomology of the constant sheaf, it carries interesting information about the local structure of a derived manifold. In particular, the non-triviality of this invariant is related to the failure of a tubular neighbourhood of a derived manifold to be homotopy equivalent to it. We also show that there exists a refinement of de Rham cohomology which is isomorphic to the cohomology of the constant sheaf. This version uses a suitable intrinsic notion of the de Rham stack. 
		\subsection{What is done in this paper?}
		Below, we sketch the most important results of the paper. They speak to the key issue this text addresses. That is, derived \(\CI\)-manifolds can have highly singular underlying topological spaces, hence their local structure can be much more intricate than that of algebraic or analytic spaces. As a result of this non-triviality, local de Rham cohomology calculated using algebraic constructions can differ from purely topological constant sheaf cohomology. However, if a suitable \(\CI\)-refinement of de Rham cohomology yields the de Rham theorem by trivializing local invariants. 
		\subsubsection{Derived de Rham cohomology}
		\begin{definition*}[{\ref{con:derived_de_Rham}}] Let \(\M\) be a derived manifold. 
			Its derived de Rham cohomology \(\cdR_\M\) is a completion of a certain derived filtered commutative algebra. Let \(\LL_\M\) be the (absolute) cotangent complex of \(\M\), i.e. the derived version of \(1\)-forms. Then the associated graded pieces of \(\cdR_\M\) are given by the (shifted) exterior powers of \(\LL_\M\) of the  form \(\wedge^k\LL_\M[-k]\).
		\end{definition*}
		The local non-triviality of derived de Rham cohomology is witnessed by the following example.
		\begin{example}[{\ref{ex:de_Rham_counterexample}}]
			Consider the following null sequence on the real line. \[\left\{\frac{1}{n}\right\}_{n=1}^\infty\cup \{0\} \sse \R\] 
			We view it as a zero set of a nowhere negative function flat at the origin. We show that derived de Rham cohomology drastically differs from constant sheaf cohomology for this derived manifold. This difference stems from the failure of the ``Poincar\'{e} lemma'' for derived de Rham cohomology at the limit point of the sequence. 
		\end{example}
		The following hypothesis, however, ensures local triviality of derived de Rham cohomology and thus yields a version of the de Rham theorem.
		\begin{definition*}[{\ref{def:ldRa}}]
			A derived manifold \(\M\) is said to be \emph{algebraically locally acyclic} (or \ldRa) if for every point \(x\in \M\) the stalk of the derived de Rham cohomology is equivalent to \(\R\).
		\end{definition*}
		By \Cref{ala_is_determined_by_classical}, the condition on a derived manifold to be \ldRa depends only on the classical structure. That is, it is a property of the \(0\)-th truncation of the derived \(\CI\)-algebra of functions on \(\M\). 
		Examples of derived manifolds satisfying this condition include zero sets of analytic equations (\Cref{thm:de_Rham_for_analytic}), zero sets of subanalytic functions (\Cref{thm:de_Rham_for_Lojasiewicz}). Other, maybe more surpising, examples are locally \(\CI\)-contractible derived manifolds such as the spectrum of \(\CI(\R)/(e^{-\frac{1}{x^2}})\) (see \Cref{example_contraction}). 

		In this direction, we have the following result.
		\begin{theorem*}[{\ref{st:local_to_global_de_Rham}}]
			Let \(\M\) be an \ldRa derived manifold. Then, the derived de Rham cohomology of \(\M\) is isomorphic to the cohomology of the constant sheaf.
		\end{theorem*}
		\begin{remark}
			From an algebro-geometric perspective, the above theorem is somewhat surprising. Indeed, in the case of algebraic varieties, the derived de Rham cohomology is isomorphic to the cohomology of the constant sheaf, but only over the complex numbers. If we try to extend the above theorem to real algebraic varieties, we immediately get counterexamples even in the smooth case. For instance, degree zero algebraic de Rham cohomology of a hyperbola over the reals is one-dimensional, even though the constant sheaf cohomology is two-dimensional because there are two connected components. 
		\end{remark}

		\Cref{st:local_to_global_de_Rham} easily extends to the following more general class of stacks (see \S \ref{sec:geometric_stacks} for a formal definition), which includes quotients of derived manifolds by smooth group actions, as well as other ``derived Lie groupoids''.
		\begin{definition}[Informal]
			A \emph{geometric stack} is presentable by a (higher) groupoid internal to derived manifolds. 
		\end{definition}	
		We say that a geometric stack is \ldRa if it satisfies the following definition.
		\begin{definition}[{\ref{def:ala_for_geometric_stacks}}]
			Let \(\X\) be a geometric stack. Then we say that \(\X\) is \ldRa if there exists a smooth presentation of the form \(U\to \X\) with \(U\) being a coproduct of \ldRa derived manifolds.
		\end{definition}
		We note that by \Cref{prop:ala_is_presentation_independent} if a geometric stack is \ldRa, then it is \ldRa for any smooth presentation. The following theorem shows that the derived de Rham cohomology of a geometric stack is equivalent to the cohomology of the constant sheaf on it.
		\begin{theorem*}[{\ref{thm:de_Rham_for_geometric}}]
			Let \(\X\) be an \ldRa geometric derived differentiable stack. Then, the derived de Rham cohomology of \(\X\) is equivalent to the cohomology of the constant sheaf on \(\X\).
		\end{theorem*}
		As a particular case, when the stack is non-derived, i.e. presented by smooth manifolds, we recover the de Rham theorem for Lie groupoids proved by Behrend \cite{Behrend_De_Rham}. 

		\subsubsection{\texorpdfstring{\(\CI\)-refinement of de Rham cohomology}{C-infinity refinement of de Rham cohomology}}
		We have the following description for ordinary derived de Rham cohomology.
		\begin{theorem*}[{\ref{thm:de_Rham_vs_functions_on_inf}}]
			Let \(\X\) be a derived differentiable stack. Let \(\X^{\dR}\) denote its ``algebraic'' de Rham stack (\Cref{con:de_Rham_stacks}). Then the derived de Rham cohomology is equivalent to the derived algebra \(\CI(\X^{\dR}).\) 
		\end{theorem*}
		This suggests that one can obtain a different notion of de Rham cohomology for a differently defined de Rham stack. We show that for a suitable definition of the de Rham stack, one can always get an isomorphism with the cohomology of the constant sheaf.
		\begin{definition*}[{\ref{con:CI-derived-de-Rham}}]
			Let \(\X\) be a derived differentiable stack. Let \(\X^{\dR_{\CI}}\) denote its \(\CI\)-de Rham stack (\Cref{con:de_Rham_stacks}). Then its \emph{\(\CI\)-de Rham cohomology} is defined to be the algebra of \(\CI\)-functions on its \emph{\(\CI\)-de Rham stack} (\Cref{def:CI-de_Rham-coh}).  
		\end{definition*}
		\begin{theorem*}[{\ref{CI-infinity-de-Rham=const-sheaf}}]
			Let \(\M\) be a derived manifold. Then, the \(\CI\)-derived de Rham cohomology of \(\M\) (as a sheaf on the underlying topological space \(\M(\R)\)) is equivalent to the cohomology of the constant sheaf on the space \(\M(\R)\).
		\end{theorem*}
		The above result immediately implies the following a priori unexpected statement.
		\begin{corollary*}[{\ref{st:Hodge_filtration_on_the_constant_sheaf}}]
			A canonical (Hodge) filtration exists on the constant sheaf cohomology of a topological space underlying a derived manifold. This filtration is induced by the powers of the cotangent complex of the derived manifold. In particular, it is a derived version of the Hodge filtration on the constant sheaf cohomology of a smooth manifold.
		\end{corollary*}
		The corollary above shows that one has a canonical filtration on the constant sheaf cohomology for a large class of topological spaces. Hence, now it is easy to define variations of Hodge filtrations for derived manifolds. This could potentially provide interesting invariants of families.

		\Cref{CI-infinity-de-Rham=const-sheaf} immediately globalizes for arbitrary derived differentiable stacks, see \Cref{st:general_de_Rham_iso}. 
		These results suggest that a version of the Riemann--Hilbert correspondence holds for arbitrary derived differentiable stacks and an appropriate notion of \(D\)-modules. This closely aligns with the observations made by Scholze in their work on the geometrization of the local Langlands correspondence \cite{Scholze_Langlands} and the work of Rodríguez Camargo on analytic de Rham stacks \cite{Camargo}. In this direction, we observe that the above theorem is independent of a particular flavour of derived geometry and holds in the general setting of Fermat theories, see \Cref{sec:Fermatic}.
		\subsection{Why derived differential geometry?} Most of the results we prove are of interest just for closed subsets of \(\R^n\) equipped with a ring of smooth functions coming from the ambient space. So why bother to use the formalism of derived differential geometry, derived \(\CI\)-algebras and higher stacks? The answer is that we are naturally interested in the derived structure in the primary practical contexts, such as the study of moduli spaces of solutions to elliptic PDEs. For instance, in their recent work, Steffens \cite{Steffens_2024} shows how any such space of solutions is naturally a quasismooth derived manifold. 
		
		As a special case, we have the following fundamental problem arising in Floer's theory: we want to show that we have a moduli \emph{manifold} of solutions to the Floer equation (parametrizing flows between critical trajectories). This, in principle, requires a hard transversality result in the infinite-dimensional setting. However, suppose we are willing to work with \emph{derived manifolds}. In that case, we can reduce the problem to a much simpler Fredholm analysis question: is the moduli space of solutions finite-dimensional? Indeed, the derived structure on the non-transverse intersection captures most of the data we need to understand the moduli space. For instance, we can count the number of points in the moduli space by calculating the Euler characteristic of the corresponding derived algebra of functions. See \cite{Pardon}, \cite{Steffens_2024} for a discussion of this phenomenon in the language of derived manifolds, see \cite[Chapter 9]{Audin_Damian} for a textbook exposition of the classical approach. Moreover, to achieve similar results over an infinite-dimensional base or in equivariant settings, one is naturally drawn to using stacks as suitable global receptacles for this kind of locally derived data. This is one of the reasons why derived differential geometry is useful in practice.  

		Another reason is that most nice properties of de Rham cohomology remain true for singular spaces only if one passes to the appropriate derived versions. For example, contractibility invariance of derived de Rham cohomology is an essentially derived phenomenon, and we show that classical de Rham cohomology can be non-trivial even for locally contractible derived manifolds. This issue is discussed in \Cref{Appendix:Reiffen}, which provides a version of Reiffen's counterexample to the Poincaré lemma for singular hypersurfaces (\cite[\S 3]{Reiffen1967}) in the \(\CI\)-setting.
		\subsection{Contents} We now turn to a more systematic review of the text's contents. 

		\subsubsection{} \Cref{review_of_DDG} of the paper reviews a \enquote{light} version of the foundations of derived differential geometry. Over the past decade there was no shortage (\cite{spivak2010derived}, \cite{Borisov_Noel}, \cite{Carchedi_Steffens}, \cite{Nuiten}, \cite{BLX}, \cite{Steffens_thesis}, \cite{Carchedi_DG_manifolds}, \cite{Taroyan_23} to name some) of work on the subject of foundations, and we do not attempt to add another one to this list. Rather, we try to give all definitions as quickly as possible, somewhat sacrificing the theory's generality. In particular, we restrict our attention to derived manifolds in the sense of Spivak \cite{spivak2010derived} and Carchedi--Steffens \cite{Carchedi_Steffens}. However, importantly, the version of the theory that we introduce is as general as the original version of Spivak as is demonstrated in the foundational paper \cite{Carchedi_Steffens}, in particular, even though the definitions we use have a somewhat global appearence they are in fact fully local. We then define the category of derived stacks using the open cover topology on the category of derived manifolds. This version of the theory is justified not just by expository purposes but also by the fact that most of the constructions in higher algebra dealing with completions require a finite presentation condition. This condition is not satisfied by the general smooth loci. We return to this point in \Cref{questions_about_completions}. 
		
		We can not, however, avoid the general context of \(\CI\)-algebras because we need to consider non-finitely generated ideals in the rings of smooth functions. Such are, for example, ideals of functions vanishing at a general closed subset. Other examples of such non-finitely generated ideals are various radicals of ideals of smooth functions. These play an essential role in the technical part of our theory since the distinction between the \(\CI\)-refinement and commutative versions of the theory is based on the difference between the notion of \(\CI\) and the commutative radical of an ideal. We review these notions in \Cref{sec:reductions,sec:CI-algebras}.
        \subsubsection{} \Cref{completions_of_derived_CI_stacks} introduces the main higher algebra concepts we use in the text. There, we discuss several notions of completion for \(\CI\)-algebras. We show that these are equivalent under a suitable finiteness hypothesis. The results of this section are generally not surprising and are definitely known to the experts. The section exists mainly to alleviate psychological difficulties arising from considering completions of non-Noetherian rings and setting the notation.
		
		Importantly, in this section, we demonstrate (\Cref{prop:Adams_vs_Derived}) that the derived completion of a closed immersion of \(\CI\)-algebras can be calculated using a kind of (co)\v{C}ech resolution known as \emph{Adams completion}. This result is then used in the sequel to relate the derived de Rham cohomology with a Hartshorne-style cohomology of completion. The entirety of this and the following section is heavily inspired by the foundational work of Gaitsgory--Rozenblyum \cite{Gaitsgory_Rozenblyum_Ind} on derived algebraic geometry and essentially transports their results and ideas to the context of derived differential geometry.

		\subsubsection{} \Cref{derived_de_Rham} implements Bhatt's approach to derived de Rham cohomology \cite{Bhatt_Derived} in the setting of derived manifolds. There, we show that Hodge completed derived de Rham cohomology could be defined using Adams completion. We also show that the derived de Rham cohomology satisfies natural properties like homotopy invariance, which shows that for locally smoothly contractible derived manifolds (no matter how singular), the derived de Rham cohomology is equivalent to the cohomology of the constant sheaf. 
		\subsection{} \Cref{sec:de_Rham_theorems_for_cdR} investigates additional cases where the de Rham theorem based on derived de Rham cohomology holds or fails. In particular, we use the results of \Cref{completions_of_derived_CI_stacks} to show that the derived de Rham cohomology is not equivalent to the cohomology of the constant sheaf \(\R\) for some derived manifolds (\Cref{ex:de_Rham_counterexample}). On the other hand, we show that the de Rham theorem holds for analytic spaces and zero sets of subanalytic functions. This is a consequence of the highly regular local structure of these spaces. See the fundamental works of Bierstone--Milman \cite{Bierstone_Milman_Subanalytic,Bierstone_Milman} as well as the work of Brasselet--Pflaum \cite[\S 8]{Brasselet_Pflaum} for the kind of results on the local structure we are referring to.
		\subsubsection{} \Cref{sec:de_Rham_stacks} introduces de Rham stacks of Borisov--Kremnizer \cite{Borisov_Kremnizer} in the context of derived differentiable stacks. Importantly, we show (\Cref{thm:de_Rham_vs_functions_on_inf}) that the derived de Rham cohomology of \Cref{derived_de_Rham} is equivalent to the (derived) algebra of functions on the \enquote{algebraic} de Rham stack. Defining the \(\CI\) derived de Rham cohomology analogously as functions on the \(\CI\)-de Rham stack, we obtain a comparison morphism between the two theories. 
        \subsubsection{} \Cref{sec:de_Rham_theorems} contains the main abstract result of the paper: the \(\CI\)-derived de Rham cohomology is equivalent to the cohomology of the constant sheaf (\Cref{CI-infinity-de-Rham=const-sheaf}). We then put together the results of the previous sections to demonstrate that the failure of the \enquote{algebraic} de Rham theorem is a manifestation of the difference between the two kinds of de Rham stacks. When the derived manifold or stack is locally de Rham acyclic the functions on the two kinds of de Rham stack coincide. As shown by \Cref{ex:de_Rham_counterexample}, they could be different otherwise. We also globalize this result for derived \(n\)-geometric stacks in \Cref{thm:de_Rham_for_geometric}. 
		
		\subsubsection{} \Cref{sec:Fermatic} provides a vast generalization of the results in the paper to the setting of arbitrary Fermat theories. This includes holomorphic functions, polynomials, analytic functions, and others. We show that the dichotomy between the commutative and intrinsic notion of the de Rham stack persists in most of these contexts. However, in most known cases, the intrinsic notion of de Rham cohomology coincides with the commutative one. This, however, is highly non-formal and usually requires a type of \enquote{regularity} argument, such as Hironaka's resolution of singularities. Another point to make here is that, in principle, such a comparison should be possible not only between some Fermat theory and the theory of commutative algebras but between an arbitrary extension of Fermat theories, such as, for example, the theory of analytic and \(\CI\)-functions. We hope to address this idea in future work.
        
		One can interpret the results of \Cref{sec:Fermatic} in the following way. There is a different notion of de Rham stack for each kind (Fermat theory) of algebras of functions, and only this notion can a priori yield a type of the ``de Rham theorem'' for schemes of this type. Moreover, in the context of \(\CI\)-algebras, the \(\CI\)-de Rham theorem hints at the existence of a different notion of differential operators and \(D\)-modules, which is a proper place for the study of the derived de Rham cohomology. These differential operators morally correspond to the derivatives with respect to functions that can decay faster than polynomials. To our knowledge, the first proposal that this should be the case was made by Borisov--Kremnizer in \cite{Borisov_Kremnizer}. The study of these objects and their connections to the existing theory of analytic de Rham stacks of Rodríguez Camargo \cite{Camargo} and other approaches to analytic geometry \cite{Clausen_Scholze}, \cite{Ben_Bassat_Kremnizer} is a topic for future research.
		 
		\subsection{} In \Cref{Appendix:Reiffen} we discuss a classical counterexample of Reiffen to the holomorphic Poincaré lemma \cite[\S 3]{Reiffen1967}. We observe that the same example works for the \(\CI\)-de Rham cohomology. This shows that the de Rham theorem for derived manifolds is a genuinely derived phenomenon.
		\subsection{Algebraic apparatus and notational conventions} In this section, we review the basic instruments from category theory and higher algebra that we use. We also introduce some notational conventions present in the text.
        \subsubsection{} This paper is written using the language of higher categories (in the sense of Lurie \cite{Lurie_HTT}, \cite{Lurie_HA}). We try to formulate everything model-independently, but a favourite model, such as quasicategories or complete Segal spaces, can be used for concreteness. Also, all constructions can be made using model categories or even just explicit differential-graded algebras with special properties. We choose to work with higher categories because we believe the exposition is more streamlined in this language and, thus, more accessible to a broader audience. 
		
		We recognize that we sacrifice some transparency by adopting this approach. Still, we think that the explicit models existent in the literature (\cite{AKSZ}, \cite{Borisov_Noel}, \cite{Carchedi_Roytenberg}, \cite{BLX}, \cite{Carchedi_DG_manifolds}) can be easily used to translate our results into a more explicit language if need be.
        \subsubsection{} We use the formalism of derived \(\CI\)-algebras to describe derived manifolds. That is, we consider functors from the category \(\CI\) to the category \(\Spc\) of homotopy types, which are weakly product preserving. Thus, it is not exactly the same category of simplicial \(\CI\)-algebras; however, they are equivalent as higher categories, and we do not distinguish between them. Another small point is that we call \(\CI\)-algebras objects that are usually called \(\CI\)-rings to make a point that they are algebras over the algebraic theory \(\CI\), this coincides with the terminology which is usually used for commutative algebras. The choice to use derived \(\CI\)-algebras rather than more explicit models is primarily motivated by the simplicity of the exposition. One could do all the constructions using differential graded algebras (see \cite{Carchedi_DG_manifolds}) or curved \(L_\infty[1]\)-algebras (see \cite{BLX}), but we choose to work with derived \(\CI\)-algebras because most of the constructions are conceptually easier. 
        \subsubsection{} To capture the sheaf theory of topological spaces underlying derived manifolds and derived differentiable stacks, we use the notion of shape. We interpret it following Lurie's approach to geometric topology \cite[\S 7]{Lurie_HTT}. The primary motivation for this is that in many cases of interest, such as quotient stacks for group actions with open orbits, a derived differentiable stack has no ``nice'' underlying topological space. In addition, this approach enables us to perform all constructions in a nice categorical way and connects this work to the recent developments in the geometric topology of Ayala--Francis--Tanaka \cite{Ayala_Francis_Tanaka}, Volpe \cite{Volpe_Six} and others.
		
		\subsection{Derived \texorpdfstring{\(\CI\)}{C-infinity}-stacks and other formalisms for differentiable spaces} To globalize our results, we use the formalism of higher stacks. As with many other things in this text, this is done to streamline the exposition and cover all constructions of practical interest. Below, we quickly explain how this formalism compares to other common approaches to differentiable spaces.
        \subsubsection{Diffeology} This formalism develops an approach to differentiable spaces that may not have enough distinct points in the usual sense. Examples of these are quotients for some group actions, leaf spaces of foliations, etc. See the foundational text of Iglesias-Zemmour \cite{Diffeology_I-Z} for a textbook exposition. 
        
 		In the language of derived differentiable stacks, diffeological spaces are a particular kind of stack. In these terms, a diffeological space is a \(0\)-truncated concrete derived differentiable stack with trivial derived structure. See \cite{Pavlov_Diffeo} for a review of this connection in the language of differentiable stacks. Finally, an interesting connection between diffeology and non-commutative geometry, and thus our set-up is provided by the work of Iglesias-Zemmour--Prato \cite{IZ-Prato}.
        \subsubsection{Quasifolds} Quasifolds are another notion aimed at describing spaces that are not manifolds in the usual sense. They were introduced in \cite{Prato_Quasifolds} and describe differentiable spaces which can be glued from copies of \(\R^n\)'s using smooth transition functions, which are not necessarily diffeomorphisms. In the language of derived differentiable stacks, quasifolds are a particular kind of \(1\)-truncated stacks with trivial derived structure. 
		\subsubsection{Subcartesian spaces and Sikorski's differential spaces} 
		Subcartesian space is another approach to describing singular space with a \(\CI\)-structure. Recently, there has been a renewed interest in the subject in the works of Lerman (\cite{Lerman_DF_a,Lerman_Cartan_b}) and Karshon--Lerman (\cite{Karshon_Lerman}). These are special kinds of \(\CI\)-schemes, which are point determined (see \cite{Karshon_Lerman}). As a result, subcartesian spaces can be treated using derived manifolds and stacks on derived manifolds. 
		\subsubsection{Banach and Fréchet manifolds}
		Banach or Fréchet manifolds can also be presented by derived differentiable stacks. They can be viewed as stacks on finite-dimensional derived manifolds despite their infinite-dimensional nature. As such, they could be treated using the formalism of this text. See \cite[\S 3]{Steffens_2024} for a detailed comparison of the two approaches. 
		\subsection{Acknowledgements} I would like to thank Nick Rozenblyum and Dmitri Pavlov for their guidance and support while writing this paper. I would also like to thank Marco Volpe for explaining the key results of modern geometric topology to me. 
		I am also grateful to Joost Nuiten and Pelle Steffens for insightful discussions on the foundations of derived differential geometry. I would  like to thank Dennis Borisov, David Carchedi, and Owen Gwilliam for their interest in my work and their support. I would also like to thank the following people for providing extremely valuable comments for the first version of this text, listed here in alphabetic order: Aidan Lindberg, James Munday, Elisa Prato, Jon Pridham, and Arun Soor.
		Finally, I thank my partner Alisa Chistopolskaya for her constant moral support and encouragement. This research is supported by Vanier Canada Graduate Scholarship, funding number CGV --- 192668.
	\part{Derived de Rham cohomology}
	\section{A review of derived differential geometry}\label{review_of_DDG}
	In this section, we review basic ideas of derived differential geometry. The theory presented here is based on the framework of \(\CI\)-algebras. These provide an axiomatic version of the notion of a \enquote{ring of smooth functions}. A detailed exposition of the theory and its generalizations can be found in the work of Steffens \cite{Steffens_thesis}, \cite{Steffens_2023}, Nuiten \cite{Nuiten}, Pridham \cite{Pridham_DG}, and others. A more synthetic approach based on universal properties of derived manifolds is also possible, see the work of Carchedi--Steffens \cite{Carchedi_Steffens}. The gist of this approach is given by their \cite[Theorem 1.4]{Carchedi_Steffens}, we recount below (\Cref{thm:uni_property_CS}).
	
	We choose a somewhat simplified version of the geometric theory to make the exposition more accessible. In particular, on the geometric side, we ignore non-finitely presented \(\CI\)-algebras, thus avoiding the discussion of germ-determined and fair \(\CI\)-algebras.
	
	Non-finitely presented \(\CI\)-algebras are, of course, extremely important for the theory, and we will use them extensively on the algebraic side. They are, however, not necessary to set up a sufficiently robust theory of derived manifolds and stacks. We also ignore several other possible topologies one can put on the category of derived manifolds, see \cite[Appendix 2]{Moerdijk_Reyes_book}, \cite[\S 6.1]{Taroyan_23} for a list. In principle, the results obtained here should be independent of such a choice, although we do not investigate this matter in sufficient detail. 
	A more thorough discussion of arbitrary \(\CI\)-algebras is relegated to \Cref{sec:CI-algebras}.
	\subsection{Derived manifolds and derived \texorpdfstring{\(\CI\)}{C-infinity}-stacks}\label{sec:intro_to_der_mfld}
	\begin{definition}[{Category \(\CI\)}]\label{def:category_CI}
		Consider a category \(\CI\) with objects being the non-negative integers corresponding to real Cartesian spaces \(\R^n\). The morphisms in \(\CI\) are given by \(\CI\) maps between these spaces. That is, we have the following sets of morphisms \(\CI(m,n)=\CI(\R^m,\R^n)\).
	\end{definition}
	\begin{definition}
		The category \(\CI\Alg\) of derived \(\CI\)-algebras is the category of product-preserving functors from \(\CI\) to the category \(\Spc\) of spaces. 
	\end{definition}
	\begin{remark}
		The above definition can be placed in the larger framework of Fermat theories, which also includes the theory of holomorphic functions and polynomials. For more, detail see \Cref{sec:Fermat_theories_def}. Even more broadly, this definition is an instance of the ``animation'' procedure, which is a way to systematically construct derived versions of mathematical objects, see for instance \cite[\S 5.1.4]{Cesnavicius_Scholze}.
	\end{remark}
	\begin{definition}\label{def:der_mfld}
		We define the category of \emph{derived manifolds} \(\dMfld\) as the full subcategory of \(\CI\Alg^\op\) spanned by derived pullbacks of the following form, as well as their finite limits and retracts.
		\[\begin{tikzcd}
			\M && {\R^n} \\
			\\
			{*} && {\R^m}
			\arrow[from=1-1, to=1-3]
			\arrow[from=1-3, to=3-3]
			\arrow[from=1-1, to=3-1]
			\arrow["O"', from=3-1, to=3-3]
			\arrow["\lrcorner"{anchor=center, pos=0.125}, draw=none, from=1-1, to=3-3]
		\end{tikzcd}\]
	\end{definition}
	\begin{remark}
		Equivalently, one may say that the category of \(\dMfld\) is opposite to the full subcategory of categorically compact (see \cite[Definition A.1.1.1]{Lurie_HTT}) derived \(\CI\)-algebras (see \Cref{def:fg_and_fp_CI}).
	\end{remark}
	Remarkably, the above definition of derived manifolds is fully local, even though in its construction we use objects which are global fibres of smooth maps between Cartesian spaces. Let us briefly explain why, simultaneously showing how to present every smooth manifold as a derived manifold in the sense of the above definition.
	\begin{example}[Smooth manifold as a derived manifold]\label{ex:smooth_mfld_derived}
		Consider a smooth manifold \(M\). By Whitney's embedding theorem, we may regard it as a closed submanifold of \(\R^N\) for some large enough \(N.\) By the tubular neighbourhood theorem, there exists an open subset \(M\sse U\sse \R^N\) which retracts onto \(M.\) Every open subset of \(\R^N\) can be presented as a closed subset in \(\R^{N+1}=\R^N_x\times \R_t\) given by the fibre of a single smooth function. Indeed, let \(\chi_U\) be the characteristic smooth function of an open subset \(U\), then \(U\) is diffeomorphic to the following global fibre.
		\[
		\begin{tikzcd}
			U\ar[r]\ar[d]& \R^{N+1}\ar[d, "\chi(x)t-1"]\\
			0\ar[r] & \R
		\end{tikzcd}
		\]
		Thus \(U\) is a derived manifold by the above argument and then by the tubular neighbourhood theorem the manifold \(M\) is a retract of \(U\) and thus is also a derived manifold.
	\end{example}

	In the sequel, we will restrict our attention to the category of derived manifolds. There is a more general notion of a derived \(\CI\)-scheme which relaxes finite presentation conditions of \Cref{def:der_mfld}; however, we will not consider it here. See \cite{Joyce_CInf_Geom} and \cite{Moerdijk_Reyes_book} for more details on this more general approach. The main reason is that completions of non-Noetherian rings are hard to understand. There exists an extensive literature on the subject, including \cite{Porta_et_al}, \cite{Yekuteli_WPR}, \cite{Yekuteli_Derived_Complete}, and \cite{Positselski}, to name a few.
	This leads to similar difficulties in building a robust theory of derived de Rham cohomology for non-Noetherian rings. 
	
	The following result of Carchedi--Steffens justifies the definition of derived manifolds given above.

	\begin{theorem}[{\cite[Theorem 1.4, Corollary 1.5]{Carchedi_Steffens}}]\label{thm:uni_property_CS}
		Let \(\Cc\) be an \(\infty\)-category with finite limits. Denote by \(\Fun^\pitchfork(\Mfld,\Cc)\) the \(\infty\)-category of \(\infty\)-functors preserving transverse pullbacks in the category \(\Mfld\) of manifolds. Then there is a canonical equivalence of categories induced by the inclusion of \(\Mfld\) into \(\dMfld\) as follows.
		\[
			\Fun^{\operatorname{lex}} (\dMfld,\Cc)\xrightarrow{\sim} \Fun^\pitchfork(\Mfld,\Cc).
		\]
		Here \(\Fun^{\operatorname{lex}}\) denotes the category of left exact functors from the category of derived manifolds. 
	\end{theorem}

	\begin{definition}
		We call the finite limit defining \(\M\) calculated in the category of topological spaces the \emph{underlying topological space} of \(\M\). We denote it by \(\M(\R)\) in analogy with the algebro-geometric situation. 
	\end{definition}
	\begin{remark}
		In \Cref{con:spectrum_of_a_CI_algebra}, we explain how the above definition can be used to produce a locally ringed space corresponding to a derived manifold. In particular, one can think of a derived manifold as a closed subset of \(\R^n\) with a sheaf of derived \(\CI\)-algebras encoding the smooth structure and infinitesimal data.
	\end{remark}
	\begin{example}\label{ex:non-zero_div_in_CI}
		Consider a smooth function \(f\) on \(\R^n\) such that its zero set is nowhere dense. Then the derived manifold corresponding to its zero locus corresponds to the following derived \(\CI\)-algebra presented as a differential graded \(\CI\)-algebra.
		\[
		\CI(\R^n)[\xi],\quad |\xi|=-1,\quad \pd\xi=f.
		\]
		Note that the condition on the zero set is necessary for \(f\) to be a non-zero divisor in the ring \(\CI(\R^n)\). 
	\end{example}
	\begin{definition}[{Coverage on \(\dMfld\)}]
		We define the coverage on the category \(\dMfld\) by considering families of morphisms inducing \emph{jointly surjective open embeddings} on the underlying topological spaces and having compatible derived structure.
	\end{definition}
	\begin{remark}
		It might not be quite clear why the above definition makes sense, so we briefly sketch the main idea. In \Cref{ex:smooth_mfld_derived} we have shown how to present an arbitrary open subset of \(\R^N\) as a fibre of single smooth function, the same reasoning applies here, by adding one more coordinate keeping track of invertibility one recovers an arbitrary open subset the topological space underlying a derived manifold as a derived manifold. See \cite[\S 5]{Carchedi_Steffens} for more detail.
	\end{remark}
	\begin{construction}\label{con:derived_CI_stacks}
		We define the category \(\CI\Stk\) of \emph{derived \(\CI\)-stacks} as the full subcategory of \(\Spc\)-presheaves in \(\CI\PreStk\) on \(\dMfld\) satisfying the sheaf condition with respect to the coverage defined above.
	\end{construction}
	Clearly, any closed subset of \(\R^n\) can be given a structure of a derived manifold. The following example shows that, similarly, one can present any open subset of \(\R^n\) as a derived manifold.
	\begin{example}[{\cite[Corollary 2.2]{Moerdijk_Reyes_book}}]\label{ex:open_subsets_of_derived_manifolds}
		Let \(U\) be an open subset of \(\R^n\), let \(\chi_U\) be a characteristic function of \(U\). Then \(\CI(U)\simeq \CI(\R^n\times\R_y)/(y\cd \chi(x)-1).\)	This is similar to the way one introduces a structure of an affine variety on an open subset of an affine variety in algebraic geometry. By, for instance, \cite[Proposition 4.20]{Carchedi_Steffens} every open subset of the underlying topological space of a derived manifold can be presented using a similar construction. 

		Similar to the algebro-geometric case, the case of an open subset in \(\R^n\) generalizes to the case of an arbitrary derived manifold and an open subset of its underlying topological space. Thus, taking open subsets of derived manifolds preserves the finite presentation condition that defines the category of derived manifolds. 
	\end{example}
	\subsection{Reductions and localizations of \texorpdfstring{\(\CI\)}{C-infinity}-algebras}\label{sec:reductions}

	We start by recalling the classical notion of localization for \(\CI\)-algebras.  See the classical textbook \cite[\S I.1]{Moerdijk_Reyes_book} for the discussion in the case of non-derived \(\CI\)-algebras. See \cite[Chapter 4]{Steffens_thesis} for the derived version of the same theory.
	\begin{definition}[{\(\CI\)-localization}]\label{def:CI-localization}
		Let \(A\) be a derived \(\CI\)-algebra. Consider an element \(s\in\pi_0A\). Then the \(\CI\)-localization of \(A\) in \(s\) is the derived \(\CI\)-algebra \(A\{s^{-1}\}\) which is the derived \(\CI\)-algebra presented by the following (derived) pushout.
\[\begin{tikzcd}
	{\CI(\R)} && {\CI(\R\setminus 0)} \\
	\\
	A && {A\{s^{-1}\}}
	\arrow[hook, from=1-1, to=1-3]
	\arrow["{x\mapsto s}"', from=1-1, to=3-1]
	\arrow[from=3-1, to=3-3]
	\arrow[from=1-3, to=3-3]
	\arrow["\lrcorner"{anchor=center, pos=0.125, rotate=180}, draw=none, from=3-3, to=1-1]
\end{tikzcd}\]
		Here \(x\) is the \(\CI\)-generator of \(\CI(\R)\). 
		
		More conceptually, the \(\CI\)-localization of \(A\) in \(s\) is the derived \(\CI\)-algebra satisfying the classical universal property of localization \emph{in the category of derived \(\CI\)-algebras}. That it is the initial object in the category of derived \(\CI\)-algebras equipped with a morphism from \(A\) inverting \(s\).
		It is a classical result in \(\CI\)-geometry
		 that the two definitions of \(\CI\)-localizations coincide. See \cite[Proposition 1.6]{Moerdijk_Reyes_book} for the classical version and \cite[Proposition 4.1.3.13]{Steffens_thesis} for the derived version. 
	\end{definition}
	\begin{example}
		As shown by \Cref{def:CI-localization} a prototypical example of a \(\CI\)-localization is the localization of the \(\CI\)-algebra \(\CI(\R)\) in the generator \(x\). This is the \(\CI\)-algebra of smooth functions on \(\R\setminus 0\). In this case, we can already see that this is quite far from being the same as the algebraic localization of \(\CI(\R)\) in~\(x\). The algebraic localization in this case would be the ring \(\CI\{x\}[x^{-1}]\) of rational functions with \(\CI\)-numerator and polynomial denominator. As such, it is not even a \(\CI\)-algebra.  
	\end{example}
	\begin{remark}[{\cite[Proposition 1.6]{Moerdijk_Reyes_book}}]
		There is also the following more intuitive way to think about \(\CI\)-localizations. Let \(A\) be a derived \(\CI\)-algebra. Pick some free presentation \(F\to A\). Then any element \(s\in \pi_0A\) lifts to an element in \(\wt{s}\in F\). One can think of \(F\) as an algebra of \(\CI\)-functions on the affine space \(\R^S\) for some set \(S\). Then the \(\CI\)-localization of \(A\) in \(s\) is the algebra of \(\CI\)-functions on the intersection of \(\Spec A\) and the complement of the zero locus of \(\wt{s}\) in \(\R^N\). 
	\end{remark}
	We can now define the notion of \(\CI\)-nilpotent elements in derived \(\CI\)-algebras. The definition is essentially the same as the classical one (see \cite{Moerdijk_Reyes_Localizations}), with the only difference being that we consider the elements in \(\pi_0 A\) for some derived \(\CI\)-algebra \(A\).
	\begin{definition}[{Nilpotent elements in \(\CI\)-algebras}]\label{def:nilpotent_elements} Let \(A\) be a derived \(\CI\)-algebra. We have two natural notions of nilpotent elements in \(A\).
		\begin{itemize}
			\item[\(\CI\)] An element \(s\in \pi_0A\) is called \emph{\(\CI\)-nilpotent} if the \(\CI\)-localization of \(A\) in \(s\) is trivial.
			\item[\(\Com\)] An element \(s\in \pi_0A\) is called \emph{\(\Com\)-nilpotent} if the algebraic localization of \(A\) in \(s\) is trivial. This is, of course, the usual notion of nilpotent elements.
		\end{itemize}
	\end{definition}
	\begin{definition}\label{definition_of_support}
		Let \(f\in \CI(\R^n)\) be a smooth function. We define the \emph{carrier set} of \(f\) as the set of points in \(\R^n\) where \(f\) is non-zero. We denote it by \(\Carr(f)\). Accordingly, the \emph{support set} of \(f\) is the closure of \(\Carr(f)\) in \(\R^n\) and is denoted by \(\Supp(f)\). Similarly, the \emph{zero set} \(Z(f)\) of \(f\) is the set of points in \(\R^n\) where \(f\) vanishes.
	\end{definition}
	\begin{remark}
		In terms of \Cref{definition_of_support} the \(\CI\)-nilpotent elements in a derived \(\CI\)-algebra \(A\) are precisely the elements \(s\in \pi_0A\) such that their zero set coincides with the zero set of some function in the ideal defining \(A\). For a detailed proof and exposition, see \cite[\S 2]{Moerdijk_Reyes_Localizations}.
	\end{remark}
	One can give a more general version of \Cref{def:nilpotent_elements} for an arbitrary ideal \(I\) in \(A\) and not just the zero ideal. There are, in fact, many other natural notions of radicals of ideals in \(\CI\)-algebras. These are listed and explored in \cite[\S 2]{Borisov_Kremnizer}.
	\begin{definition}[{Radicals of ideals in \(\CI\)-algebras}]\label{def:radicals_of_ideals}
		Let \(A\) be a derived \(\CI\)-algebra. Let \(I\) be an ideal in \(\pi_0 A\). We have two natural notions of the radical of \(I\).
		\begin{itemize}
			\item[\(\CI\)] An element \(s\in \pi_0A\) belongs to \(\sqrt[\CI]{I}\) if the \(\CI\)-localization of \(A/I\) in \(s\) is trivial.
			\item[\(\Com\)] An element \(s\in \pi_0A\) belongs to \(\sqrt[\Com]{I}\) if the algebraic localization of \(A/I\) in \(s\) is trivial. This is, of course, the usual radical of an ideal.
		\end{itemize}
	\end{definition}
	\begin{definition}
		Let \(A\) be a derived \(\CI\)-algebra. We have two natural notions of it being reduced.
		\begin{itemize}
			\item[\(\CI\)] A derived \(\CI\)-algebra \(A\) is called \emph{\(\CI\)-reduced} if the \(\CI\)-localization of \(A\) in any non-zero element is non-trivial and \(A\) is discrete (as a derived algebra).
			\item[\(\Com\)] A derived \(\CI\)-algebra \(A\) is called \emph{\(\Com\)-reduced} if the algebraic localization of \(A\) in any non-zero element is non-trivial and \(A\) is discrete (as a derived algebra). This is, of course, the usual notion of reducedness.
		\end{itemize}
	\end{definition}
	We can now define two reductions of a derived \(\CI\)-algebra.
	\begin{construction}[{Reductions of \(\CI\)-algebras}]\label{con:reductions_of_CI_rings}
		Let \(A\) be a derived \(\CI\)-algebra. We have two natural reductions of \(A\).
		\begin{itemize}
			\item[\(\CI\)] \(A^{\red_{\CI}}\) is the derived \(\CI\)-algebra given by the homotopy quotient and truncation \(\tau^{\ge 0} A/\sqrt[\CI]{0}\).
			\item[\(\Com\)] \(A^{\red_{\Com}}\) is the derived \(\CI\)-algebra given by the homotopy quotient and truncation \(\tau^{\ge 0} A/\sqrt[\Com]{0}\).
		\end{itemize}
	\end{construction}
	\begin{proposition}
		Let \(A\) be a derived \(\CI\)-algebra. Then, the \(\CI\)-reduction (respectively \(\Com\)-reduction) of \(A\) is the universal \(\CI\)-reduced (respectively \(\Com\)-reduced) derived \(\CI\)-algebra equipped with a map from \(A\).
	\end{proposition}
	\begin{proof}
		The fact that a \(\CI\) or \(\Com\) reduction is reduced immediately follows from the fact that taking the \(\CI\) or nilradical of an ideal is an idempotent operation by \cite[Lemma 4]{Borisov_Kremnizer}. The universal property of the reduction follows from the definition as an adjoint functor, as observed in \cite[Definition 3]{Borisov_Kremnizer}. 
	\end{proof}
	We believe that the following result was already known to experts, but we could not find it formulated separately. Despite its simplicity it is crucial for identifying different kinds of de Rham stacks in the context of derived manifolds.
	\begin{proposition}\label{st:CI-inf_radical=RJ-radical}
		Consider a finitely generated ideal \(I\) in a free ring \(\CI(\R^S)\). Then a function \(h\) belongs to the \(\CI\) radical of \(I\) if and only if it vanishes on the common zero set of the ideal \(I\).
	\end{proposition}
	\begin{proof}
		Consider a collection of generators \(g_1,\ldots,g_n\) of the ideal \(I\). Define the function \(g=\sum_{i=1}^n g_i^2\). Then the ideal \(\sqrt[\CI]{g}\) coincides with \(\sqrt[\CI]{(g_1,\ldots,g_n)}\) by \cite[Lemma 6]{Borisov_Kremnizer}. Now, consider \(h\in \sqrt[\CI]{(g)}\). Then \(h\) has the same zero set as \(gt\) for some \(t\in \CI(\R^S)\). This implies that \(h\) vanishes on the common zero set of the ideal \(I\). The converse is also easy, let \(h\) vanish on the common zero set of the ideal \(I\) then \(hg\) have the same zero set as \(h\) and thus \(h\in \sqrt[\CI]{(g)}\).
	\end{proof}
	\Cref{st:CI-inf_radical=RJ-radical} can be interpreted using the following auxiliary definition. See \cite[\S 2]{Borisov_Kremnizer} for a more detailed and systematic discussion.
	\begin{definition}\label{def:RJ}
		Let \(A\) be a derived \(\CI\)-algebra. We define the \emph{real Jacobson radical} of \(A\) as the following intersection.
		\[
		\sqrt[\RJ]{0}=\bigcap_{\vp\in \Hom_{\CI\Alg}(A,\R)}\ker\vp.
		\]
		Similarly, one defines the real Jacobson radical of an ideal \(I\) in \(A\) as the intersection of the kernels of all \(\CI\)-algebra morphisms from \(A\) to \(\R\) vanishing on \(I\).

		Consequently, one defines the \(\RJ\)-reduction of \(A\) as the homotopy quotient \(A/\sqrt[\RJ]{0}\).
	\end{definition}
	Then \Cref{st:CI-inf_radical=RJ-radical} can be restated as follows.
	\begin{corollary}\label{cor:CI-red=RJ-red}
		Let \(I\) be a finitely generated ideal in a free ring \(\CI(\R^S)\). Then the \(\CI\)-radical of \(I\) coincides with the real Jacobson radical of \(I\). Consequently, for a derived \(\CI\)-algebra \(A\) of functions on a derived manifold \(\M\) the \(\CI\)-reduction of \(A\) is the same as the real Jacobson reduction of \(A\).
	\end{corollary}
	\begin{remark}
		\Cref{cor:CI-red=RJ-red} shows that since derived manifolds are given by spectra of finitely presented derived \(\CI\)-algebra there is no need to distinguish between the \(\CI\)-reduction and the real Jacobson reduction of a derived \(\CI\)-algebra. This result, which drastically fails in general (see \cite[After Lemma 3]{Borisov_Kremnizer}), will be essential in the proof of the de Rham theorem for the \(\CI\)-de Rham stack (\Cref{CI-infinity-de-Rham=const-sheaf}).
	\end{remark}
	\subsection{Cotangent complex}
	\begin{definition}[{\cite{dubuc19841}}]\label{def:CI-derivations}
		Let \(A\) be a (non-derived) \(\CI\)-algebra. Let \(M\) be a module over \(A\) regarded as a commutative algebra. Then a map \(d:A\to M\) is called a \emph{\(\CI\)-derivation} if it satisfies the Leibniz rule with respect to all \(\CI\)-operations. That is an \(\R\)-linear additive map \(d:A\to M\) satisfying the following property.
		\[
			d(\vp(a_1,\ldots,a_n))=\sum_{i=1}^n \frac{\pd \vp}{\pd x_i}(a_1,\ldots,a_n)d(a_i).
		\]
		Thus, the module of \(\CI\)-K\"{a}hler differentials for a morphism of \(\CI\)-algebras is a module over the target algebra that corepresents the functor of \(\CI\)-derivations relative to the source algebra.
	\end{definition}
	In the derived setting this notion is formalized using the following definition of Steffens, that follows from the general approach of Lurie \cite[\S 7.3]{Lurie_HA}.
	\begin{definition}[{\cite[Definition 5.1.0.2]{Steffens_thesis}}]
		For \(\CI\)-algebra, the tangent category is the presentable fibration 
		\[
		p\colon \Mod \to \CI\Alg.
		\]
		For a derived \(\CI\)-algebra \(A\), the \emph{cotangent complex} 
		\(
		\LL_A\in \Mod_A \simeq T\bigl(\CI\Alg\bigr)\times_{\CI\Alg}\{A\}
		\)
		is defined as the value of the cotangent complex functor at \(A\). 

		For a morphism \(f\colon A\to B\) of derived \(\CI\)-algebras, the relative cotangent complex 
		\(\LL_f\in \Mod_B\) often denoted \(\LL_{B/A}\)	is defined as the value of the relative cotangent complex functor at \(f\).
	\end{definition}
	The following example shows how the cotangent complex can be computed for a derived manifold.
	\begin{example}[{\cite[Example 2.2.19]{Nuiten}}]\label{ex:cotangent_complex_of_dMfld}
		Consider the ring \(\CI(\M)\) of smooth functions on a derived manifold \(\M\) defined by the pullback of \Cref{def:der_mfld} of the map \((f_1,\ldots,f_m)\). Then the ring \(\CI(\M)\) can be presented by the following dg \(\CI\)-algebra.
		\[
		\CI(\M)=\CI(\R^n)[\xi_1,\ldots,\xi_m]
		\] 
		Here \(\xi_i\) are odd variables of degree \(-1\). The differential \(\pd\) is given by \(\pd(\xi_i)=f_i\) for \(i\) from \(1\) to \(m\). The cotangent complex of \(\M\) is then given by a free \(\CI(\M)\)-module of the following form.
		\[
			\Omega^1(\R^n)[\xi_1,\ldots,\xi_m]\oplus\CI(\M)\la d_{\dR}\xi_i\ra_{i=1}^m,\quad \pd d_{\dR}\xi_i=d_{\dR}f_i\in \Omega^1(\R^n).
		\]
	\end{example}
	The following result is standard and follows from the general result of Lurie on cotangent complexes in presentable categories, see \cite[\S 7.3.3]{Lurie_HA}.
	\begin{proposition}[{\cite[Remark 5.1.0.4]{Steffens_thesis}}]\label{prop:functoriality_properties_of_LL}
		The cotangent complex of a morphism of derived \(\CI\)-algebras satisfies the following standard properties. For a morphism \(f:A \to B\) we denote by \(f_*\) the functor \(-\otimes_A B\) between categories of modules.

		\begin{enumerate}
			\item For a commuting triangle of derived $\CI$-algebras
			\[
			\begin{tikzcd}
			& B \arrow[dr,"f"] \\
			A \arrow[rr] \arrow[ur] &  & C
			\end{tikzcd}
			\]
			there is a cofibre sequence in $\mathrm{Mod}_C$:
			
			\[
			\begin{tikzcd}
				f_* \LL_{B/A} \arrow[r] \arrow[d] & \LL_{C/A} \arrow[d] \\
			  0 \arrow[r] & \LL_{C/B}
			\end{tikzcd}
			\]
			
			\item For a pushout square of derived $\CI$-algebras
			\[
			\begin{tikzcd}
			  A \arrow[r] \arrow[d] & B \arrow[d,"f"] \\
			  A' \arrow[r] & B'
			\end{tikzcd}
			\]
			there is an equivalence
			\[
			f_* \LL_{B/A} \xrightarrow{\simeq} \LL_{B'/A'}.
			\]
		\end{enumerate}
	\end{proposition}
	The following result shows that the cotangent complex of an open embedding is trivial.
	\begin{proposition}[{\cite[Corollary 5.1.0.14]{Steffens_thesis}}]\label{cotangent_complex_of_localization_vanishes}
		Let \(A\) be a derived \(\CI\)-algebra. Let \(S\) be a finite collection of elements in \(\pi_0(A)\). Let \(A\to A\{S^{-1}\}\) be a \(\CI\)-localization of a derived \(\CI\)-algebra. Then the cotangent complex of the map \(A\to A\{S^{-1}\}\) is trivial.
	\end{proposition}
	For completeness, we provide a slightly simplified proof of this result following Steffens \cite[\S 5.1]{Steffens_thesis}
	\begin{proof}
		Observe that by \Cref{def:CI-localization} and part (ii) of \Cref{prop:functoriality_properties_of_LL} it is enough to show that the cotangent complex of the morphism \(\CI(\R)\to \CI(\R\setminus \{0\})\) is trivial. This is a direct calculation using the exact triangle for the cotangent complex, i.e. (i) of \Cref{prop:functoriality_properties_of_LL} and \Cref{ex:cotangent_complex_of_dMfld}. Indeed, consider the following triangle of morphisms between (derived) \(\CI\)-algebras.
		\[
			\begin{tikzcd}
			& \CI(\R) \arrow[dr,"f"] \\
			\R \arrow[rr] \arrow[ur] &  & \CI(\R\setminus \{0\})
			\end{tikzcd}
		\]
		Then we observe, that the pushforward \(f_*\LL_{\CI(\R)/\R}\) is isomorphic to the free module of rank \(1\) over \(\CI(\R\setminus \{0\})\). This is clearly the same as the cotangent complex \(\LL_{\CI(\R\setminus \{0\})/\R}\). This together with the exact triangle \(f_* \LL_{\CI(\R)/\R}\to \LL_{\CI(\R\setminus \{0\})/\R}\to \LL_{\CI(\R\setminus \{0\})/\CI(\R)}\) shows that the cotangent complex of the morphism \(\CI(\R)\to \CI(\R\setminus \{0\})\) is trivial.
	\end{proof}
	\begin{definition}\label{def:smooth_morphism}
		A (finitely presented) morphism of derived \(\CI\)-algebras \(A\to B\) is called \emph{smooth} if the cotangent complex of the map is \(0\)-truncated and projective.
		
		Similarly, a smooth and surjective morphism of derived \(\CI\)-algebras is called a \emph{submersion}. 
	\end{definition}
	By \cite[Proposition 5.1.3.32]{Steffens_thesis} we have the following implicit function theorem for derived manifolds.
	\begin{proposition}[{\cite[Proposition 5.1.3.32]{Steffens_thesis}}]\label{implicit_function_theorem}
		If a (finitely presented) morphism of derived \(\CI\)-algebras \(A\to B\) is submersive in the sense of \Cref{def:smooth_morphism} then it is locally (on \(\Spec B\)) equivalent to a map of the form \(A\to A\otimes^\infty\CI(\R^n)\) for some \(n\).  
	\end{proposition}
	
	\subsection{General \texorpdfstring{\(\CI\)}{C-infinity}-algebras and schemes}\label{sec:CI-algebras}
	We now turn to a more general discussion of \(\CI\)-algebras. In particular, we will consider several classes of non-finitely presented algebras naturally arising in the theory of derived manifolds. 
	\begin{definition}\label{def:fg_and_fp_CI}
		Let \(A\) be a derived \(\CI\)-algebra. We say that \(A\) is \emph{finitely generated} if the following conditions hold
		\begin{enumerate}
			\item \(\pi_0A\) is a finitely generated \(\CI\)-algebra, i.e. a quotient of the \(\CI\)-algebra \(\CI(\R^n)\) for some finite \(n\).
			\item \(\pi_n A\) is finitely generated over \(\pi_0A\) for all \(n\).
		\end{enumerate}
		We say that \(A\) \emph{finitely presented} if it is finitely generated and \(\pi_0\) can be presented as a quotient with respect ot a finitely generated ideal in \(\CI(\R^n)\).
	\end{definition}
	\begin{definition}
		The category \(\CI\Alg_{fg}\) of finitely generated derived \(\CI\)-algebras is the full subcategory of \(\CI\Alg\) spanned by finitely generated derived \(\CI\)-algebras. The opposite category \(\CI\Lc=\CI\Alg_{fg}^\op\) is the category of \emph{derived \(\CI\)-loci}.
	\end{definition}
	\begin{remark}
		The category \(\CI\Lc\) should be thought of as a category of \enquote{formal spaces}\footnote{Here formal is used in its plain meaning without any relation to what is usually called formal geometry.}. In particular, as we will see, a derived \(\CI\)-locus does not necessarily arise from some locally ringed topological space. See \Cref{ex:non_fair} for an example of a derived \(\CI\)-algebra which does not arise from any \(\CI\)-ringed space. 
	\end{remark}
	\begin{definition}
		An ideal \(I\triangleleft \CI(\R^n)\) is {\it germ determined} if for any \(f\in \CI(\R^n)\setminus I\) there exists \(p\in \R^n\) such that in the localization we have \(f_p\notin I_p.\) A finitely generated classical \(\CI\)-algebra is \emph{germ-determined} if it is a quotient by a germ-determined ideal in \(\CI(\R^n)\).
	\end{definition}
	\begin{definition}
		An ideal \(I\triangleleft \CI(\R^n)\) is {\it point determined} if it has the form of \(\m_Z\) for some closed subset \(Z\) of \(\R^n\).
	\end{definition}
	\begin{example}
		Let \(Z\) be a closed subset of \(\R^n\). Then the ideal \(\m_Z\) of functions vanishing on \(Z\) is germ-determined. Similarly, the ideal \(\m_Z^\infty=\cap_{k=1}^\infty \m_Z^k\) is also germ-determined. This ideal consists of functions vanishing to infinite order on \(Z\) at each point of \(Z\). See \cite[\S I.4]{Moerdijk_Reyes_book} for a detailed discussion of this. Even more details on the structure of ideals of smooth functions can be found in monographs \cite{Malgrange} and \cite{Tougeron}.		
	\end{example}
	\begin{construction}\label{con:spectrum_of_a_CI_algebra}
		Let \(A\) be a finitely generated derived \(\CI\)-algebra. We define the \emph{spectrum} of \(A\) to be the following topological space with a sheaf of derived \(\CI\)-algebras.
		\begin{enumerate}
			\item As a set \(\Spec A\) is the set of all homomorphisms \(A\to \R\). The topology on \(\Spec A\) is induced from the classical topology on \(\R\) via the Zariski topology construction and has the following basis.
			\[
				U_f=\{p\in \Spec A\mid p(f)\neq 0\},\quad f\in \pi_0(A).
			\]
			\item The sheaf of derived \(\CI\)-algebras \(\CI_{\Spec A}\) is obtained from \(A\) via \(\CI\)-localizations. That is, for any basis open set \(U_f\subset \Spec A\) we have \(\O_{\Spec A}(U_f)=A\{f^{-1}\}\). 
		\end{enumerate}
	\end{construction}
	The following definition characterizes \(\CI\)-algebras with nice categorical properties. Note that this condition is vacuous in the case of commutative algebras, i.e., in derived algebraic geometry.
	\begin{definition}[{\cite[Definition 2.16]{Joyce_CInf_Geom}}]
		A \(\CI\)-algebra is \emph{fair} if it is equivalent to the global sections of the structure sheaf on its spectrum. A derived \(\CI\)-locus dual to a fair \(\CI\)-algebra is called an \emph{affine derived \(\CI\)-scheme}.
	\end{definition}
	\begin{example}\label{ex:non_fair}
		The following example is standard, and we include it for completeness. Let \(\CI_c(\R^n)\) be the \(\CI\)-algebra of compactly supported smooth functions on \(\R^n\). Then the spectrum of \(\CI_c(\R^n)\) is locally isomorphic to \(\R^n\) and hence the algebra of functions on it is \(\CI(\R^n)\).
	\end{example}
	\section{Completions of derived differentiable stacks}\label{completions_of_derived_CI_stacks}
	In this section, we consider various notions of completions for derived \(\CI\)-stacks and how they compare to each other.
	\subsection{Completions of derived \texorpdfstring{\(\CI\)}{C-infinity}-stacks}
	As we have seen in \Cref{sec:reductions}, there are two natural reductions of a derived \(\CI\)-algebra. Now we define the completions of derived \(\CI\)-stacks in two natural ways based on these reductions. This definition follows the general approach of \cite[\S 6]{Gaitsgory_Rozenblyum_Ind} to de Rham stacks in algebraic geometry. A version of the \(\Com\)-completion for \(\CI\)-derived stacks was previously given in \cite[\S 7.1]{Nuiten}. 
	\begin{definition}\label{def:formal_completion}
		If \(\X\to \Y\) is a morphism of derived \(\CI\)-stacks, then the \emph{derived \(\CI\)-completion} of \(\X\) over \(\Y\) is the derived \(\CI\)-stack \(\X^{\wedge_{\CI}}_{\Y}\) which is defined as follows on a derived \(\CI\)-algebra \(R\).
		\[\X^{\wedge_{\CI}}_{\Y}(R)=\left\{\begin{tikzcd}
			{\Spec R^{\red_{\CI}}} && \Y \\
			\\
			{\Spec R} && \X
			\arrow[from=1-1, to=1-3]
			\arrow[from=1-3, to=3-3]
			\arrow[from=1-1, to=3-1]
			\arrow[from=3-1, to=3-3]
		\end{tikzcd}\right\}\]
		More formally, the value \(\X^{\wedge_{\CI}}_{\Y}(R)\) is the homotopy fibre product of the following form.
		\[
			\X^{\wedge_{\CI}}_{\Y}(R)=\Y(R^{\red_{\CI}})\times_{\X(R^{\red_{\CI}})} \X(R)
		\]
		Similarly, we define the \emph{derived \(\Com\)-completion} of \(\X\) over \(\Y\). It is the derived \(\CI\)-stack \(\X^{\wedge_{\Com}}_{\Y}\) which is defined as follows on a derived \(\CI\)-algebra \(R\).
		\[\X^{\wedge_{\Com}}_{\Y}=\left\{\begin{tikzcd}
			{\Spec R^{\red_{\Com}}} && \Y \\
			\\
			{\Spec R} && \X
			\arrow[from=1-1, to=1-3]
			\arrow[from=1-3, to=3-3]
			\arrow[from=1-1, to=3-1]
			\arrow[from=3-1, to=3-3]
		\end{tikzcd}\right\}\]
		More formally, the value \(\X^{\wedge_{\Com}}_{\Y}(R)\) is the homotopy fibre product of the following form.
		\begin{equation}\label{eq:Com_completion}
			\X^{\wedge_{\Com}}_{\Y}(R)=\Y(R^{\red_{\Com}})\times_{\X(R^{\red_{\Com}})} \X(R)
		\end{equation}
	\end{definition}
	While the above definition is great for formal purposes, it is not sufficiently hands-on. To deal with this, we will relate the abstract completion defined above first with the sequential or Greenlees--May completion in \Cref{sec:sequential_and_GM} and then to the Adams completion of a morphism of derived \(\CI\)-algebras in \Cref{sec:Adams_completion}.
	
	\subsection{Sequential completion and Greenlees-May completion}\label{sec:sequential_and_GM} In this section, we review two algebraic constructions of derived completions for derived \(\CI\)-algebras. Then we show that they are equivalent and yield the derived \(\CI\)-completion of affine derived \(\CI\)-schemes in the sense of \Cref{def:formal_completion}. 

	The first construction we consider is a \emph{sequential completion}. See, for example, \cite{Yekuteli_WPR} for a discussion of various types of completions for classical non-Noetherian rings, as well as \cite[\S 4.1]{Lurie_DAGXII} and \cite[\S 6.7.3]{Gaitsgory_Rozenblyum_Ind} for a discussion in the derived setting.

	The idea is to construct a sequence of derived \(\CI\)-algebras that approximate the completion of a given morphism of derived \(\CI\)-algebras.
	\begin{construction}\label{con:sequence_presenting_completion}
		Let \(\vp:A\to B\) be a morphism of derived \(\CI\)-algebras corresponding to a closed immersion of derived \(\CI\)-schemes \(Y\to X\). Assume that the kernel \(I\) of \(\vp\) is generated by a finite set \(S=\{f_1,\ldots,f_m\}\). Then we define the following sequence of derived connective \(\CI\)-algebras.

		\[
			\K_{n,\Com}(A,S)=A[t_{i,n}],\quad \deg t_{i,n}=-1,\quad \pd t_{i,n}=f_i^n.
		\]
		
		Algebras \(A_n\) assemble into a pro-object in the category of derived \(\CI\)-algebras using the following maps.
		\[
			\K_{N,\Com}(A,S)\to \K_{n,\Com}(A,S),\quad t_{i,N}\mapsto f_i^{N-n}\cd t_{i,n},\quad N\ge n.
		\]
		This pro-object is the \emph{sequential completion} of \(A\) in \(S\).
	\end{construction}
	The following result relating the sequential completion of \(\CI\)-algebras to the derived completion of \(\CI\)-stacks (see \Cref{def:formal_completion}) is due to Nuiten. A similar construction in the setting of Noetherian algebraic geometry is given in \cite[\S 6.7.3]{Gaitsgory_Rozenblyum_Ind}.
	\begin{proposition}[{\cite[Proposition 7.1.10]{Nuiten}}]\label{prop:sequence_presents_completion}
		Let \(\vp:A\to B\) be a morphism of derived \(\CI\)-algebras corresponding to a closed immersion of derived \(\CI\)-schemes \(Y\to X\). Assume that the kernel \(I\) of \(\vp\) is generated by a finite set \(S=\{f_1,\ldots,f_m\}\). Then, the derived completion of stacks of \(X\) in \(Y\) is equivalent to the colimit of spectra of the algebras \(\K_{n,\Com}(A,S)\) in the category of prestacks. 
	\end{proposition}

	The second construction we introduce is due to Greenlees and May \cite{Greenlees_May}. We give two versions of it for derived \(\CI\)-algebras. These two versions again arise from two notions of localization for derived \(\CI\)-algebras. The \(\CI\)-version will not be used in the sequel, but we include it here for the sake of conceptual symmetry.
	\begin{construction}\label{con:GM}
		Let \(S\) be a finite set of elements in a derived \(\CI\)-algebra \(A\). One associates to it two complexes.
		\begin{itemize}
			\item[\(\CI\)] Define the \emph{\(\CI\)-infinite dual Koszul complex} \(\K^\vee_{\infty,\CI}(A,S)\) as the following two term complex.
			\[
				A\to A\{S^{-1}\}
			\]
			Here \(A\{S^{-1}\}\) is the \(\CI\)-localization of \(A\) in \(S\).
			\item[\(\Com\)] Define the \emph{\(\Com\)-infinite dual Koszul complex} \(\K^\vee_{\infty,\Com}(A,S)\) as the following two term complex.
			\[
				A\to A[S^{-1}]
			\]
			Here \(A[S^{-1}]\) is the \(\Com\)-localization of \(A\) in \(S\).
		\end{itemize}
		Then we have two versions of the Greenlees-May completion of \(A\) in the ideal \(I\) generated by \(S\). 
		\begin{itemize}
			\item[\(\CI\)] Define the \emph{\(\CI\)-Greenlees-May completion} of \(A\) at \(I\) as the following derived \(\CI\)-algebra.
			\[
				A^{\wedge_{\CI,\GM}}_I=\Hom_A(\K^\vee_{\infty,\CI}(A,S),A)
			\]
			\item[\(\Com\)] Define the \emph{\(\Com\)-Greenlees-May completion} of \(A\) at \(I\) as the following derived \(\CI\)-algebra.
			\[
				A^{\wedge_{\Com,\GM}}_I=\Hom_A(\K^\vee_{\infty,\Com}(A,S),A)
			\]
		\end{itemize}
	\end{construction}
	\begin{remark}
		In the \(\Com\) case, the infinite dual Koszul complex is a colimit of \emph{finite Koszul complexes}, i.e. Koszul complexes for finite powers of elements in \(S\). This description is more subtle in the \(\CI\)-case. One should think of \enquote{finite Koszul complexes} as Koszul complexes for sets \(S_f\) for all \(\CI\)-functions \(f:\R^n\to \R\) vanishing at \(0\). However, we have not explored this further here.
	\end{remark}

	The following result is well-known, see e.g. \cite[\S 4]{Greenlees_May}.
	\begin{proposition}
		The \(\Com\)-infinite dual Koszul complex can be presented by the following colimit.
		\[
			\K^\vee_{\infty,\Com}(A,S)=\colim_{n\in \N}\K^\vee_{n,\Com}(A,S).
		\]
	\end{proposition}
	\begin{corollary}\label{GM_is_the_same_as_Sequential}
		The \(\Com\)-Greenlees-May completion of \(A\) at a finitely generated ideal \(I\) is equivalent to sequential completion of \(A\) at \(I\). 
	\end{corollary}
	\begin{notation}\label{there_is_only_one_completion}
		\Cref{GM_is_the_same_as_Sequential} allows us to identify the Greenlees-May completion of a derived \(\CI\)-algebra with the sequential completion and by \Cref{prop:sequence_presents_completion} with the derived completion of stacks. As a result, we will refer to all of them as the \(\Com\)-\emph{derived completion}. 
	\end{notation}
	\subsection{Adams completion of a \texorpdfstring{\(\CI\)}{C-infinity}-algebra morphism}\label{sec:Adams_completion}
	In this section, we review the Adams completion of Carlsson \cite{Carlsson}. We then show that for finitely presented morphisms of derived affine \(\CI\)-schemes, the Adams completion coincides with the \(\Com\)-derived completion. 
	\begin{construction}[\v{C}ech conerve of a \(\CI\)-algebra map]\label{def:Cech_conerve}
		Let \(f:A\to B\) be a morphism of \(\CI\)-algebras. We define the \v{C}ech conerve of \(f\) as the following cosimplicial derived \(\CI\)-algebra.
		\[
			\check{C}^n(f)=B^{\otimes_A^\infty n}.
		\]
		The cosimplicial structure is defined by the canonical maps of the cofibered coproduct. Equiavlently, one can think of \(\check{C}^\bt(f)\) as the cosimplicial algebra of \(\CI\) functions on the simplicial derived affine \(\CI\)-scheme given by the \v{C}ech nerve of the morphism \(\Spec(f):\Spec B\to \Spec A\).
	\end{construction}
	\begin{remark}
		It might seem that it should be possible to define a ``\(\Com\)''-version of the \v{C}ech conerve. However, this is not the case. The reason is that the \v{C}ech conerve should be a cosimplicial object in the category of derived \(\CI\)-algebras. However, the commutative tensor product of derived \(\CI\)-algebras is not a derived \(\CI\)-algebra.
	\end{remark}
	The following result of Spivak \cite[Lemma 8.1]{spivak2010derived}, see also in Carchedi--Steffens \cite[Lemma 4.6]{Carchedi_Steffens} and \cite[Corollary 4.1.3.6]{Steffens_thesis}, and generally due to Lurie \cite[Lemma 11.10]{Lurie_DAG9} is essential. It shows that for \(\pi_0\)-epimorphisms of \(\CI\)-algebras, the derived \(\CI\)-pushout of a morphism of \(\CI\)-algebras coincides with the derived pushout of commutative algebras. Equivalently, given a closed immersion of derived affine \(\CI\)-schemes, derived \(\CI\)-intersection can be calculated as the derived intersection of the underlying commutative schemes.
	\begin{proposition}\label{unframifiedness}
		If \(f:A\to B\) is a morphism of \(\CI\)-algebras, such that \(\pi_0(f)\) is an epimorphism of rings, then the derived \(\CI\)-pushout of \(f\) along any map coincides with the derived pushout of commutative algebras.
	\end{proposition}
	The next corollary is immediate.
	\begin{corollary}\label{cor_of_unframifiedness}
		If \(f:A\to B\) is a morphism of \(\CI\)-algebras, such that \(\pi_0(f)\) is an epimorphism of rings, then the \v{C}ech conerve of \(f\) in the sense of \Cref{def:Cech_conerve} coincides with the \v{C}ech conerve of \(f\) as a morphism of derived \(\R\)-algebras.
	\end{corollary}
	\begin{definition}
		Given a morphism of \(\CI\)-algebras \(f:A\to B\), we define the \emph{Adams completion} of \(B\) in \(A\) as the totalization of the \(\CI\)-\v{C}ech conerve. This is also known as the \emph{completed Amitsur complex}.
		\[
			B^\wedge_A=\Tot\left(\check{C}^\bt f\right).
		\]
	\end{definition}
	Now we connect Adams completion with the derived \(\Com\)-completion of derived \(\CI\)-schemes. This is done essentially in the same way as in \cite[\S 5]{Carlsson}. The only difference is that we deal with \(\CI\)-algebras instead of plain commutative algebras. However, this difference is irrelevant since all constructions of \(\Com\)-completion and Adams completion for \(\pi_0\)-epimorphism proceed by essentially forgetting the \(\CI\)-structure.
	\begin{proposition}[{\cite[Proposition 5.4]{Carlsson}}]\label{prop:Adams_vs_Derived}
		Adams completion of a \(\pi_0\)-epimorphism of \(\CI\)-algebras \(f:A\to B\) satisfies the universal property of \(\Com\)-completion given in \Cref{def:formal_completion}. That is, for any derived \(\CI\)-algebra \(C\), we have the following equivalence.
		\[
			\Map_{\CI\Alg}(B^\wedge_A,C)\simeq \Hom_{\CI\Alg}(B,C^\redcom)\times_{\Map_{\CI\Alg}(A,C^\redcom)}\Hom_{\CI\Alg}(A,C).
		\]
	\end{proposition}
	\begin{proof}
		By \cite[Proposition 5.4]{Carlsson}, the Adams completion of a derived epimorphism of commutative algebras is equivalent to the Greenlees-May completion. However, note that by \Cref{cor_of_unframifiedness} the Adams completion of a derived epimorphism of \(\CI\)-algebras is equivalent to the Adams completion of the underlying commutative algebras. The result now follows from the comparison between Greenlees--May completion and derived completion of stacks given by \Cref{there_is_only_one_completion}.
	\end{proof}
	\subsection{Classical vs Derived formal completion} In this section, we review the notion of the classical formal completion and how it compares to its derived counterpart.

	We have two main results in this direction. The first one deals with principal ideals in free \(\CI\)-algebras. That is, given a principal ideal in a free \(\CI\)-algebra, the derived formal completion in it is classical. The argument is quite simple and relies on the fact that we can easily describe the ideal of functions annihilating the generator of the given principal ideal.
	
	The second result deals with the case of ideals generated by analytic germs in a local free \(\CI\)-algebra. In this case, the derived formal completion is also classical; here, the idea is to relate the local ring with the Noetherian ring of analytic germs to obtain the result.
	\subsubsection{The case of a single equation}
	In this section, we prove a weaker version of the result analogous to \cite[Proposition 6.8.2]{Gaitsgory_Rozenblyum_Ind} of Gaitsgory--Rozenblyum. We show that given a closed immersion of classical affine \(\CI\)-schemes defined by a \emph{single equation}, the derived completion of the derived \(\CI\)-scheme determined by this equation coincides with the classical formal completion of the underlying classical scheme.

	The following result is proved essentially in the same way as \cite[Proposition 6.8.2]{Gaitsgory_Rozenblyum_Ind}. 
	\begin{proposition}\label{prop:completion_of_classical_is_classical}
		Let \(Y\ss \R^n\) be a closed immersion of classical \(\CI\)-schemes, such that either of the two conditions holds.
		\begin{enumerate}
			\item \(Y\) is a zero locus of single equation \(f\).
			\item \(Y\) is a zero locus of a finite ordered set of functions \(S=\{f_1,\ldots,f_m\}\) forming a regular sequence in \(\CI(\R^n)\).
		\end{enumerate}
		Then, the derived formal completion of the derived scheme \(\R^n\) in \(Y\) coincides with the classical formal completion of the classical \(\CI\)-scheme \(\R^n\) in \(Y\). 
	\end{proposition}
	\begin{remark}
		Note that the two conditions in \Cref{prop:completion_of_classical_is_classical} are independent since there are many zero-divisors in the ring of smooth functions on \(\R^n\). More precisely, a function \(f\in \CI(\R^n)\) is regular if and only if \(Supp(f)=\R^n\) (see \Cref{definition_of_support}). Indeed, if \(\Supp(f)\neq \R^n\), one can find a bump function supported on the complement which multiplies to zero with \(f\). 
	\end{remark}
	\Cref{prop:completion_of_classical_is_classical} admits the following concrete reformulation in terms of the Koszul algebras \(\K_{n,\Com}(A,\{f\})\) from \Cref{con:sequence_presenting_completion} which we denote for brevity by \(A_n\). Denote for \(A_n\) its zeroth truncation by \(A_n'\). Then we claim that sequences \(A_n\) and \(A_n'\) are equivalent as pro-objects. That is, we have the following equivalence of derived prestacks.
		\[
			\colim_n\Spec A_n'\xrightarrow{\sim} \colim_n\Spec A_n.
		\]
	\Cref{prop:completion_of_classical_is_classical} directly follows from the following result.
	\begin{lemma}\label{lemma:completion_factorization}
		Let \(Y\ss \R^n\) be a closed immersion of classical \(\CI\)-schemes as in the formulation of \Cref{prop:completion_of_classical_is_classical}. Denote for \(A_n\) its zeroth truncation by \(A_n'\). Then, we claim that for each \(n\ge 0\), a factorization of the following form exists for some \(N_0\ge n\) and consequently for all \(N\ge N_0\).
		\[
			A_N\to A_N'\to A_n.
		\]	
	\end{lemma}
	The second case of \Cref{prop:completion_of_classical_is_classical} is immediate since for a regular sequence (or its powers) the Koszul complexes are acyclic. For the first case, the proof of
	\Cref{lemma:completion_factorization} follows from the following preliminary homological version.
	\begin{lemma}\label{lemma:tedious_completions}
		Let \(A\to B\) be a morphism of derived \(\CI\)-algebras. Assume that the kernel of \(A\to B\) is generated by a single element \(f\). Then, pro-systems \(H^{-i}(A_n)\) are equivalent to \(0\) for each \(i>0\). That is for every \(n\ge 0\) there exists \(N_0\) such that for all \(N\ge N_0\) the following map is zero.
		\[
			H^{-i}(A_N)\to H^{-i}(A_n).
		\]
	\end{lemma}
	\begin{proof}
		To see this, we will show that for each \(n\ge 0\) and \(i>0\) there exists \(N\) such that the map \mbox{\(\tau^{\ge-i}A_N\to\tau^{\ge-i} A_n\)} factorizes as follows.
		\[
			\tau^{\ge-i}A_N\to \tau^{\ge-i+1}A_N\to \tau^{\ge-i}A_n.
		\]
		The above would yield \Cref{lemma:tedious_completions} since the Koszul complex construction of \(A_n\)'s implies that they are eventually coconnective. We observe the following. By the construction of \(A_n\), the only non-trivial cohomology module is concentrated in degree \(-1\). There, it is given by the following vector space of smooth functions on \(\R^n\).
		\[
			A^{(n)}=\ker\left(f^n\cd-:A\to A\right).
		\]
		Clearly, this space consists of functions that vanish on the open support set of \(f\) (see \Cref{definition_of_support}). Consequently, \(A^{(n)}=A^{(1)}\) for all \(n>0\). This, in turn, implies that for \(N\ge n+1\), we have the requisite factorization.
	\end{proof}
	\begin{proof}[Proof of \Cref{lemma:completion_factorization}]
		The argument proceeds verbatim as \cite[Proof of Lemma 6.8.6]{Gaitsgory_Rozenblyum_Ind}.
	\end{proof}
	\subsubsection{Completions in analytically generated ideals}
	We consider another case where it is relatively easy to establish a version of \cite[Proposition 6.8.2]{Gaitsgory_Rozenblyum_Ind}. Namely, we consider the case of an ideal generated by a finite set of functions that are \emph{real analytic}. To establish the result, we will use a result of Malgrange. Note that since the question is local in nature, we freely replace the ring of functions on \(\R^n\) with the ring of germs of functions at a point. 
	\begin{theorem}[{\cite[Corollary 1.12]{Malgrange}, \cite[Chapitre VI, Corollaire 1.3]{Tougeron}}]\label{thm:Tougeron's_flatness}
		The module of germs of \(\CI\)-functions over the ring of germs of real analytic functions is flat.
	\end{theorem}
	\begin{proposition}\label{prop:hypo_Noetherian_completion}
		Let \(Y\ss \R^n\) be a closed immersion of classical \(\CI\) schemes, such that \(Y\) is determined in a local ring of a point by an ideal generated by analytic functions. Then, the derived completion of the derived scheme \(\R^n\) in \(Y\) coincides with the classical formal completion of the classical \(\CI\)-scheme \(\R^n\) in \(Y\). 
	\end{proposition}
	\begin{proof}
		The proof follows from \Cref{thm:Tougeron's_flatness}. Using the fact that the ring of germs of analytic functions is Noetherian and \cite[Proposition 6.8.2]{Gaitsgory_Rozenblyum_Ind}, we see that the statement is true in the analytic case. Now it remains to use the flatness of the module of germs of \(\CI\)-functions over the ring of germs of real analytic functions to deduce the result in the \(\CI\)-case.
	\end{proof}
	\subsection{Some questions about completions of \texorpdfstring{\(\CI\)}{C-infinity}-algebras}\label{questions_about_completions}
	\begin{question}
		Is there a statement analogous to \Cref{prop:completion_of_classical_is_classical} for general closed immersions of classical \(\CI\)-schemes?
	\end{question}
	As we have seen in the proof of \Cref{lemma:tedious_completions}, our arguments rely essentially on the fact that we are considering a single function on the affine space. In particular, we could characterize the homology groups of the Koszul complex for this function in terms of the geometry of the zero locus of \(f\). It would be interesting to find a similar characterization for several functions; however, the situation is already quite complicated even in the case of two functions. 

	In general, questions of this nature for abstract commutative rings are quite hard. For example, even the case of a single element is non-trivial, see \cite[\S 3]{Yekuteli_WPR} and references there for a discussion of these properties. Essentially, for a single element \(f\) in a ring, one needs the boundedness of the torsion associated with this element. That is, the ascending chain of ideals annihilating \(f^n\)'s stabilizes. Below, we give an example where this fails for a non-free \(\CI\)-algebra. It seems that it could be possible to use the formalism of liquid modules in the sense of Clausen--Scholze (see \cite{Clausen_Scholze}) to study this question.
	\begin{example}
		Consider the function \(h=e^{-\frac{1}{e^{-\frac{1}{x^2}}}}\) on \(\R\). This function is smooth and vanishes to infinite order at \(0\). Consider the ideal \(I\) generated by \(h\) in \(\CI(\R)\). Now consider the element in \(\CI(\R)/h\) represented by the function \(g=e^{-\frac{1}{x^2}}\). This function is not divisible by \(h\) in \(\CI(\R)\). However, it is easy to see that \(h\) is divisible by \(g^n\) for all \(n>0\). Consequently, the sequence of ideals annihilating the function \(g^n\) coincides with the sequence of principal ideals \(\left(\frac{h}{g^n}\right)\) and thus does not stabilize.

		Of course, it is completely inessential that we consider this precise pair of functions, the only feature we used here is that \(h\) decays to zero at \(0\) faster than any polynomial of \(g.\)
	\end{example}
	\begin{question}
		Is there a version of the sequential completion of \Cref{con:sequence_presenting_completion} that yields the \(\CI\)-completion rather than the \(\Com\)-completion?
	\end{question}
	Intuitively, it seems that this construction should proceed over the poset of higher-order infinitesimals rather than just polynomials. 
	\begin{question}
		Is there a version of \Cref{GM_is_the_same_as_Sequential} identifying \(\CI\)-Greenlees-May completion with the \(\CI\)-derived completion of stacks?
	\end{question}
	This seems like the most straightforward aspect of the theory one could generalize, and we hope to address it in future work.
	\section{Hodge completed derived de Rham cohomology}\label{derived_de_Rham}
	\subsection{Construction}
	In this section, we will define the derived de Rham cohomology of a derived affine \(\CI\)-scheme following the approach of Bhatt~\cite{Bhatt_Derived}. 
	\begin{construction}\label{con:derived_de_Rham}
		Let \(A\to B\) be a morphism of derived \(\CI\)-algebras. Let \(F\to B\) be a free \(\CI\)-resolution of \(B\) over \(A\). We define the \emph{Hodge completed derived de Rham cohomology} of \(B\) over \(A\) as the completion of \(\left|\Omega^\bt_{F/A}\right|\) for its Hodge filtration. Concretely, the complete filtered derived commutative algebra \(\cdR_{B/A}\) is the limit of the following directed system.
		\[
			\cdR_{B/A}/\Fil_H^k=\Tot(\sigma^{\leq k}\Omega^\bt_{F/A}).
		\]
		Here \(\sigma^{\leq k}\) denotes the truncation in cohomological degrees \(\leq k\). 
	\end{construction}
	
	\begin{proposition}
		Let \(A\to B\) be a morphism of derived \(\CI\)-algebras. Consider the complete filtered derived commutative algebra \(\cdR_{B/A}\) as in \Cref{con:derived_de_Rham}.
		The associated graded algebra of the induced filtration on~\(\cdR_{B/A}\) is isomorphic to the exterior algebra on the cotangent complex of the map \(A\to B\). Namely, we have the following isomorphism.
		\begin{equation}\label{eq:Cartier_isomorphism}
			\gr^k(\cdR_{B/A})\cong\Lambda^k\LL_{B/A}[-k],\quad k\ge 0.
		\end{equation}
	\end{proposition}
	\begin{proof}
		This result follows from the fact that the associated graded algebras are preserved by completion. Consequently, it is enough to prove isomorphism \eqref{eq:Cartier_isomorphism} for the non-completed case. Concretely, we need to see that there is the following isomorphism.
		\[
			\Tot(\sigma^{\leq k}\Omega^\bt_{F/A})/\Tot(\sigma^{\leq k-1}\Omega^\bt_{F/A})\cong\Lambda^k\LL_{B/A}.
		\]
		This, however, follows from the fact that totalization of a truncated cosimplicial object commutes with taking quotients. Now it remains to be observed that by definition, we have the following equivalence:
		\[
			\Lambda^k\LL_{B/A}\cong\Tot(\Omega^k_{F/A}).
		\]
		The compatibility of the multiplicative structure follows from the properties of the relative differential forms. 
	\end{proof}
	\subsection{Quillen's convergence and comparison with Adams completion}
	In this section, we describe the relationship between Adams completion and derived de Rham cohomology following Bhatt \mbox{\cite[\S 4.2, 4.3]{Bhatt_Derived}}. We will show that the Adams completion of a closed immersion is equivalent to the derived de Rham cohomology. This result will be crucial for identifying Hodge completed derived de Rham cohomology with functions on the \(\Com\)-de Rham stack in \Cref{thm:de_Rham_vs_functions_on_inf}.
	\begin{proposition}
		Let \(A\) be a derived \(\CI\)-algebra. Let \(A\to B\) be a \(\pi_0\)-isomorphism. Then, the completion of \(A\) at \(B\) is equivalent to \(A\).
	\end{proposition}
	\begin{proof}
		The proof is a verbatim repetition of the proof given by Bhatt \cite[Proof of Proposition 4.11]{Bhatt_Derived}.
	\end{proof}
	\begin{corollary}\label{cor:de_Rham_equiv_for_quotients}
		Let \(A\) be a derived \(\CI\)-algebra. Let \(A\to B\) be a \(\pi_0\)-isomorphism. Then, the following filtered equivalence holds for the relative de Rham cohomology of the quotient morphism \(A\to B\).
		\[
			\Big\{\cdR_{B/A},\Fil_H^\bt\Big\}\simeq \Big\{A,\Fil_H^\bt\Big\}
		\]	
		Consequently, \(A\simeq \cdR_{B/A}\).
	\end{corollary}
	We have the following version of Bhatt's theorem \cite[Proposition 4.16]{Bhatt_Derived} that connects the Adams completion of a morphism of derived \(\CI\)-algebras (see \Cref{sec:Adams_completion}) with the derived de Rham complex.
	\begin{theorem}\label{thm:Adams_completion_vs_de_Rham}
		Let \(f:A\to B\) be a \(\pi_0\)-epimorphic morphism of derived \(\CI\)-algebras. Then, we have a filtered equivalence of complete filtered derived commutative algebras.
		\[
			\left\{\cdR_{B/A},\Fil_H^\bt\right\}\simeq \Big\{\Comp_A(A,f),\Fil_H^\bt\Big\}
		\]
		Consequently, the Adams completion of \(f\) is equivalent to the derived de Rham cohomology of \(B\) over \(A\).
		\[
			\cdR_{B/A}\simeq \Comp_A(A,f).
		\]
	\end{theorem}
	The proof of this result is essentially the same as the proof of \cite[Proposition 4.16]{Bhatt_Derived}. We provide a detailed argument here both for the sake of completeness and to show that the case of \(\CI\)-algebras is essentially the same as the case of commutative algebras. 
	\begin{proof}
		Since \(f\) is a \(\pi_0\)-epimorphism, we have the following isomorphism.
		\[
			\pi_0(B\otimes_A^\infty \ldots \otimes_A^\infty B)\cong \pi_0(B).
		\]
		Thus, by \Cref{cor:de_Rham_equiv_for_quotients} there is a complete \(\N^\op\)-indexed filtration on \(B^{\otimes_A^\infty m}\) together with a filtered equivalence as follows.
		\[
			\Big\{B^{\otimes_A^\infty m}, \Fil_H^\bt\Big\}\simeq \Big\{\cdR_{B^{\otimes_A^\infty m}/B},\Fil_H^\bt\Big\}
		\]
		Consequently, by functoriality, there is a filtered equivalence as follows.
		\[
			\Big\{\check{C}(f),\Fil_H^\bt\Big\}\simeq \Big\{\cdR_{B/\check{C}(f)},\Fil_H^\bt\Big\}
		\]
		By passing to limits, we obtain a filtered equivalence as follows.
		\[
			\Big\{\Comp_A(A,f),\Fil_H^\bt\Big\}\simeq \Big\{\Tot\cdR_{B/\check{C}f},\Fil_H^\bt\Big\}
		\]
		Finally, we must compare the right-hand side with \(\cdR_{B/A}\). The result now follows from the following equivalence of the associated graded algebras.
		\[
			\Lambda^k\LL_{B/A}[-k]\cong\gr^k(\cdR_{B/A})\cong\gr^k\left(\Tot\cdR_{B/\check{C}f}\right)\cong\Tot\Lambda^k\LL_{B/\check{C}f}[-k].
		\]
	\end{proof}
	
	\subsection{Properties of derived de Rham cohomology}
	\begin{proposition}\label{prop:de_Rham_of_localization}
		Let \(A\to B\) be a \(\CI\)-localization of derived \(\CI\)-algebras. Then, the derived de Rham cohomology of \(B\) over \(A\) is trivial.
	\end{proposition}
	\begin{proof}
		The result follows from isomorphism \eqref{eq:Cartier_isomorphism} and the fact that by \Cref{cotangent_complex_of_localization_vanishes} the cotangent complex of a localization morphism is trivial.
	\end{proof}
	The following result is due to Bhatt and can be translated verbatim to the \(\CI\)-setting.
	\begin{lemma}[{\cite[Example 4.17]{Bhatt_Derived}}]\label{section_means_trivial_de_Rham}
		Let \(A\to B\) be a morphism of derived \(\CI\)-algebras. Assume that it admits a section. Then \(\cdR_{A/B}\) is trivial.
	\end{lemma}
	\begin{example}\label{ex:de_Rham_coh_of_Weyl}
		A \emph{Weil algebra} \(W\) is a local \(\CI\)-algebra which is finite-dimensional as an \(\R\)-vector space, and such that \(W\) can be presented as a sum of the following form.
		\[
			W=\R\oplus \m.
		\]
		Here, \(\m\) is the maximal ideal. Note that the spectrum of a Weil algebra has a single closed point. By \cite[Theorem 3.17]{Moerdijk_Reyes_book} any Weil algebra is canonically a \(\CI\)-algebra. 
		
		We have a canonical section of the unit homomorphism \(\R\to W\) given by taking the quotient with respect to \(\m\). As a result, by \Cref{section_means_trivial_de_Rham} de Rham cohomology of \(W\) vanishes.
	\end{example}
	\begin{proposition}\label{prop:de_Rham_depends_on_pi0}
		Let \(f:A\to B\) be a \(\pi_0\)-equivalence of derived \(\CI\)-algebras. Then, the derived de Rham cohomology of \(B\) over \(A\) is trivial, and there is an equivalence as follows.
		\[
			\cdR_{A}\simeq \cdR_B.
		\]
	\end{proposition}
	\begin{proof}
		The result is immediate from \Cref{thm:de_Rham_vs_functions_on_inf}.
	\end{proof}
	\begin{definition}
		Let \(f,g:X\to Y\) be a morphism of derived \(\CI\)-stacks. We say that \(f\) is smoothly homotopic to \(g\) if there exists a morphism of derived \(\CI\)-stacks \(H:\R\times X\to Y\) such that the following diagram commutes. 
\[\begin{tikzcd}
	X \\
	\\
	{\R\times X} && Y \\
	\\
	X
	\arrow["{\{a\}\times\id_X}"', hook, from=1-1, to=3-1]
	\arrow["f", from=1-1, to=3-3]
	\arrow["H", from=3-1, to=3-3]
	\arrow["{\{b\}\times\id_X}", hook, from=5-1, to=3-1]
	\arrow["g"', from=5-1, to=3-3]
\end{tikzcd}\]
		For some real numbers \(a<b.\)
	\end{definition}
	\begin{proposition}[Homotopy invariance]
		Let \(X\) be a derived manifold. Assume that there exists a smooth homotopy contracting \(X\). Then, the derived de Rham cohomology of \(X\) is trivial. 
	\end{proposition}
	\begin{proof}
		First, we consider the case of affine \(X\) and \(Y\). We have the following morphisms at the level of rings.

\[\begin{tikzcd}
	{\CI(X)} \\
	\\
	{\CI(X)\{t\}} && {\CI(Y)} \\
	\\
	{\CI(X)}
	\arrow["{t\mapsto a}", from=3-1, to=1-1]
	\arrow["{t\mapsto b}"', from=3-1, to=5-1]
	\arrow["{f^*}"', from=3-3, to=1-1]
	\arrow["{H^*}"', from=3-3, to=3-1]
	\arrow["{g^*}", from=3-3, to=5-1]
\end{tikzcd}\]
		We now examine the effect on the derived de Rham cohomology. We have the following map
		\[
			\cdR_{Y}\xrightarrow{H^*} \cdR_{X\times \R}
		\]
		Since \(\CI(\R)\) is free, we have a completely explicit description of \(\cdR_{\R}\). Namely, this is just the usual de Rham complex of \(\R\) concentrated in degrees \(0\) and \(1\). It now remains to observe that because \(\cdR_{\R}\) has a finite filtration, we have an equivalence as follows. 
		\[
			\cdR_{X\times \R}\simeq \cdR_X\otimes^\infty \Lambda_{\CI\{t\}}[dt]
		\]
		Here \(\otimes^\infty\) is the usual \(\CI\)-tensor product between \(\CI(X)\) and \(\CI(\R)\) and the commutative tensor product for other terms of the filtration. 
		Hence, we have a chain homotopy between the maps \(f^*\) and \(g^*\) in the usual sense.
	\end{proof}
	\begin{remark}
		Note that the condition of being smoothly contractible is subtle. In particular, a priori it does not coincide with contractibility of the underlying topological space since the homotopy should not only preserve the set of points, but also the \emph{\(\CI\)-structure}. The following example gives an instance where such a local contraction exists even though the \(\CI\)-structure is non-trivial. 
	\end{remark}
	\begin{example}\label{example_contraction}
		Consider the following derived affine \(\CI\)-scheme. 
		\[
			Z(e^{-\frac{1}{x^2}})\sse \R.
		\]
		This scheme has a single closed point \(0\in \R\), still its \(\CI\)-structure is non-trivial. However, a smooth contracting homotopy exists. Denote by \(\psi(t)\) a \(\CI\)-function of \(t\) which takes values between \(0\) and \(1\), is flat at \(0\) and \(1\), plateaus at \(0\) for \(t\le 0\), and at \(1\) for \(t\ge 1\).
		\begin{figure}[H]
			\includegraphics[width=0.5\textwidth]{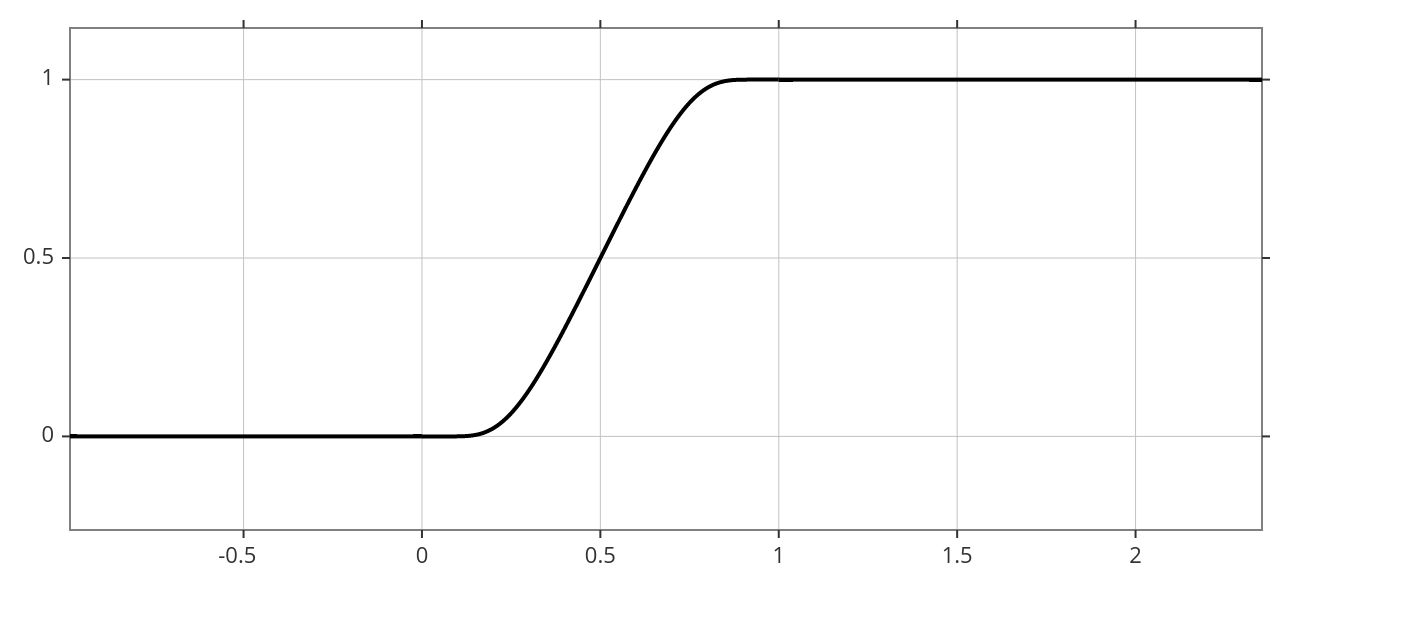}
			\caption{A possible choice of \(\psi(t)\).}
		\end{figure}
		Consider the following map.
		\[
			H^*:\CI\{x\}=\CI(\R)\to \CI(\R)\otimes^\infty\CI(\R)=\CI\{x,t\},\quad x\mapsto \psi(t)\cd x.
		\]
		This map clearly is a morphism of \(\CI\)-algebras.

		Now, we need to check that this map descends to a morphism \(Z(e^{-\frac{1}{x^2}})\times \R\to Z(e^{-\frac{1}{x^2}})\). This is equivalent to checking that we have the following identity.
		\[
			H^*(e^{-\frac{1}{x^2}})=e^{-\frac{1}{x^2}}\cd \vp(x,t),\quad \vp \in \CI\{x,t\}.
		\]
		We thus need to \enquote{divide} \(e^{-\frac{1}{\psi^2(t) x^2}}\) by \(e^{-\frac{1}{x^2}}\). Thus, it is enough to see that \(e^{-\frac{1-\psi(t)^2}{\psi(t)^2}}\) is a regular function. This is immediate for \(0<t\le 1\). At \(0\) this is also clear since in \(\frac{1-\psi(t)^2}{\psi(t)^2}\) the numerator is strictly positive and the denominator is flat and non-negative. Thus, we showed that the map \(H^*\) descends to a morphism of derived \(\CI\)-schemes and gives the desired contracting homotopy. 
	\end{example}
	\begin{remark}
		As we saw in \Cref{example_contraction}, the existence of a smooth contracting homotopy relied on an explicit form of the functions used to define the scheme. This condition is thus more reminiscent of \(\AA^1\)-contractibility in algebraic geometry (see \cite{Morel} for a textbook exposition) than the usual topological contractibility for, say, manifolds where specific homotopies and equations play virtually no role.
	\end{remark}
	Recall the definition of a smooth morphism given in \Cref{def:smooth_morphism}. 	
	\begin{proposition}[Smooth descent for derived de Rham cohomology]\label{st:smooth_descent_for_de_Rham}
		Let \(f:U\to X\) be a smooth surjective morphism of derived manifolds. Then the derived de Rham cohomology of \(X\) can be calculated by descent as follows.
		\[
			\cdR_X\simeq \Tot(\cdR_{\check{C}^\bt(f)}).
		\]
		Here \(\check{C}^\bt(f)\) is the \v{C}ech groupoid of the morphism \(f\).
	\end{proposition}
	\begin{proof}
		By \Cref{implicit_function_theorem} locally on \(U\) we can present the morphism \(f\) as the projection \(X\times \R^n\to X\). Consequently, a local section of the projection exists. The statement now follows by \Cref{section_means_trivial_de_Rham}.
	\end{proof}
	\subsection{De Rham cohomology of derived stacks}
	\begin{construction}\label{con:derived_de_Rham_of_stacks}
		Let \(\X\) be a derived \(\CI\)-stack. Then we have the following construction of the derived de Rham cohomology algebra of \(\X\) by a Kan extension from the affine case.
		\[
			\cdR_{\X}:=\lim_{\Spec A\to \X} \cdR_{\Spec A}.
		\]
		Note that the Kan extension on the right is taken in the category of complete filtered commutative algebras.
	\end{construction}
	\section{de Rham theorems for Hodge completed derived de Rham cohomology}\label{sec:de_Rham_theorems_for_cdR}
	\subsection{A general result}
	The following result can be considered a refinement of the Poincar\'{e} lemma for a class of local \(\CI\)-algebras. In particular, it shows that while the naive form of the de Rham theorem does not hold for general \(\CI\)-algebras, there is still a way to express the cohomology of the constant sheaf using differential forms. 
	
	\begin{proposition}\label{prop:de_Rham_of_completion}
		Let \(A\) be a finitely generated \(0\)-truncated derived \(\CI\)-algebra. Consider a presentation of \(A\) as a quotient of a free derived \(\CI\)-algebra.
		\[
			I\to E\to A
		\]
		Assume that the classical completion of the morphism \(E\to A\) coincides with the derived completion of the morphism \(E\to A\). Then, there is a fibre sequence as follows.
		\[
			I^\infty\Omega^\bullet_E\to \Omega^\bullet_E\to \cdR_{A}
		\]
		Here \(I^\infty\) is the intersection ideal \(\bigcap_{n=1}^\infty I^n\). 
	\end{proposition}
	\begin{proof}
		Since the derived and classical completions coincide by assumption, we have the following presentation of the derived completion of \(E\to A\).
		\[
			\bigcap_{n\ge 0}I^n\to E\to E^\wcom_I
		\]
		The leftmost term is, of course, equal to \(I^\infty\). The result now follows from the fact that both \(E\) and \(E^\wcom_I\) are formally smooth. 
	\end{proof}
	\begin{corollary}
		The term \(I^\infty\Omega^\bt_E\) in the fibre sequence of \Cref{prop:de_Rham_of_completion} is independent of the choice of the presentation of \(A\).
	\end{corollary}
	\begin{proof}
		By \Cref{prop:de_Rham_of_completion} the term \(I^\infty\Omega^\bt_E\) is a fibre of the canonical map \(\Omega^\bt_E\to \cdR_{A}\). Since the derived completion of \(E\to A\) is independent of the choice of the presentation and \(\Omega^\bt_E\simeq \R\), the result follows.
	\end{proof}
	We give the following definition to summarize the properties a derived manifold should satisfy to admit a de Rham theorem.
	\begin{definition}\label{def:ldRa}
		A derived manifold \(\M\) is said to be \emph{algebraically locally acyclic} (or \ldRa) if for every closed point \(x\in \M\) the stalk of the de Rham complex \(\cdR_{\M,x}\) is equivalent to \(\R\).
	\end{definition}
	\begin{theorem}\label{st:local_to_global_de_Rham}
		Let \(\M\) be an \ldRa derived manifold. Then, the de Rham cohomology of \(\M\) is isomorphic to the cohomology of the constant sheaf.
	\end{theorem}
	\begin{proof}
		The result is immediate because the topological space underlying any derived manifold has (by definition) a covering dimension bounded by the covering dimension of the ambient affine space.
	\end{proof}
	\subsection{de Rham theorem for analytic spaces}
	The following result of Brasselet--Pflaum \cite{Brasselet_Pflaum} proves that the term \(I^\infty \Omega_{\R^n}^\bt\) is acyclic for an analytic subspace of \(\R^n\).
	\begin{theorem}[{\cite[Theorem 7.1]{Brasselet_Pflaum}}]\label{thm:Brasselet_Pflaum_acyclicity}
		Consider a subanalytic subset \(X\) of \(\R^n\). Denote by \(I^\infty_X\) of functions flat at \(X\). Then, the complex \(I^\infty_X\Omega^\bt_{\R^n}\) of differential forms on \(\R^n\) with coefficient from the ideal \(I^\infty_X\) is acyclic.
	\end{theorem}
	\begin{theorem}\label{thm:de_Rham_for_analytic}
		Given a derived manifold \(X\) defined by analytic equations. Then it is an \ldRa derived manifold (\Cref{def:ldRa}). Consequently, the cohomology of the constant sheaf on \(X\) is isomorphic to the derived de Rham cohomology of \(X\).
	\end{theorem}
	\begin{proof}
		First, we observe that we have the following fibre sequence by \Cref{prop:hypo_Noetherian_completion}.
		\[
			I^\infty \Omega^\bt\to \Omega^\bt_{\R^n}\simeq\ul{\R}\to \Omega_{E^\wedge_I}^\bt
		\]
		Second, we claim that the ideal \(I^\infty\) in the case when \(I\) is generated by real analytic functions coincides with the ideal of functions which are flat on the set of points of \(X\). Indeed, it consists of functions which vanish on \(X\). Moreover, if a function is flat on \(X\), it is divisible by any analytic function vanishing on \(X\); hence, it is in the ideal \(I^\infty\). Conversely, any function in \(I^\infty\) vanishes on \(X\) up to any order.
		
		Finally, because the ideal \(I^\infty\) consists of functions which are flat on \(X\), the complex \(I^\infty \Omega^\bullet\) is acyclic by \Cref{thm:Brasselet_Pflaum_acyclicity} and the first part of the result follows. The second part of the statement follows from \Cref{st:local_to_global_de_Rham}. 
	\end{proof}
	\begin{remark}
		A more elegant proof of the above result can be given using the notion of a \emph{multiplier}, a kind of meromorphic function in the \(\CI\)-setting. However, we are not pursuing this direction here. See \cite[\S IV.4]{Tougeron} for more detail on multipliers.
	\end{remark}
	\subsection{de Rham theorem and the Łojasiewicz condition}
	Using \Cref{thm:Brasselet_Pflaum_acyclicity} of Brasselet--Pflaum, one also easily shows the de Rham theorem for derived de Rham cohomology on derived manifolds defined by \emph{principal Łojasiewicz} ideals with subanalytic zero set. 
	\begin{definition}
		A finitely generated ideal \(I=(f_1,\ldots,f_k)\in \CI(\R^n)\) is \emph{Łojasiewicz} if either of the two equivalent conditions are satisfied.
		\begin{enumerate}
			\item \(\m^\infty_{Z(I)}\sse I\);
			\item The function \(f_1^2+\ldots+f_k^2\) is Łojasiewicz, that is, the following inequality holds for every compact subset \(K\sse \R^n\) for some real constants \(C,\lambda\ge 0\):
			\[
				f_1^2(x)+\ldots+f_k^2(x)\ge C \dist(x,K)^\lambda.
			\]
		\end{enumerate}
	\end{definition}
	The following result is due to Tougeron.
	\begin{proposition}[{\cite[Proposition 4.3]{Tougeron}}]\label{prop:Tougeron_Lojasiewicz}
		Let \(I\) be a finitely generated Łojasiewicz ideal. Then there is an inclusion of ideals \(\m^\infty_{Z(I)}\sse I\).
	\end{proposition}
	The following corollary is immediate.
	\begin{corollary}
		Let \(I\) be a finitely generated Łojasiewicz ideal. Then the ideal \(I^\infty=\cap_{m=1}^\infty I^m\) coincides with the ideal \(\m^\infty_{Z(I)}\).
	\end{corollary}
	\begin{proof}
		Observe, that if \(h\in I^\infty\) then \(h\in I^m\) for all \(m\). Thus, \(h\) vanishes on \(Z(I)\) up to any order. Consequently, \(h\in \m^\infty_{Z(I)}\). The other inclusion follows from \Cref{prop:Tougeron_Lojasiewicz}.
	\end{proof}
	The following result is due to Bierstone--Milman.
	\begin{proposition}[{\cite[Theorem 6.4, Remark 6.5]{Bierstone_Milman_Subanalytic}}]
		Let \(f\) be a subanalytic function. Then the ideal \((f)\) is Łojasiewicz. 
	\end{proposition}
	\begin{theorem}\label{thm:de_Rham_for_Lojasiewicz}
		Let either of the two conditions be satisfied.
		\begin{enumerate}
			\item Let \(I=(f)\) be a principal Łojasiewicz ideal. Assume additionally that \(Z(I)\) is a subanalytic set in \(\R^n\). Both conditions are satisfied, for instance, when \(f\) itself is subanalytic.
			\item Let \(I=(f_1,\ldots,f_k)\) be a Łojasiewicz ideal such that \(\{f_1,\ldots,f_k\}\) form a regular sequence. Assume additionally that \(Z(I)\) is a subanalytic set in \(\R^n\).
		\end{enumerate}
		Then \(Z(I)\) with the \(\CI\)-structure given by \(\CI(\R^n)/I\) is an \ldRa derived manifold (\Cref{def:ldRa}). Consequently, the derived de Rham cohomology of \(Z(f)\) is isomorphic to the cohomology of the constant sheaf.
	\end{theorem}
	\begin{proof}
		By \Cref{thm:Brasselet_Pflaum_acyclicity} we know that the complex \(\m_{Z(I)}^\infty\Omega^\bt_{\CI(\R^n)}\) is acyclic. It remains to note that because \(I\) is a Łojasiewicz ideal the ideal \(\cap_{m=1}^\infty I^m\) coincides with \(\m_{Z(I)}^\infty\) and the derived completion of \(\CI(\R^n)\) in \(I\) coincides with the classical completion by \Cref{prop:completion_of_classical_is_classical}. The last part of the statement follows from \Cref{st:local_to_global_de_Rham}.
	\end{proof}
	\subsection{Counterexample to de Rham theorem for derived manifolds}
	In the following example, we demonstrate that the term \(I^\infty\cdR_A\) in the fibre sequence of \Cref{prop:de_Rham_of_completion} is generally non-trivial.
	\begin{example}\label{ex:de_Rham_counterexample}
		First, we provide a counterexample to the local de Rham theorem, i.e. the Poincar\'{e} lemma.

		Consider a function on \(\R\), which is flat at the origin, non-flat at other points, non-negative, and whose set of zeroes is a null sequence, together with the origin. Such a function is not hard to construct. For example, the following explicit formula (with appropriate extensions by zero) will do.
		\[
			\vp(x)=\sin^2\left(\frac{1}{x}\right)
			\cd e^{-\frac{1}{x^2}}.
		\]

		We consider a derived \(\CI\)-scheme \(X\) defined as the derived fibre of this map \(\R\xrightarrow{\vp}\R\). Denote by \(E\) the free \(\CI\)-algebra \(\CI(\R)\). Then we have a presentation of \(\CI(X)\) as follows. 
		\[
			0\to I=(\vp)\to E\to \CI(X)\to 0.
		\]
		Observe, that since \(\vp\) is a non-zero divisor by \Cref{ex:non-zero_div_in_CI} hence the derived \(\CI\)-algebra \(\CI(X)\) is \(0\)-truncated. Moreover, by \Cref{prop:completion_of_classical_is_classical} and \Cref{prop:de_Rham_of_completion}, the derived de Rham cohomology of \(X\) has the following presentation. 
		\[
			I^\infty\Omega^\bullet_E\to \Omega^\bullet_E\to \cdR_{X}
		\]

		Consider the function \(\tau=e^{-\frac{1}{\vp}}\). Then, the function \(\tau\) belongs to \(I^\infty\). However, the form \(\tau dx\) is not exact in the complex \(I^\infty\cdR_X\). Indeed, the only possible antiderivative of \(\tau dx\) that could belong to the ideal \(I^\infty\) is the function \(T\) equal to \(0\) at the origin. However, since \(\tau\) is non-negative on \(\R\) and strictly positive outside a countable set, its integral is strictly positive outside \(0\). Thus, the form \(\tau dx\) is not exact. This shows that at the point \(0\in\R\) the stalk of \(\cdR_X\) is not quasi-isomorphic to \(\R\). More precisely, we get the following long exact sequence in cohomology.

		\[\begin{tikzcd}
					{H^0\left(\left(I^\infty\Omega_\R^\bt\right)_0\right)=0} && {H^0\left(\left(\Omega_\R^\bt\right)_0\right)=\R} && {H^0\left(\left(\cdR_X\right)_0\right)} \\
					\\
					{H^1\left(\left(I^\infty\Omega_\R^\bt\right)_0\right)=\R\oplus\ldots} && {H^1\left(\left(\Omega_\R^\bt\right)_0\right)=0}
					\arrow[from=1-1, to=1-3]
					\arrow[from=1-3, to=1-5]
					\arrow[from=1-5, to=3-1]
					\arrow[from=3-1, to=3-3]
		\end{tikzcd}\]
		Now, we see that the stalk of \(\cdR_X\) at \(0\) has at least two-dimensional degree \(0\) cohomology. In fact, one can similarly construct infinitely many non-homologous non-exact closed forms by iterating the application of \(e^{-\frac{1}{-}}\). Thus, we see that the dimension of degree \(0\) cohomology of the stalk is, in fact, infinite. 

		We turn now to showing that this example can be extended to a counterexample to the global de Rham theorem for derived \(\CI\)-schemes. For this, we must compare (derived) global sections of the constant sheaf and the Hodge-completed de Rham cohomology sheaf. 

		Note that as a topological space, the zero set of \(\vp\) is homeomorphic to \(\N\cup\{0\}\) with the standard topology of the one-point compactification. We consider the following diagram of inclusions.
		\[
		\{0\}\xhookrightarrow{i} \N\cup\{0\}\xhookleftarrow{j} \N.
		\]
		Thus, we get the following pair of fibre sequences and a map between them. For a general six functor formalism in a topological context, we use here see \cite{Volpe_Six}.
		\[\begin{tikzcd}
			{j_!j^! \ul{\R}} && {\ul{\R}} && {i_*i^*\ul{\R}} \\
			{j_!j^!\widehat{\operatorname{dR}}_X} && {\widehat{\operatorname{dR}}_X} && { i_*i^*\widehat{\operatorname{dR}}_X}
			\arrow[from=1-1, to=1-3]
			\arrow[from=1-1, to=2-1]
			\arrow[from=1-3, to=1-5]
			\arrow[from=1-3, to=2-3]
			\arrow[from=1-5, to=2-5]
			\arrow[from=2-1, to=2-3]
			\arrow[from=2-3, to=2-5]
		\end{tikzcd}\]
		We now observe that the left-most vertical arrow is an equivalence by the standard de Rham theorem for manifolds, and the right-most vertical arrow is not an equivalence based on the previous discussion. Hence, the middle arrow is not an equivalence. 
	\end{example}
	\subsection{de Rham theorem implies conical o-minimality}
	\begin{definition}
		A smooth curve \(\gamma:(-\ve,\ve)\to \R\) in \(\R^n\) is called \emph{non-degenerate} at a point \(p=\gamma(0)\in \R^n\) if the derivative of \(\gamma\) at \(0\) is non-zero.
	\end{definition}
	\begin{definition}
		A subset of \(\R^n\) is called \emph{conically o-minimal} at a point \(p\) if its intersection with any non-degenerate curve through \(p\) is an o-minimal subset of \(\R\). In particular, for closed subsets of \(\R^n\), this means that the intersection of the subset with any non-degenerate curve is a finite union of closed intervals.
	\end{definition}
	\begin{proposition}
		Let \(X\) be a derived manifold defined by an ideal \(I=(\vp).\)	The de Rham theorem for \(\cdR\) at a point \(p\) implies conical o-minimality of the underlying closed subset of \(X\) in \(\R^n\) at \(p\).
	\end{proposition}
	\begin{proof}
		The result follows from \Cref{ex:de_Rham_counterexample}. Indeed, in \(1\)-dimensional case, one easily constructs a non-exact form on a closed subset of \(\R\) which is not conically o-minimal. In general we apply the following integration argument. Consider an arbitrary non-degenerate curve \(\gamma\) through the given point \(p.\) Since it is non-degenerate we can rectify it to be just the first coordinate axis.

		Consider the case of dimension \(n=2\), the general case is very similar. Consider a form \(\omega=f dx\wedge dy\) for some non-negative \(f\in I^\infty.\) It is automatically closed. We show that it is non-exact as follows. Integrate \(\omega\) along the flow of the constant vector field corresponding to the first coordinate. Then for the same reason as in \Cref{ex:de_Rham_counterexample} the resulting \(1\)-form on \(\R\) can't belong to \(I^\infty\Omega^{1}_{\R^2}\). Thus, the original form can not be exact in \(I^\infty\Omega^\bt_{\R^2}\).
	\end{proof}
	\begin{remark}
		If instead of derived de Rham cohomology we consider the de Rham--Whitney cohomology of \cite{Brasselet_Pflaum}, then the above result holds with the same proof for any closed subset of \(\R^n\).
	\end{remark}
	\part{General de Rham theory}
	\section{de Rham stacks in derived differential geometry}\label{sec:de_Rham_stacks}
	In this section, we define and compare two forms of the de Rham stack of a derived \(\CI\)-stack. This discussion mostly follows the approach of Borisov--Kremnizer \cite{Borisov_Kremnizer} with the main difference being that we work with derived \(\CI\)-stacks built out of derived manifolds instead of classical \(\CI\)-schemes.	
	
	First, we introduce de Rham stacks associated to a derived \(\CI\)-stack. The idea of this definition goes back to the work of Simpson \cite{Simpson} and Simpson--Teleman \cite{Simpson_Teleman} and in a more sophisticated set-up of derived algebraic geometry to Gaitsgory--Rozenblyum \cite{Gaitsgory_Rozenblyum_Crys}. 
	
	There is however, a crucial subtlety here which was realized by Borisov and Kremnizer in \cite{Borisov_Kremnizer}. Namely, \(\CI\)-algebras carry two natural notions of \enquote{reducible elements}. The first one is the usual notion of reducible (or \emph{nilpotent}) elements for the underlying commutative ring. The second one is the notion of reducible elements for the \(\CI\)-algebra. These were first emphasized by Moerdijk--Reyes \cite{Moerdijk_Reyes_Localizations} in the context of their categorical properties.
	
	\subsection{de Rham stacks and infinitesimal groupoids} 
	
	The following construction is due to Borisov--Kremnizer \cite[\S 3]{Borisov_Kremnizer}.
	\begin{construction}\label{con:de_Rham_stacks}
		Let \(\M\) be a derived manifold. Then we have two natural \enquote{de Rham} stacks associated to it.
		\[
			\M^{\dR_{\CI}}(A):=\M(A^{\red_\CI}),\quad \M^{\dR_\Com}(A):=\M(A^{\red_\Com}).
		\]
		Here, we consider the derived manifold \(\M\) as a spectrum of a derived \(\CI\)-algebra, and \(\M(A^{\red})\) is calculated as the mapping space in the \(\infty\)-category of \(\CI\)-algebras. This remark is important, since \(A^{\red_\CI}\) and \(A^{\red_\Com}\) are usually not \emph{finitely presented} derived \(\CI\)-algebras.

		For an arbitrary derived \(\CI\)-stack \(\X\) we define the \emph{derived de Rham stack} as the left Kan extension of the above construction.
	\end{construction}
	\begin{proposition}
		Let \(\X\) be a derived \(\CI\)-stack. Then, the derived \(\CI\)(resp. \(\Com\))-de Rham stack is the same as the derived \(\CI\) (resp. \(\Com\))-completion of the canonical morphism \(\X\to *\) in the sense of \Cref{def:formal_completion}.
	\end{proposition}
	\begin{definition}[{Infinitesimal groupoids}] 
		Let \(\X\) be a derived \(\CI\)-stack. Recall the notion of the formal and \(\CI\)-completion of a morphism of derived \(\CI\)-prestacks (see \Cref{def:formal_completion}).
		We define its \emph{\(\CI\)-infinitesimal groupoid} \(\X^{\inft_{\CI}}\) as follows.
\[\begin{tikzcd}
	\vdots \\
	{(\X^3)^{\wedge_{\CI}}_{\Delta_3}} \\
	{(\X^2)^{\wedge_{\CI}}_{\Delta_2}} \\
	\X
	\arrow[shift right=3, from=3-1, to=4-1]
	\arrow[shift left=3, from=3-1, to=4-1]
	\arrow[shift right=4, from=2-1, to=3-1]
	\arrow[shift left=4, from=2-1, to=3-1]
	\arrow[from=2-1, to=3-1]
	\arrow[shift right=5, from=1-1, to=2-1]
	\arrow[shift left=4, from=1-1, to=2-1]
	\arrow[shift right=2, from=1-1, to=2-1]
	\arrow[shift left, from=1-1, to=2-1]
\end{tikzcd}\]

		We define its \(\Com\)-\emph{infinitesimal groupoid} \(\X^{\inft_{\Com}}\) as follows.
\[\begin{tikzcd}
	\vdots \\
	{(\X^3)^{\wedge_{\Com}}_{\Delta_3}} \\
	{(\X^2)^{\wedge_{\Com}}_{\Delta_2}} \\
	\X
	\arrow[shift right=3, from=3-1, to=4-1]
	\arrow[shift left=3, from=3-1, to=4-1]
	\arrow[shift right=4, from=2-1, to=3-1]
	\arrow[shift left=4, from=2-1, to=3-1]
	\arrow[from=2-1, to=3-1]
	\arrow[shift right=5, from=1-1, to=2-1]
	\arrow[shift left=4, from=1-1, to=2-1]
	\arrow[shift right=2, from=1-1, to=2-1]
	\arrow[shift left, from=1-1, to=2-1]
\end{tikzcd}\]
	\end{definition}
	\begin{proposition}\label{st:de_Rham_stack_is_presented_by_the_infinitesimal_groupoid}
		Under the assumption that \(\X\) is a formally smooth derived \(\CI\)-stack, the derived \(\Com\)-de Rham stack of \(\X\) is equivalent to the quotient of the \(\Com\)-infinitesimal groupoid of \(\X\).
	\end{proposition}
	\begin{proof}
		Let \(S\) be a derived manifold. Then by the formal smoothness of \(\X\) we have the following \(\pi_0\)-epimorphism of mapping spaces.
		\[
			\Map(S,\X)\to \Map(S^{\red_{\Com}},\X). 
		\]
		The kernel of this map is by definition represented by the \(\Com\)-infinitesimal groupoid of \(\X\). This follows directly from how we defined the \(\Com\)-completion in \Cref{completions_of_derived_CI_stacks}.
	\end{proof}
	
	The following result allows one to identify functions on de Rham stacks in terms of functions on the infinitesimal groupoid.
	\begin{theorem}\label{thm:functions_on_inf_vs_functions_on_dR}
		Consider an arbitrary derived affine \(\CI\)-scheme \(X\) and assume that it is embedded into a formally smooth scheme \(Y\). Then there is an equivalence of the following form.
		\[
			\CI(X^{\inft_\Com})\simeq \CI(X^{\dR_{\Com}}).
		\]
	\end{theorem}
	\begin{proof}
		We start by replacing \(Y\) with \(Y^{\wcom}_X\). Since \(Y\) is formally smooth \(Y^{\dR_{\Com}}\) is equivalent to the quotient of \(Y^{\inft_\Com}\) by \Cref{st:de_Rham_stack_is_presented_by_the_infinitesimal_groupoid}. In addition this forces the following equivalence to hold.
		\[
			Y^{\dR_\Com}\simeq X^{\dR_\Com}.
		\]
		Now we consider the following bisimplicial diagram of pullbacks of formal completions in the main diagonals.
\[\begin{tikzcd}[ampersand replacement=\&]
	\&\&\&\& \vdots \&\& \vdots \\
	\\
	\&\& \ddots \&\& {X\times_{Y^{\dR_\Com}}(X\times X)^{\wcom}_{\Delta_2}} \&\& {(X\times X)^{\wcom}_{\Delta_2}} \\
	\\
	\dots \&\& {X\times_YX\times_{Y^{\dR_\Com}}Y} \&\& {X\times_{Y^{\dR_\Com}}Y} \&\& X \\
	\\
	\dots \&\& {X\times_YX} \&\& Y \&\& {X^{\dRcom}}
	\arrow[shift right=4, from=1-5, to=3-5]
	\arrow[shift left=4, from=1-5, to=3-5]
	\arrow[from=1-5, to=3-5]
	\arrow[shift right=4, from=1-7, to=3-7]
	\arrow[shift left=4, from=1-7, to=3-7]
	\arrow[from=1-7, to=3-7]
	\arrow[shift right=2, from=3-3, to=3-5]
	\arrow[shift left=2, from=3-3, to=3-5]
	\arrow[shift right=2, from=3-3, to=5-3]
	\arrow[shift left=2, from=3-3, to=5-3]
	\arrow["{\mathbf{\Large\lrcorner}}"{description, pos=0.1}, draw=none, from=3-3, to=5-5]
	\arrow[two heads, from=3-5, to=3-7]
	\arrow[shift right=3, from=3-5, to=5-5]
	\arrow[shift left=3, from=3-5, to=5-5]
	\arrow["\lrcorner"{anchor=center, pos=0.125}, draw=none, from=3-5, to=5-7]
	\arrow[shift right=3, from=3-7, to=5-7]
	\arrow[shift left=3, from=3-7, to=5-7]
	\arrow[from=5-1, to=5-3]
	\arrow[shift left=3, from=5-1, to=5-3]
	\arrow[shift right=3, from=5-1, to=5-3]
	\arrow[shift right=3, from=5-3, to=5-5]
	\arrow[shift left=3, from=5-3, to=5-5]
	\arrow[two heads, from=5-3, to=7-3]
	\arrow["\lrcorner"{anchor=center, pos=0.125}, draw=none, from=5-3, to=7-5]
	\arrow[two heads, from=5-5, to=5-7]
	\arrow[two heads, from=5-5, to=7-5]
	\arrow["\lrcorner"{anchor=center, pos=0.125}, draw=none, from=5-5, to=7-7]
	\arrow[from=5-7, to=7-5]
	\arrow[from=5-7, to=7-7]
	\arrow[from=7-1, to=7-3]
	\arrow[shift left=3, from=7-1, to=7-3]
	\arrow[shift right=3, from=7-1, to=7-3]
	\arrow[shift right=3, from=7-3, to=7-5]
	\arrow[shift left=3, from=7-3, to=7-5]
	\arrow[two heads, from=7-5, to=7-7]
\end{tikzcd}\]

		Observe that along the columns of the diagram, except for the left-most, there is an equivalence on functions. Indeed, by assumption \(Y=Y^\wcom_X\) hence \(X\times_{Y^{\dR_\Com}}Y=Y.\) Hence the vertical bottom arrows are epimorphisms induce an equivalence on functions. 
		
		Along the bottom row, there is an equivalence on functions by \Cref{prop:Adams_vs_Derived} since the limit on functions is precisely the Adams completion and by assumption \(Y\) is the derived completion of \(X\) in \(Y\). Finally, it remains to note that \(Y\) is smooth and thus \(Y\) maps to \(X^{dR_{\Com}}\) via an effective epimorphism, the map on functions is an equivalence. Hence, the limit of the entire bisimplicial diagram of function algebras is an equivalence.
	\end{proof}
	\begin{remark}
		One similarly proves that the functions on the \(\CI\)-de Rham stack are equivalent to functions on the \(\CI\)-infinitesimal groupoid. However, at the moment, there is no direct way to identify the latter with something more explicit, so we forego a detailed proof of this result. 
	\end{remark}
	\subsection{de Rham cohomology as functions on the de Rham stack}
	\begin{definition}\label{def:CI-de_Rham-coh} For a derived differentiable stack we define the \(\CI\)-de Rham cohomology of \(\X\) as the functions on the de Rham stack of \(\X\) defined in \Cref{con:de_Rham_stacks}. Similarly, we define the \(\Com\)-de Rham cohomology of \(\X\) as the functions on the de Rham stack of \(\X\).
	\end{definition}
	\begin{theorem}\label{thm:de_Rham_vs_functions_on_inf}
		Let \(\X\) be a derived \(\CI\)-stack. Then, the derived de Rham cohomology of \(\X\), as defined in \Cref{con:de_Rham_stacks}, is equivalent to the functions on the de Rham stack of \(\X\) defined in \Cref{con:derived_de_Rham_of_stacks}.
	\end{theorem}
	We use the following lemma for the affine case to prove this result.
	\begin{lemma}\label{lem:de_Rham_vs_functions_on_inf}
		Let \(X=\Spec(A)\) be a derived manifold. Then, derived de Rham cohomology \(\cdR_X\) is equivalent to the ring of functions on the \(\Com\)-infinitesimal groupoid of \(X\).
	\end{lemma}
	\begin{proof}
		Denote by \(\Delta_\bt\) the morphism \(X\to X^\infcom\). 
		The morphism \(\Delta_\bt\) induces a collection of fibre sequences of cotangent complexes as follows.
		\[
			\Delta_n^*\LL_{X^\infcom_n}\to \LL_{X}\to \LL_{\Delta_n}.
		\]
		Note that the limit of the leftmost term of this sequence with respect to \(n\) vanishes by the connectivity estimates for tensor products. Hence, there is an equivalence of the following form.
		\[
			\cdR_X\to \Tot_n(\cdR_{\Delta_n}).
		\]
		
		Now, observe, that by \Cref{thm:Adams_completion_vs_de_Rham} there is an equivalence between \(\cdR_{\Delta_n}\) and the Adams completion of \(X\to X^\infcom_n\) in each degree. Hence, by passing to the limit there is an equivalence as follows.
		\[
			\Tot(\cdR_{\Delta_n}) \simeq \CI(X^\infcom).
		\]
		Thus, we conclude that there is the desired equivalence of the following form.
		\[
			\cdR_X\simeq \CI(X^\infcom).\qedhere
		\]
	\end{proof}
	\begin{proof}[\ul{Proof of \Cref{thm:de_Rham_vs_functions_on_inf}}]
		By \Cref{lem:de_Rham_vs_functions_on_inf} and \Cref{thm:functions_on_inf_vs_functions_on_dR} derived de Rham cohomology of \Cref{con:derived_de_Rham} is the same as the algebra of functions on the de Rham stack for any derived manifold.

		Moreover, we observe that for a derived prestack \(\X\) the de Rham stack can be obtained as a left Kan extension from de Rham stacks of derived manifolds. That is the following equivalence holds.
		\[
			\X^{\dR_{\Com}}\simeq \colim_{\Spec(A)\to \X}\Spec(A)^{\dR_{\Com}}.
		\]
		This is immediate since limits and colimits in the category of derived \(\CI\)-prestacks \(\CI\PreStk\) are computed objectwise (i.e. the argument is the same as in the algebraic case, see \cite[Lemma 1.1.4]{Gaitsgory_Rozenblyum_Crys}). As a result the \(\CI\)-algebra of functions on \(\X^{\dR_{\Com}}\) is the limit of the \(\CI\)-algebras of functions on the de Rham stacks of derived manifolds. Thus derived de Rham cohomology of a derived \(\CI\)-stack defined in \Cref{con:derived_de_Rham_of_stacks} is equivalent to the algebra of function on the de Rham stack of \(\X\).
	\end{proof}
	\begin{corollary}\label{ala_is_determined_by_classical}
		A derived manifold \(\M\) is \ldRa if and only if its classical part is \ldRa. That is, if derived de Rham cohomology sheaf of \(\pi_0(\CI(\M))\) defined as functions on \(\pi_0(\M)^{\dR_{\Com}}\) is locally acyclic.
	\end{corollary}
	\begin{proof}
		The result follows from the fact that \(\M\) and its classical part have the same reduced part and \Cref{thm:de_Rham_vs_functions_on_inf}.
	\end{proof}
	\begin{remark}
		Note that there is a subtelty here, since the classical part \(\pi_0(\M)\) might not be finitely presented as a derived algebra even if \(\M\) itself is. An example of this is given by a ring of the form \(\CI\{x,y\}/ (x^2,xy,y^2).\) Here, if we understand this ring as a strict quotient it is not finitely presented as a derived algebra. Indeed, it can be viewed just as a commutative ring since it is Artinian it has a canonical \(\CI\)-structure. Thus it falls into the purview of Avramov's theorem \cite[Theorem A]{Avramov_Halperin} and we can see that since it is not lci its cotangent complex has infinite amplitude. However, as a homotopy quotient it is a perfectly good finitely presented algebra.
	\end{remark}
	The following result is immediate from the functorial description of the de Rham cohomology afforded by \Cref{thm:de_Rham_vs_functions_on_inf}.	
	\begin{corollary}
		Let \(A\to B\) be a morphism of finitely presented derived \(\CI\)-algebras such that the induced map on \(\Com\)-reductions is an equivalence. Then the induced map on derived de Rham cohomology algebras is an equivalence. In particular, the derived de Rham cohomology of a derived manifold is equivalent to the derived completion of this manifold in a smooth ambient space.
	\end{corollary}
	\subsection{Comparison morphism for de Rham stacks}
	In this section we spell out the comparison morphisms between \(\Com\) and \(\CI\) de Rham stacks and infinitesimal groupoids.
	\begin{proposition}\label{prop:comparison_morphism_for_de_Rham_stacks}
		There is a \enquote{comparison} morphism of derived stacks as follows.
		\begin{equation}\label{eq:de_Rham_fibre_sequence}
			\X^{\dR_{\Com}}\to \X^{\dR_{\CI}}
		\end{equation}
		When \(\X\) is formally smooth, this morphism is induced by the morphism of infinitesimal groupoids as follows.
		\begin{equation}\label{eq:infinitesimal_groupoid_comparison}
			\X^{\inft_{\Com}}\to \X^{\inft_{\CI}}
		\end{equation}
	\end{proposition}
	\begin{proof}
		By \Cref{con:reductions_of_CI_rings} we have morphisms of derived \(\CI\)-algebras for a derived \(\CI\)-algebra \(A\) as follows.
		\[
			A\to A^{\red_{\Com}}\to A^{\red_{\CI}}.
		\]
		These, of course, induce morphisms of functors as follows.
		\[
			\X(A)\to \X(A^{\red_{\Com}})\to \X(A^{\red_{\CI}}).
		\]
		This gives us the desired morphism \eqref{eq:de_Rham_fibre_sequence} of derived stacks.
	\end{proof}
	\begin{remark}
		In the smooth case, morphism \eqref{eq:infinitesimal_groupoid_comparison} can be thought of as a quotient map for the action of the \emph{germ} of the diagonal on \(\X^{\inft_{\Com}}\).
	\end{remark}
	\section{de Rham theorems for derived differentiable stacks}\label{sec:de_Rham_theorems}
	\subsection{Preliminaries}\label{sec:preliminaries_on_stacks}
	In this section, we review the basic notions of shape theory and the theory of representability for derived stacks.
	\subsubsection{Shapes and constant sheaves}
	The following subsection follows the approach of Lurie to geometric topology \cite[\S 7]{Lurie_HTT}. Below, we review several basic categorical notions needed to formalize this viewpoint. As mentioned in the introduction, the primary motivation for using this formalism is to treat cases where the underlying topological space of a derived stack does not have enough points. Situations like this are ubiquitous in geometry. For instance, the quotient of the point \(*\) with respect to any Lie group \(G\) has only one point in the naive sense, but has nevertheless an interesting category of sheaves. See \cite[\S 2.3]{Behrend_De_Rham}, \cite[\S 5]{Moerdijk_Mrcun} for other interesting examples of this phenomenon.
	\begin{definition}[{\cite[Definition 5.4.2.1]{Lurie_HTT}}]
		The \(\infty\)-category \(\mathcal{C}\) is \(\kappa\)-\emph{accessible} for some regular cardinal \(\kappa\) if it satisfies the following conditions.
		\begin{itemize}
			\item is locally small,
			\item has all \(\kappa\)-filtered colimits, and
			\item there is some essentially small sub-\(\infty\)-category \(\mathcal{C}' \hookrightarrow \mathcal{C}\) of \(\kappa\)-compact objects which generates \(\mathcal{C}\) under \(\kappa\)-filtered \(\infty\)-colimits.
		\end{itemize}
	\end{definition}
	\begin{definition}[{\cite[Definition 5.5.0.1]{Lurie_HTT}}]
		An $\infty$-category $\mathcal{C}$ is presentable if $\mathcal{C}$ is accessible (for some small cardinal \(\kappa\)) and admits small colimits. The category of presentable $\infty$-categories with left adjoint functors for morphisms is denoted by $\Pr^L$.
	\end{definition}
	\begin{definition}[{\cite[Definition 5.4.2.5]{Lurie_HTT}}]
		If \(\mathcal{C}\) is an accessible \(\infty\)-category, then a functor \(F : \mathcal{C} \to \mathcal{C}'\) is accessible if it is \(\kappa\)-continuous for some regular cardinal \(\kappa\) (and therefore for all regular cardinals \(\tau \geq \kappa\)).
	\end{definition}
	\begin{definition}
		Denote by \(\Shape\) the pro-completion of the category \(\Spc\) of spaces (i.e. \(\infty\)-groupoids). That is the category of accessible left exact functors \(f:\Spc\to \Spc\). 
	\end{definition}
	\begin{construction}[{The shape of a topological space}]\label{con:shape_of_a_space}
		Given a topological space \(X\), we associate to it its shape as follows. Consider the \(\infty\)-topos \(\X\) of \(\infty\)-sheaves on \(X\). Denote by \(q_*:\X \to \Spc\) the global sections functor and by \(q^*\) its left adjoint. Then the shape of \(X\) is given by the following composition \(q_*q^*:\Spc\to \Spc\).
	\end{construction}
	\begin{remark}
		Another more classical approach to shape theory views shapes as formal pro-objects in the category of CW-complexes (see \cite{Shape_history} for a historical overview of the subject). These pro-objects are thus defined as towers of CW-complexes indexed by the natural numbers. The relation between this approach and the approach we use is given by viewing pro-systems as the functors  out of the category of CW-complexes they represent. 
	\end{remark}
	\begin{proposition}
		Given a shape \(X\) arising from a topological space \(\ul{X}.\) Then the constant sheaf cohomology on \(X\) can be calculated using the \enquote{singular complex} arising from the Yoneda lemma.
		\[
			H^*(X,\ul{\R}_X)\simeq \colim_{S\to X} C^*(S;\R).
		\]
		Here \(C^*(S;\R)\) is the singular cochain complex of a CW-complex \(S\).
	\end{proposition}
	\begin{proof}
		The result follows from the \Cref{con:shape_of_a_space}. Indeed, consider the constant sheaf \(\ul{\R}_X=q^*(\ul{\R}_*).\) Then its global sections are given by \(q_*\ul{\R}_X.\) Consequently, the cohomology of the constant sheaf on \(X\) is given by evaluating the shape of \(X\) on the constant sheaf \(\ul{\R}_*.\)
	\end{proof}
	\begin{definition}[{\cite[Definition 7.2.3.1]{Lurie_HTT}}]
		A paracompact topological space \(X\) has covering dimension \(\leq n\) if the following condition is satisfied: for any open covering \(\{U_\alpha\}\) of \(X\), there exists an open refinement \(\{V_\alpha\}\) of \(X\) such that each intersection \(V_{\alpha_0} \cap \cdots \cap V_{\alpha_{n+1}} = \emptyset\) provided the \(\alpha_i\) are pairwise distinct.
	\end{definition}
	\begin{theorem}[{\cite[Corollary 7.2.1.12]{Lurie_HTT}, \cite[Theorem 7.2.3.6]{Lurie_HTT}}]
		If a paracompact topological space \(X\) has covering dimension \(\leq n\), then its \(\infty\) sheaf topos is hypercomplete. That is one can check the equivalences between \(\infty\)-sheaves on \(X\) on stalks.
	\end{theorem}
	\begin{lemma}[{\cite[Theorem 2.7]{Dimension_theory}}]
		Topological space \(\R^n\) has covering dimension \(\leq n\).
	\end{lemma}
	\begin{corollary}\label{st:hypercompleteness_of_dMfld}
		The underlying topological space of any derived manifold has finite covering dimension. Hence, the topos of sheaves on any derived manifold is hypercomplete. 
	\end{corollary}
	\begin{proof}
		This follows from the monotonicity property of covering dimension, see \cite[Theorem 6.4]{Dimension_theory}.
	\end{proof}
	\subsubsection{Shape and sheaves on derived \texorpdfstring{\(\CI\)}{C-infinity}-stacks}
	Following Lurie \cite[\S 7.1]{Lurie_HTT}, we use the category of shapes as a more refined version of homotopy types. This category will target the \emph{underlying shape} function from the category of derived \(\CI\)-stacks.
	\begin{construction}
		 Given a derived manifold \(X\) we define its \emph{shape} \(\Shape(X)\) as the shape of the underlying topological space of \(X\) viewed as a closed subset of \(\R^n\). This homotopy type is a pro-object in the category of spaces via the co-Yoneda embedding. Now we define the shape of a derived \(\CI\)-stack \(\X\) using the following left Kan extension formula.
		\[
			\Shape(\X):=\colim_{\Spec A\to \X}\Shape(\Spec A).
		\]
		Similarly, one defines the category of \emph{Betti sheaves}  \(\Shv_{\Betti}(\X)\) on \(\X\) as the colimit along the same diagram of the presentable categories of sheaves on the underlying topological spaces of derived manifolds. 
	\end{construction}
	\begin{remark}
		Observe that the shape of a derived differentiable stack has the same relation with the category of Betti sheaves on a stack \(\X\) as the shape of a topological space with its category of sheaves in \Cref{con:shape_of_a_space}. That is the shape \(\Shape(\X)\) is equivalent to the composition \(q_*q^*:\Spc\to \Spc\) where \(q^*\) is the left adjoint to the global sections functor \(q_*:\Shv_{\Betti}(\X)\to \Spc\).
	\end{remark}
	\subsubsection{Geometric stacks}\label{sec:geometric_stacks} The following material is standard; the original idea goes back to Simpson \cite{Simpson_n_stacks}. See \cite[\S 1.3]{HAGII}, \cite[\S I.2.4]{Gaitsgory_Rozenblyum_book_1} for a general review in the setting of derived geometry and \cite[\S 4]{Pridham_DG}, \cite[\S 4]{Pridham_Rep}. Also, see \cite[\S 4.2]{Steffens_thesis} for a review specifically in the setting of derived differential geometry.
	In principle, geometric stacks can be defined using the following one-line definition. 
	\begin{informal_definition}\label{def:geometric_stack}
		An \(n\)-geometric stack is a stack \(\X\) presentable by an internal \(n\)-groupoid on the site of derived manifolds. 
	\end{informal_definition}
	More concretely, we give the following standard inductive definition (see, for example, \cite[\S I.2.4]{Gaitsgory_Rozenblyum_book_1}).
	\begin{definition}\label{def:n-geometric-stacks}
		A \(0\)-geometric stack is a coproduct of derived manifolds. A morphism of \(0\)-geometric stacks is \(0\)-smooth if it is an ordinary smooth morphism of derived manifolds. An \(n\)-geometric stack is a derived \(\CI\)-stack \(\Xc\) possessing a morphism from derived \((n-1)\)-geometric stack \(\Uc\to \Xc\) which is \((n-1)\)-smooth, surjective and representable by an \((n-1)\)-geometric stack.
		A morphism of \(n\)-geometric stacks is \(n\)-smooth if it is representable by \((n-1)\)-smooth morphism of \((n-1)\)-geometric stacks.
	\end{definition}
	\begin{example}
		The basic examples of \(n\)-geometric stacks are given by derived Deligne--Mumford stacks. These are defined as \(n\)-geometric stacks such that the morphisms in the definition of the \(n\)-geometric stack are \'{e}tale. These, for instance, include quotient stacks of derived manifolds with respect to discrete Lie group actions.
	\end{example}
	\begin{definition}\label{def:ala_for_geometric_stacks}
		A \(0\)-geometric stack is said to be \ldRa if it is a coproduct of \ldRa derived manifolds. An \(n\)-geometric stack is said to be \ldRa if it has a presentation by \((n-1)\)-geometric stack which is \ldRa.
	\end{definition}
	\begin{proposition}\label{prop:ala_is_presentation_independent}
		Let \(\X\) be an \ldRa \(n\)-geometric derived differentiable stack. Assume that \(U\to \X\) is a presentation witnessing the \ldRa condition. Let \(U'\to \X\), then it is also \ldRa. That is, the \ldRa condition is independent of presentation.
	\end{proposition}
	\begin{proof}
		Consider the case of \(n=1\). Let \(*\xrightarrow{p} U'\) be a point of \(U'\). Consider a fibre product \(U\times_\X U'\), then it is smooth over \(U'\). Consequently, by \Cref{implicit_function_theorem} the fibre over \(p\) is equivalent to \(\R^n\) for some \(n\). Consequently, the stalk of derived de Rham cohomology at the point \(p\) is trivial. The case of general \(n\) follows by induction. 
	\end{proof}
	\subsection{de Rham theorem for \texorpdfstring{\(\CI\)}{C-infinity}-de Rham cohomology}
	\begin{construction}\label{con:CI-derived-de-Rham}
		Given a derived (affine) manifold \(X\), we define \(\CI\)-derived de Rham cohomology sheaf on the topological space \(X(\R)\) as follows. Given an open subset \(U\sse X(\R)\) we define the \(\CI\)-structure on \(U\) by restricting the structure sheaf on \(X\) to \(U\) as in \Cref{ex:open_subsets_of_derived_manifolds}. That is, we invert a characteristic function of \(U\) (any characteristic function will do). Then we define the sheaf of derived \(\CI\)-algebras as follows.
		\[
			U\mapsto \CI(U^{\dRcinf}).
		\]
		We call this sheaf the \(\CI\)-derived de Rham cohomology sheaf of \(X\).
		If \(\X\) is a derived \(\CI\)-stack, we define the \(\CI\)-derived de Rham cohomology of \(\X\) via the right Kan extension from the category of derived manifolds.
	\end{construction}
	\begin{proposition}\label{st:open_means_open_on_deRham}
		Let \(U\to V\) an open embedding the induced morphism of \(\Com\) and \(\CI\)-de Rham stacks is an open embedding.
	\end{proposition}
	
	\begin{proof}
		First, we treat the case of the \(\CI\)-de Rham stack. 
		Consider the following diagram. 
		\[\begin{tikzcd}
			T \\
			& X && S \\
			\\
			& {U^{\dRcinf}} && {V^{\dRcinf}}
			\arrow[dashed, from=1-1, to=2-2]
			\arrow[curve={height=-12pt}, from=1-1, to=2-4]
			\arrow[curve={height=12pt}, from=1-1, to=4-2]
			\arrow[from=2-2, to=2-4]
			\arrow[from=2-2, to=4-2]
			\arrow["\lrcorner"{anchor=center, pos=0.125}, draw=none, from=2-2, to=4-4]
			\arrow[from=2-4, to=4-4]
			\arrow[from=4-2, to=4-4]
		\end{tikzcd}\]
		Let \(R\) be the algebra of functions on \(S\). Denote by \(A\) the algebra of functions on \(V\). Then the algebra of functions on \(U\) is given by a \(\CI\)-localization in a single element \(\chi\). Denote by \(\vp:A\to R^{\redcinf}\). Then one can take \(X\) to be the spectrum of the algebra \(R\) localized in the element \(\vp(\chi)\) lifted to \(R\) along the projection \(R\to R^{\redcinf}.\) It is clear that it verifies the requisite universal property and is independent of the choice of the lift.

		The case of the \(\Com\)-de Rham stack is treated similarly.
	\end{proof}
	\begin{theorem}\label{CI-infinity-de-Rham=const-sheaf}
		Let \(X\) be a derived manifold. Then, the \(\CI\)-derived de Rham cohomology of \(X\) (as a sheaf on \(X(\R)\)) is equivalent to the cohomology of the constant sheaf on the underlying topological space \(X(\R)\).
	\end{theorem}
	\begin{proof}
		The statement is local in nature since the underlying topological space of \(X\) has finite covering dimension by \Cref{st:hypercompleteness_of_dMfld}.
		Thus, it is enough to verify the following statement on stalks. 
		\[
			\CI(X^{\dRcinf})_x:=\colim_{x\in U,\\ U\underset{\clap{\scriptsize open}}{\sse} X} \CI(U^{\dRcinf})\simeq \R.
		\]
		To see this, we first observe that the requisite equivalence holds on the level of the de Rham stacks. Namely, we assert the following equivalence of derived \(\CI\)-stacks.
		\[
			\lim_{x\in U,\\ U\underset{\clap{\scriptsize open}}{\sse} X} U^{\dRcinf}\simeq (\Spec \R)^{\dRcinf}
		\]
		Indeed, consider the following limit on functors of points.
		\begin{align*}
			\left(\lim_{\substack{x\in U\\ U \subseteq X \text{ open}}} U^{\dRcinf}\right)(R) 
			&= \lim_{\substack{x\in U\\ U \subseteq X \text{ open}}} \bigl(U^{\dRcinf}(R)\bigr) \\
			&= \lim_{\substack{x\in U\\ U \subseteq X \text{ open}}} \bigl(\Hom_{\CI\Alg}(\CI(U),R^{\redcinf})\bigr) \\
			&= \Hom_{\CI\Alg}\left(\colim_{\substack{x\in U\\ U \subseteq X \text{ open}}} \CI(U), R^{\redcinf}\right)
		\end{align*}
		By \Cref{st:open_means_open_on_deRham}, the stalk of the \(\CI\)-de Rham stack at a point is equivalent to the \(\CI\)-de Rham stack of the local ring since limits of affine morphisms are affine. This, however, is equivalent to the \(\CI\)-de Rham stack of \(\Spec\R\), since the \(\RJ\)-reduction (see \Cref{def:RJ}) of any \(\CI\)-local ring with the residue field \(\R\) is just \(\R\). Moreover, by \Cref{cor:CI-red=RJ-red} \(R^{\red_{\CI}}=R^{\red_{\RJ}}\) and the result follows by the universal propery of the \(\RJ\)-reduction.
	\end{proof}
	The following result is an immediate consequence of \Cref{CI-infinity-de-Rham=const-sheaf} by applying the right Kan extension.
	\begin{corollary}\label{st:general_de_Rham_iso}
		Let \(\X\) be a derived \(\CI\)-stack. Then, the derived \(\CI\)-de Rham cohomology of the stack \(\X\) from \Cref{con:CI-derived-de-Rham} is equivalent to the \(\R\) cohomology of the shape of \(\X\). Or equivalently, to the constant sheaf cohomology of \(\R\) in the category of Betti sheaves on \(\X\). 
	\end{corollary}
	\subsection{de Rham theorem for derived \texorpdfstring{\(\Com\)}{Com}-cohomology}
	\begin{theorem}\label{prop:Com_de_Rham_of_stacks}
		Let \(\X\) be a derived \(\CI\)-stack. Then, the derived \(\Com\)-de Rham cohomology of the stack \(\X\) from \Cref{con:derived_de_Rham_of_stacks} is a sheaf of complete filtered commutative algebras on the shape of \(\X\). There is a canonical morphism of sheaves as follows.
		\begin{equation}\label{eq:const_to_com_de_Rham}
			\ul{\R}\to \cdR_{\X}.
		\end{equation}
		Moreover, morphism \eqref{eq:de_Rham_fibre_sequence} of de Rham stacks induces morphism \eqref{eq:const_to_com_de_Rham} of sheaves by applying the functor~\(\CI(-)\). 
	\end{theorem}
	\begin{proof}
		The result follows from \Cref{thm:de_Rham_vs_functions_on_inf} and \Cref{CI-infinity-de-Rham=const-sheaf} because the left hand side is identified with coherent cohomology on the \(\CI\)-de Rham stack and the right-hand side is identified with coherent cohomology on the \(\Com\)-de Rham stack. 
	\end{proof}
	\begin{corollary}\label{st:Hodge_filtration_on_the_constant_sheaf}
		There is a canonical  (Hodge) filtration on the cohomology of the constant sheaf on the shape of a derived \(\CI\)-stack. This filtration is induced by the morphism \eqref{eq:de_Rham_fibre_sequence}.
	\end{corollary}
	\begin{proof}
		The pullback of the Hodge filtration on \(\Com\)-de Rham cohomology of a derived \(\CI\)-stack gives the filtration.
	\end{proof}
	\begin{remark}
		In the case when the derived \(\CI\)-stack in \Cref{st:Hodge_filtration_on_the_constant_sheaf} is a smooth manifold, the filtration in question is the standard Hodge filtration (i.e. the filtration by degree of differential forms) on the singular cohomology of the manifold.
	\end{remark}
	\begin{proposition}
		If \(X\) is a derived manifold such that it is \ldRa (\Cref{def:ldRa}), then the derived \(\Com\)-de Rham cohomology of \(X\) is equivalent to the cohomology of the constant sheaf on the underlying topological space of \(X\).
	\end{proposition}
	\begin{proof}
		This follows from \Cref{st:local_to_global_de_Rham}.
	\end{proof}
	\begin{theorem}\label{thm:de_Rham_for_geometric}
		Let \(\X\) be an \ldRa \(n\)-geometric derived differentiable stack. Then, the derived \(\Com\)-de Rham cohomology of \(\X\) is equivalent to the cohomology of the constant sheaf on the shape of \(\X\).
	\end{theorem}
	\begin{proof}
		We prove this result by induction on \(n\). The case of \(n=0\) follows from the assumption by \Cref{st:local_to_global_de_Rham}. The case of \(n=1\) follows from the following easy observation. Consider a \(1\)-geometric stack \(\X\). Then for any morphism \(M\to \X\) the pullback of the \ldRa derived manifold \(\Uc\) covering \(\Xc\) gives a smooth surjective morphism \(\Uc\times_\Xc M\to M\). By \Cref{implicit_function_theorem}, this map is locally equivalent to the canonical projection \(\R^n\times M\to M\). This shows both that the \(\Com\)-de Rham cohomology and constant sheaf cohomology satisfy descent along the map \(\Uc\to \Xc\) and thus are equivalent.
		
		Now, assume that the result holds for \(n-1\). Then, by the definition of an \(n\)-geometric stack we have a smooth surjective morphism \(\Uc\to \X\) from an \((n-1)\)-geometric stack. By the induction hypothesis, the derived \(\Com\)-de Rham cohomology of \(\Uc\) is equivalent to the cohomology of the constant sheaf on the shape of \(\Uc\). We want to verify this result for stalks of \(\cdR_\X\). 

		To see this, consider a base change of the morphism \(\pi:\Uc\to \X\) to a point \(x:*\to \X\). Then we have a morphism \(\Uc\times_\Xc *\to *\) by definition \(\Uc\times_\Xc *\) is an \((n-1)\)-geometric stack, thus one can use the inductive assumption. 
	\end{proof}
	\bookmarksetup{startatroot}
	\appendix
	\section{A Fermatic generalization}\label{sec:Fermatic}
	\setcounter{subsection}{-1}
	\subsection{Introduction to the appendix}
	\subsubsection{Why this section exists?}
	The main goal of this section is to show that most of the results in the previous sections easily generalize to a more general set-up of Fermat theories. That is, one can perform the same constructions we did in derived differential geometry in the context of derived holomorphic, algebraic or real analytic geometry.

	We also attempt to conceptualize the reason why the results of the previous sections hold in the first place. To that end we use the language of ring stacks due to Drinfeld \cite{Drinfeld_Prism} and show that the constructions of \(\Com\) and \(\CI\) de Rham stacks can be interpreted in terms of certain operations on the ring stack \(\AA^1\) over various Fermat theories. In particular, for the theory of holomorphic functions \(\O\) we obtain a version of the de Rham stack similar to the one given in \cite{Camargo} and similarly a version of the analytic Riemann--Hilbert correspondence compatible with the \(\CI\)-version we developed in the main text.
	\subsubsection{What is done in this section?}
	Sections \ref{sec:Fermat_theories_def}, \ref{sec:TAlg_localizations}, \ref{sec:Completions_of_derived_stacks}, and \ref{sec:TAlg_Cotangent_Complex} contain virtually no new mathematics and are just a review of the basic notions of geometry and derived geometry over Fermat theories. In \S \ref{sec:TAlg_De_Rham_And_Inf} we discuss how one presents de Rham stacks of stacks over Fermat theories via infinitesimal groupoids. In \S \ref{sec:TAlg_de_rham_thms} we formulate appropriate analogues of the results of \S \ref{sec:de_Rham_theorems} in the setting of Fermat theories. Finally, \S \ref{sec:TAlg_functoriality_of_de_Rham} discusses the general functoriality properties of de Rham stack and some concluding remarks on possible further generalizations. 
	\subsection{Geometry over Fermat theories}\label{sec:Fermat_theories_def}
	In this section, we review the ideas of \enquote{neutral geometry} based on the concept of a Fermat theory. The idea of using this formalism originates in the work of Dubuc--Kock \cite{dubuc19841}. Since then many expositions followed the general philosophy of this approach. See \cite{Porta_Yu}, \cite{Ben_Bassat_Kremnizer} for instances of this; this list, of course, is by no means exhaustive. Instead, we would mention only the most recent work done purely on the subject~\cite{carchedi2012theories,Carchedi_Roytenberg} and \cite{Pridham_DG}.

	The idea is to take as input a specific kind of functions as a foundation and then develops the geometry based on this notion of a function. In particular, using this approach, one obtains derived differential geometry, derived complex geometry, and derived algebraic geometry.
	\subsubsection{\texorpdfstring{\(\Tc\)}{T}-algebras}\label{sec:TAlg}
	\begin{definition}\label{def:RCom}
		For a commutative ring \(R\) define a Lawvere theory \(\Com_R\) by specifying
		\[
		\Com_R(R^m;R^n)=\{ \text{the set of } R\text{-polynomial maps } \AA^m_R\to \AA^n_R\}.
		\]
		Clearly, this is just the theory of commutative \(R\)-algebras.
	\end{definition}
	\begin{definition}[{\cite[\S 1]{dubuc19841}}]\label{def:fermat}
		Consider a Lawvere theory \(\Tc\) together with a morphism of Lawvere theories
		\[
		\iota:\Com_\Z\to \Tc.
		\]
		It is a {\it Fermat theory} if additionally for arbitrary objects \(S,T\in \Tc\) and any morphism 
		\[
		f:\iota(\AA^1_\Z)\times S\to T
		\]
		there exists \(g:\iota(\AA^1_\Z)\times \iota(\AA^1_\Z)\times S\to T\) such that 
		\[
		f(l_1,s)-f(l_2,s)=(l_1-l_2)\cdot g(l_1,l_2,s).
		\]
		Here \(l_1,l_2:\iota\AA^2_\Z\to\iota\AA^1_\Z\) are the images of two projections and the ring operations are the images of corresponding operations in \(\Com_\Z.\) The latter property is often called \emph{Hadamard's lemma} in reference to the result establishing this property for smooth functions on $\R^n.$
	\end{definition}
	\begin{definition}\label{def:Fermat_over_R}
		We say that \(\Tc\) is a Fermat theory over a ring \(R\) if the morphism \(\iota\) from Definition \ref{def:fermat} factors as 
		\begin{equation}\label{eq:can_morphism_from_com}
			\begin{tikzcd}
			{\Com_\Z} && \Tc \\
			& {\Com_R}
			\arrow["\iota", from=1-1, to=1-3]
			\arrow["{(\Z\to R)_*}"', from=1-1, to=2-2]
			\arrow["{\ol{\iota}}"', from=2-2, to=1-3]
		\end{tikzcd}\end{equation}
		Here, the left morphism is induced by the canonical ring homomorphism \[
		\Z\to R,\quad 1_\Z\mapsto 1_R.
		\]
	\end{definition}
	\begin{definition}
		Let \(\Tc\) be a Fermat theory.
		A \(\Tc\)-algebra is a finite product preserving functor $\Tc\to\Spc.$ 
	\end{definition}
	\begin{definition}
		A homomorphism of \(\Tc\)-algebras (or simply \(\Tc\)-homomorphism) is a natural transformation of such functors.
	\end{definition}
	
	\begin{construction}[{see \cite[\S 2.2.1]{carchedi2012theories}}]\label{con:completion}
		A morphism of Fermat theories \(\Tc\to \Tc'\) induces an adjoint pair of functors:
		\[\begin{tikzcd}
			{\Tc\Alg} && {\Tc'\Alg}
			\arrow[""{name=0, anchor=center, inner sep=0}, "{\free^{\Tc'}_{\Tc}}", shift left=3, from=1-1, to=1-3]
			\arrow[""{name=1, anchor=center, inner sep=0}, "{\oblv_{\Tc}^{\Tc'}}", shift left=3, from=1-3, to=1-1]
			\arrow["\dashv"{anchor=center, rotate=-90}, draw=none, from=0, to=1]
		\end{tikzcd}\]
		The left adjoint is called {\it \(\Tc'\)-base change} and the right one is the forgetful functor.
	\end{construction}
	\begin{remark}
		The following example justifies the name base change in \Cref{con:completion}. Consider a morphism of Fermat theories \(\Com_R\to \Com_S\) induced by a morphism of rings \(R\to S\), then the corresponding base change functor is just the extension of scalars functor \(S\otimes_R-.\)
	\end{remark}
	In particular any commutative algebra over \(K\) admits a base change to an algebra over a Fermat theory \(\Tc\) over \(K.\)
	\begin{notation}
		One can recover a \(\Tc\)-algebra \(A\) from the set \(A(1)\) and the action of \(\Tc(n)\) on this set. Thus, in the sequel we will denote the set \(A(1)\) simply by \(A.\) 
	\end{notation}
	There is a forgetful functor from \(\Tc\)-algebras to ordinary algebras. Thus, any \(\Tc\)-algebra is an ordinary algebra.
	\begin{proposition}[{\cite[Proposition 1.2]{dubuc19841}}]
		For any ring-theoretic ideal \(I\) in a discrete \(\Tc\)-algebra~\(A\) there is a natural structure of a \(\Tc\)-algebra on \(A/I\) that makes \(A\to A/I\) a \(\Tc\)-homomorphism.
	\end{proposition}
	\begin{definition}\label{def:fg_algebra}
		A \(\Tc\)-algebra is {\it finitely generated} if it is a quotient of the algebra \(y\Tc(n)\) for some \(n\ge 0.\) 
	\end{definition}
	\subsubsection{Examples}\label{sec:Fermat_examples} We list only the three most well-known instances of this approach; other examples can be found in Carchedi--Roytenberg \cite[\S 2.2.3]{carchedi2012theories}. The key reason why Fermat theories are so useful is that they allow a uniform construction of complicated geometric objects (e.g., higher sheaves) from ``affine models'', for more detail see Carchedi--Roytenberg \cite[Introduction]{carchedi2012theories}. 
	\begin{example}\label{ex:ag}
		The simplest example of a Fermat theory that we already encountered is that of \(\Com_R\) (see Definition \ref{def:RCom}). This is the theory of commutative \(R\)-algebras for a commutative ring \(R \). In particular, the theory \(\Com_\Z\) plays the role of the initial Fermat theory. Geometry over \(\Com_\Z\) reproduces the standard scheme-theoretic algebraic geometry.
	\end{example}
	\begin{example}\label{ex:cinf}
		Another example we discussed in the paper proper is provided by the Fermat theory \(\CI\) defined via the following formula.
		\[
		\CI(m,n)=\CI(\R^m,\R^n).
		\]
		The geometry over this theory reproduces derived differential geometry.
	\end{example}
	\begin{example}\label{ex:analytic}
		The theory of analytic functions \(\mc{C}^\omega\) is defined via the following formula. 
		\[
		\Cc^\omega(m,n)=\Cc^\omega(\R^m,\R^n).
		\]
		That is, the morphisms are real analytic functions between \(\R^m\) and \(\R^n\). The geometry over this theory reproduces derived real analytic geometry.
	\end{example}
	\begin{example}\label{ex:hol}
		The theory of holomorphic functions \(\O\) is defined via the following formula.
		\[
		\O(m,n)=\O(\C^m,\C^n).
		\]
		The geometry over this theory reproduces derived complex geometry.
	\end{example}
	Now, we will give two ways to define a coverage on a small category of algebras over a Fermat theory.
	\begin{construction}
		Define the set \(\Specm A\) for a \(\Tc\)-algebra \(A\) as
		\[
		\Specm A=\Nat(A,\AA^1_\Tc).
		\]
		The set \(\Specm A\) (called the {\it maximal spectrum} of \(A\)) can be naturally topologized using the {\it Zariski topology}, i.e., we declare the generating collections of open sets for this topology to be
		\[
		U_a=\{p:\pi_0(A)\to \AA^1_\Tc\mid p(a)\neq 0\},\quad \text{for } a\in \pi_0(A(1))
		\]
	\end{construction}
	\begin{remark}
		One could say that the term maximal spectrum is somewhat misleading since by construction we only consider ideals with the same residue ring. I.e. there are no field extensions in the spectrum, and the ideals might not even be maximal.
	\end{remark}
	\begin{construction}\label{con:specm_topology}
		We now define a site structure on \(\Tc\Alg^{\op}.\)
		A set of morphisms \(\{f_i:B_i\to A\vert i \in I\}\) is a {\it covering family} for a \(\Tc\)-algebra \(A\) if it is dual to an open cover in Zariski topology:
		\[
		\{\Specm(B_i)\to \Specm(A)\mid i \in I\}.
		\]
		We assume in addition that the corresponding map of algebras are \emph{strong} in the sense of Lurie \cite[Definition 2.2.2.1]{HAGII}.
	\end{construction}
	\begin{remark}
		Construction \ref{con:specm_topology} is relevant for Examples \ref{ex:cinf}, \ref{ex:hol}, and less so for Example~\ref{ex:ag}. 
	\end{remark}
	Another construction is more convenient for algebro-geometric purposes.
	\begin{definition}
		A ring-theoretic ideal \(I\triangleleft \pi_0 A\) in an algebra \(A\) over a Fermat theory \(\Tc\) is called {\it radical} if \(I\) coincides with its radical defined as
		\[
		\sqrt{I}:=\{a\in A\mid \exists n\in \Z_{\ge 0},\; a^n\in I\}.
		\]
	\end{definition}
	There is a Galois-type correspondence between radical ideals in the algebra \(A\) and {\it reduced} quotients of \(A\) (i.e., quotient \(\Tc\)-algebras of \(A\) without nilpotent elements).
	\begin{construction}
		Denote by \(\Zar(A)\) the poset of radical ideals in the algebra \(A.\)
	\end{construction}
	\begin{lemma}[{Tierney \cite[Proposition 3.1]{tierney1976spectrum}}]
		A poset \(\Zar(A)\) of radical ideals for an algebra \(A\) over a Fermat theory \(\Tc\) is a frame in the sense of MacLane--Moerdijk \mbox{\cite[\S IX.1]{maclane2012sheaves}}. The construction \(\Zar(A)\) can be promoted to a functor
		\[
		\Spec:=\Zar^{\op}:\Tc\Alg^{\op}\to \Loc.
		\] 
	\end{lemma}
	\begin{construction}\label{con:spec_topology}
		We now define another site structure on \(\Tc\Alg^{\op}.\)
		A set of morphisms \(\{f_i:B_i\to A\vert i \in I\}\) is a {\it covering family} for a \(\Tc\)-algebra \(A\) if it is dual to an open cover of locales (in the sense of MacLane--Moerdijk \cite[\S IX]{maclane2012sheaves}):
		\[
		\{\Spec(B_i)\to \Spec(A)\mid i \in I\}
		\]
		and in addition the corresponding morphisms of derived algebras are strong.
	\end{construction}
	\begin{remark}
		Construction \ref{con:spec_topology} is well-adapted for algebraic geometry (Example \ref{ex:ag}). It is, however, not convenient for Examples \ref{ex:cinf}, \ref{ex:hol} since locales produced by the functor \(\Spec\) are always compact and thus do not model smooth or holomorphic affine spaces adequately. 
	\end{remark}
	\subsubsection{Derived stacks over Fermat theories}
	\begin{construction}
		Given a (small) site of \(\Tc\)-algebras over some Fermat theory \(\Tc\) with Grothendieck topology which is induced from the homotopy category of \(\Tc\)-algebras we can define the category of derived stacks over \(\Tc\) in the same way as we did for \(\CI\)-algebras. Namely, we define them as sheaves on this site of formal duals to finitely presented \(\Tc\)-algebras.
	\end{construction}
	\subsection{Localizations and reductions of \texorpdfstring{\(\Tc\)}{T}-algebras}\label{sec:TAlg_localizations}
	In this section, we essentially repeat the constructions of \Cref{sec:reductions} for general \(\Tc\)-algebras instead of \(\CI\)-algebras. Since not much changes compared to the case we already discussed, we only sketch the main ideas.
	\begin{definition}[{\(\Tc\)-localization}]\label{def:T-localization}
		Let \(A\) be a derived \(\Tc\)-algebra. Consider an element \(s\in\pi_0A\). Then the \(\Tc\)-localization of \(A\) in \(s\) is the derived \(\Tc\) satisfying the classical universal property of localization \emph{in the category of derived \(\Tc\)-algebras}. 
	\end{definition}
	\begin{example}
		If \(A\) is a finitely generated \(\O\)-algebra, its localization can be calculated similarly to the case of \(\CI\)-algebras. That is one pick a lift of the element \(s\) to the free algebra and then consider holomorphic functions on the complement of the zero locus of this lift modulo the localization of the ideal defining \(I\). The result of this operation is obviously different from the usual ring-theoretic localization.
	\end{example}
	\begin{remark}
		In the case when \(\Tc=\Com\) the \(\Tc\)-localization of a \(\Tc\)-algebra \(A\) in an element \(s\in A\) is the usual algebraic localization of \(A\) in \(s\).
	\end{remark}
	We now define \(\Tc\)-nilpotent elements in derived \(\Tc\)-algebras. The definition is essentially the same as the classical one, with the only difference being that we consider the elements in \(\pi_0 A\) for some derived \(\Tc\)-algebra \(A\).
	\begin{definition}[{Nilpotent elements in \(\Tc\)-algebras}]\label{def:T-nilpotent_elements} Let \(A\) be a derived \(\Tc\)-algebra. We have two natural notions of nilpotent elements in \(A\).
		\begin{itemize}
			\item[\(\Tc\)] An element \(s\in \pi_0A\) is called \emph{\(\Tc\)-basic nilpotent} if the \(\Tc\)-localization of \(A\) in \(s\) is trivial.
			\item[\(\Com\)] An element \(s\in \pi_0A\) is called \emph{\(\Com\)-nilpotent} if the algebraic localization of \(A\) in \(s\) is trivial. This is, of course, the usual notion of nilpotent elements.
		\end{itemize}
	\end{definition}
	\begin{definition}[Support of a function]\label{T-definition_of_support}
		Let \(f\in \Tc(\AA_\Tc^n)\) be a \(\Tc\)-function. We define the \emph{carrier set} of \(f\) as the open subset in the space \(\Specm\Tc(\AA_\Tc^n)\) where \(f\) is non-zero. We denote it by \(\Carr(f)\). Accordingly, the \emph{support set} of \(f\) is the closure of \(\Carr(f)\) in \(\Specm\Tc(\AA_\Tc^n)\) and is denoted by \(\Supp(f)\).
	\end{definition}
	\begin{remark}
		In terms of \Cref{T-definition_of_support} the \(\Tc\)-basic nilpotent elements in a derived \(\Tc\)-algebra \(A\) are precisely the elements \(s\in \pi_0A\) such that the support of their lifts to \(\Tc(\R^n)\) doesn't intersect the classical locus of \(A\). 
	\end{remark}
	One can give a more general version of \Cref{def:nilpotent_elements} for an arbitrary ideal \(I\) in \(A\) and not just the zero ideal. There are, in fact, many other natural notions of radicals of ideals in \(\Tc\)-algebras. These are listed and explored in \cite[\S 2]{Borisov_Kremnizer} for \(\Tc=\CI.\)
	\begin{definition}[{Radicals of ideals in \(\Tc\)-algebras}]\label{def:T-radicals_of_ideals}
		Let \(A\) be a derived \(\Tc\)-algebra. Let \(I\) be an ideal in \(\pi_0 A\). We have two natural notions of the radical of \(I\).
		\begin{itemize}
			\item[\(\Tc_0\JT\)] An element \(s\in \pi_0A\) belongs to \(\sqrt[\Tc_0\JT]{I}\) if it lies in the kernel of all homomorphisms of \(\Tc\)-algebras \(A\to \Tc_0\).
			\item[\(\Tc\)] The ideal \(\sqrt[\Tc]{A}\) is generated by all \(\Tc\)-basic nilpotent elements. 
			\item[\(\Com\)] An element \(s\in \pi_0A\) belongs to \(\sqrt[\Com]{I}\) if the algebraic localization of \(A/I\) in \(s\) is trivial. This is, of course, the usual radical of an ideal.
		\end{itemize}
	\end{definition}
	\begin{remark}
		Note that the \(\Tc\)-radical apriori differs from the set of elements which \(\Tc\)-localize the given algebra to \(0.\)
	\end{remark}
	\begin{definition}
		Let \(A\) be a derived \(\Tc\)-algebra. We have two natural notions of it being reduced.
		\begin{itemize}
			\item[\(\Tc_0\JT\)] A derived \(\Tc\)-algebra \(A\) is called \emph{\(\Tc_0\JT\)-reduced} if the \(\Tc_0\JT\)-radical of the zero ideal in \(A\) is trivial.
			\item[\(\Tc\)] A derived \(\Tc\)-algebra \(A\) is called \emph{\(\Tc\)-reduced} if the \(\Tc\)-localization of \(A\) in any non-zero element is non-trivial.
			\item[\(\Com\)] A derived \(\Tc\)-algebra \(A\) is called \emph{\(\Com\)-reduced} if the algebraic localization of \(A\) in any non-zero element is non-trivial. This is, of course, the usual notion of reducedness.
		\end{itemize}
	\end{definition}
	We can now define three reductions of a derived \(\Tc\)-algebra.
	\begin{construction}[{Reductions of \(\Tc\)-algebras}]\label{con:reductions_of_T_algebras}
		Let \(A\) be a derived \(\Tc\)-algebra. We have two natural reductions of \(A\).
		\begin{itemize}
			\item[\(\Tc_0\JT\)] \(A^{\red_{\Tc_0\JT}}\) is the derived \(\Tc\)-algebra given by the homotopy quotient and truncation \(\tau^{\ge 0}A/\sqrt[\Tc_0\JT]{0}\).
			\item[\(\Tc\)] \(A^{\red_{\Tc}}\) is the derived \(\Tc\)-algebra given by the homotopy quotient and truncation \(\tau^{\ge 0}A/\sqrt[\Tc]{0}\).
			\item[\(\Com\)] \(A^{\red_{\Com}}\) is the derived \(\Tc\)-algebra given by the homotopy quotient and truncation \(\tau^{\ge 0}A/\sqrt[\Com]{0}\).
		\end{itemize}
	\end{construction}
	\begin{proposition}
		Let \(A\) be a derived \(\Tc\)-algebra. Then, the \(\Tc\)-reduction of \(A\) is the universal \(\Tc\)-reduced derived \(\Tc\)-algebra equipped with a map from \(A\).
	\end{proposition}
	\begin{remark}
		A version of \Cref{st:CI-inf_radical=RJ-radical} does not hold in general Fermat theories since its proof used the notion of positivity for \(\R\) in an essential way. However, if one has that, e.g. if \(\Tc_0\) is an ordered field, one can prove a version of this result similarly.
	\end{remark}
	\subsection{Completions of derived stacks over Fermat theories}\label{sec:Completions_of_derived_stacks}
	As we have seen in \Cref{con:reductions_of_T_algebras}, there are three natural reductions of a derived \(\Tc\)-algebra. Now we define the completions of derived \(\Tc\)-stacks based on these reductions. The conceptual framework of this section again goes back to \cite{Simpson_Teleman} and \cite{Gaitsgory_Rozenblyum_Ind}.
	\begin{definition}\label{def:T_formal_completion}
		Consider a reflexive subcategory with reflector \(\Rc\) in the category of derived finitely presented \(\Tc\)-algebras. We define the \emph{derived \(\Tc\)-completion} of a morphism of derived stacks \(\Xc\to \Yc\) as follows.
		\[
			\Xc^{\wedge_\Rc}_{\Yc}(A):=\Yc(\Rc(A))\times_{\Xc(\Rc(A))} \Xc(A).
		\] 
		In particular, for \(\Rc=(-)^{\red_{\Tc_0\JT, \Tc, \Com}}\) we get \(\T_0\JT\), \(\Tc\), and \(\Com\) -completions respectively.
	\end{definition}
	\subsection{Cotangent complex and derived de Rham cohomology}\label{sec:TAlg_Cotangent_Complex}
	\subsubsection{Cotangent complex and its properties}
	\begin{definition}
		For \(\Tc\)-algebra, the tangent category is the following presentable fibration 
		\[
		p\colon \Mod \to \Tc\Alg.
		\]
		For a derived \(\Tc\)-algebra \(A\), the \emph{cotangent complex} 
		\(
		\LL_A\in \Mod_A \simeq T\bigl(\Tc\Alg\bigr)\times_{\Tc\Alg}\{A\}
		\)
		is defined as the value of the cotangent complex functor at \(A\). 

		For a morphism \(f\colon A\to B\) of derived \(\Tc\)-algebras, the relative cotangent complex 
		\(\LL_f\in \Mod_B\) often denoted \(\LL_{B/A}\)	is defined as the value of the relative cotangent complex functor at \(f\).
	\end{definition}
	\begin{remark}
		Once again as in \Cref{def:CI-derivations}, by picking a particular model for derived \(\Tc\)-algebras one can interpret the cotangent complex in terms of the functor of \(\Tc\)-derivations. These are defined similarly as \(\CI\)-derivations by considering the Leibniz rule for some extension of polynomial operations. 
	\end{remark}
	The following example shows how the cotangent complex can be computed for a finitely presented \(\Tc\)-algebra.
	\begin{example}[{\cite[Example 2.2.19]{Nuiten}}]
		Consider the \(\Tc\)-algebra \(\Tc(X)\) of the derived fibre of the map \((f_1,\ldots,f_m):\AA^n_\Tc\to \AA^m_{\Tc}\). Then the ring \(\Tc(X)\) can be presented by the following dg \(\Tc\)-algebra 
		\[
		\Tc(X)=\Tc(\AA^n_\Tc)[\xi_1,\ldots,\xi_m]
		\] 
		Here \(\xi_i\) are odd variables of degree \(-1\). The differential \(\pd\) is given by \(\pd(\xi_i)=f_i\) for \(i\) from \(1\) to \(m\). The cotangent complex of \(X\) is then given by a free \(\Tc(X)\)-module 
		\[
			\Omega^1(\AA^n_\Tc)[\xi_1,\ldots,\xi_m]\oplus\Tc(X)\la d_{\dR}\xi_i\ra_{i=1}^m,\quad \pd d_{\dR}\xi_i=d_{\dR}f_i\in \Omega^1(\AA^n_\Tc).
		\]
	\end{example}
	The following result shows that the cotangent complex of an open embedding is trivial.
	\begin{proposition}[{\cite[Corollary 5.1.0.14]{Steffens_thesis}}]\label{T_cotangent_complex_of_localization_vanishes}
		Let \(A\to A\{S^{-1}\}\) be a \(\Tc\)-localization of a derived \(\Tc\)-algebra. Then the cotangent complex of the map \(A\to A\{S^{-1}\}\) is trivial.
	\end{proposition}
	\begin{corollary}
		The cotangent complex \(\LL_{X/Y}\) forms a sheaf on the site of \(\Tc\)-schemes over \(Y\) with the topology of open embeddings given by \Cref{con:specm_topology}.
	\end{corollary}
	\subsubsection{Derived de Rham cohomology}
	\begin{construction}\label{con:T_derived_de_Rham}
		Let \(A\to B\) be a morphism of derived \(\Tc\)-algebras. Let \(F\to B\) be a free \(\Tc\)-resolution of \(B\) over \(A\). We define the \emph{Hodge completed derived de Rham cohomology} of \(B\) over \(A\) as the completion of \(\left|\Omega^\bt_{F/A}\right|\) for its Hodge filtration. Concretely, the complete filtered derived \(\Tc\) algebra \(\cdR_{B/A}\) is the limit of the following directed system.
		\[
			\cdR_{B/A}/\Fil_H^k=\Tot(\sigma^{\leq k}\Omega^\bt_{F/A}).
		\]
		Here \(\sigma^{\leq k}\) denotes the valuable truncation in cohomological degrees \(\leq k\). 
	\end{construction}
	Just as is the case for \(\CI\)-algebras for arbitrary \(\Tc\)-algebras, we can present the derived de Rham cohomology using the Adams completion. We note here that this result uses in an essential way that geometries over Fermat theories admit a good theory of closed immersions in the sense of \cite{Lurie_DAG9}. One way to verify that is to use the proof of \cite[Lemma 4.6]{Carchedi_Steffens} and observe that one could prove the same unramifiedness statement with an arbitrary Fermat theory in place of \(\CI.\)
	\begin{theorem}\label{thm:T_Adams_completion_vs_de_Rham}
		Let \(f:A\to B\) be a \(\pi_0\)-epimorphic morphism of derived \(\Tc\)-algebras. Then, we have a filtered equivalence of complete filtered derived \(\Tc\)-algebras.
		\[
			\left\{\cdR_{B/A},\Fil_H^\bt\right\}\simeq \Big\{\Comp_A(A,f),\Fil_H^\bt\Big\}
		\]
		Consequently, the Adams completion of \(f\) is equivalent to the derived de Rham cohomology of \(B\) over \(A\).
		\[
			\cdR_{B/A}\simeq \Comp_A(A,f).
		\]
	\end{theorem}
	\subsection{De Rham stacks and infinitesimal groupoids}\label{sec:TAlg_De_Rham_And_Inf}
	Now, similarly to how we defined de Rham stacks for the case of \(\Tc=\CI\), we can now define three versions of the de Rham stack for an arbitrary Fermat theory \(\Tc.\)
	\begin{construction}
		Let \(\X\) be a derived \(\Tc\)-stack. Then we have two natural \enquote{de Rham} stacks associated to it.
		\[
			\X^{\dR_{\Tc}}(A):=\X(A^{\red_\Tc}),\quad \X^{\dR_\Com}(A):=\X(A^{\red_\Com}).
		\]
	\end{construction}

	\begin{definition}[{Infinitesimal groupoids}]
		Let \(\X\) be a derived \(\Tc\)-stack. We define its \emph{\(\Tc_0\JT\)-infinitesimal groupoid} \(\X^{\inft_{\Tc_0\JT}}\) as follows.
\[\begin{tikzcd}
	\vdots \\
	{(\X^3)^{\wedge_{\Tc_0\JT}}_{\Delta_3}} \\
	{(\X^2)^{\wedge_{\Tc_0\JT}}_{\Delta_2}} \\
	\X
	\arrow[shift right=5, from=1-1, to=2-1]
	\arrow[shift left=4, from=1-1, to=2-1]
	\arrow[shift right=2, from=1-1, to=2-1]
	\arrow[shift left, from=1-1, to=2-1]
	\arrow[shift right=4, from=2-1, to=3-1]
	\arrow[shift left=4, from=2-1, to=3-1]
	\arrow[from=2-1, to=3-1]
	\arrow[shift right=3, from=3-1, to=4-1]
	\arrow[shift left=3, from=3-1, to=4-1]
\end{tikzcd}\]
		We define its \emph{\(\Tc\)-infinitesimal groupoid} \(\X^{\inft_{\Tc}}\) as follows.
\[\begin{tikzcd}
	\vdots \\
	{(\X^3)^{\wedge_{\Tc}}_{\Delta_3}} \\
	{(\X^2)^{\wedge_{\Tc}}_{\Delta_2}} \\
	\X
	\arrow[shift right=3, from=3-1, to=4-1]
	\arrow[shift left=3, from=3-1, to=4-1]
	\arrow[shift right=4, from=2-1, to=3-1]
	\arrow[shift left=4, from=2-1, to=3-1]
	\arrow[from=2-1, to=3-1]
	\arrow[shift right=5, from=1-1, to=2-1]
	\arrow[shift left=4, from=1-1, to=2-1]
	\arrow[shift right=2, from=1-1, to=2-1]
	\arrow[shift left, from=1-1, to=2-1]
\end{tikzcd}\]

		We define its \(\Com\)-\emph{infinitesimal groupoid} \(\X^{\inft_{\Com}}\) as follows.
\[\begin{tikzcd}
	\vdots \\
	{(\X^3)^{\wedge_{\Com}}_{\Delta_3}} \\
	{(\X^2)^{\wedge_{\Com}}_{\Delta_2}} \\
	\X
	\arrow[shift right=3, from=3-1, to=4-1]
	\arrow[shift left=3, from=3-1, to=4-1]
	\arrow[shift right=4, from=2-1, to=3-1]
	\arrow[shift left=4, from=2-1, to=3-1]
	\arrow[from=2-1, to=3-1]
	\arrow[shift right=5, from=1-1, to=2-1]
	\arrow[shift left=4, from=1-1, to=2-1]
	\arrow[shift right=2, from=1-1, to=2-1]
	\arrow[shift left, from=1-1, to=2-1]
\end{tikzcd}\]
	\end{definition}
	\begin{proposition}
		Under the assumption that \(\X\) is a formally smooth derived \(\Tc\)-stack, the derived \(\Com\)-de Rham stack of \(\X\) is equivalent to the quotient of the \(\Com\)-infinitesimal groupoid of \(\X\).
	\end{proposition}
	The proof of this result is identical to the proof of \Cref{st:de_Rham_stack_is_presented_by_the_infinitesimal_groupoid} and is thus omitted.
	\begin{lemma}
		Let \(X=\Spec(A)\) be an affine derived \(\Tc\)-scheme. Then, derived de Rham cohomology \(\cdR_X\) is equivalent to the ring of functions on the \(\Com\)-infinitesimal groupoid of \(X\).
	\end{lemma}
	The proof of this result is identical to the proof of \Cref{lem:de_Rham_vs_functions_on_inf}.
	\subsection{De Rham theorems}\label{sec:TAlg_de_rham_thms}
	\subsubsection{Shape of a derived stack over a Fermat theory}
	Similar to how we define the notion of shape in the context of \(\Tc=\CI\). We freely use the notation of \Cref{sec:preliminaries_on_stacks}. Namely, for a derived \(\Tc\)-stack \(\X\), we define its shape as follows.
	\begin{definition}
		Let \(X\) be an affine derived \(\Tc\)-scheme. Then its shape is defined as the composite \(q_*q^*\) for \(q_*\) the global sections functor for the \(\infty\)-topos of \(\infty\)-sheaves on the maximal spectrum topological space associated to \(X\). Then, for an arbitrary derived \(\Tc\)-stack, we define its shape using the following left Kan extension formula.
		\begin{align*}
			\Shape(\Xc)=\colim_{\Spec A\to \Xc}\Shape(\Spec A).
		\end{align*}
		Here the indexing category is the category of all \(\Tc\)-stack morphism from affine derived \(\Tc\)-schemes into \(\Xc\). 
	\end{definition}
	We make the following crucial assumption on the Fermat theory \(\Tc\).
	\begin{definition}
		A Fermat theory \(\Tc\) is called \emph{hypercomplete} if the \(\infty\)-sheaf topos on the maximum spectrum topological space of \(\AA^1_\Tc\) is hypercomplete.
	\end{definition}
	This condition in particular implies hypercompleteness of the sheaf topoi on all derived \(\Tc\)-schemes since the underlying maximum spectra embed as closed subspaces into that of some \(\AA^n_\Tc\). Thus, by \cite[Remark 4.11]{Volpe_Six}.
	From now on, we assume that \(\Tc\) is hypercomplete.
	\subsubsection{Local acyclicity}
	\begin{definition}
		We call an affine derived \(\Tc\)-scheme \emph{locally de Rham acyclic} (or \ldRa) if the stalk of the \(\Com\)-de Rham cohomology at each point of the associated maximal spectrum topological space is trivial.
	\end{definition}
	\begin{proposition}
		If the derived \(\Tc\)-scheme \(X\) is \ldRa then the \(\Com\)-de Rham cohomology of \(X\) is equivalent to \(\Tc_0J\)-de Rham cohomology of \(X\).
	\end{proposition}
	The proof of this result is identical to that of \Cref{st:local_to_global_de_Rham}.
	\begin{lemma}
		Consider an affine \(\Tc\)-scheme \(X\). Then if it is locally \(\AA^1_\Tc\)-contractible it is \ldRa. Consequently, the \(\Com\)-de Rham cohomology of \(X\) is equivalent to the \(\Tc_0\JT\)-de Rham cohomology of \(X\).
	\end{lemma}

	\begin{theorem}\label{thm:de_Rham_TJ_vs_locale_coh}
		Let \(\X\) be a derived \(\Tc\)-stack. Then if it is \(n\)-geometric with a presentation by \ldRa affine derived \(\Tc\)-schemes there is an equivalence between the constant sheaf cohomology of the associated shape of \(\Xc\) and the \(\Com\)-de Rham cohomology of \(\Xc\).
	\end{theorem}
	This result is proved identically to the case of \(\Tc=\CI\).
	\begin{theorem}
		Let \(\X\) be a derived \(\Tc\)-stack. Then there is an equivalence between the constant sheaf cohomology of the associated shape and the \(\Tc_0\JT\)-de Rham cohomology of \(\X\). 
	\end{theorem}
	This result is proved identically to the case of \(\Tc=\CI\), except that in the case of general \(\Tc\), we can not identify the \(\Tc_0\JT\) and \(\Tc\) de Rham stacks.

	\subsection{Functoriality of de Rham stacks}\label{sec:TAlg_functoriality_of_de_Rham}
	In this section, we explore the functorial behaviour of de Rham stacks with respect to the change of Fermat theory. In particular, we show how the two different kinds of the de Rham stack explored in the rest of the paper naturally arise as an instance of the change of Fermat theory construction.
	\subsubsection{Ring stacks and de Rham stacks}
	The idea of the following construction is due to Drienfeld, see e.g. \cite{Drinfeld_Prism}.
	\begin{definition}
		A prestack \(\X\) on a category \(\Cc\) is a \emph{ring prestack} if it takes values in the category of derived commutative rings. Similarly, for an arbitrary Fermat theory \(\Tc\), one says that \(\X\) is \(\Tc\)-algebra prestack if it takes values in derived \(\Tc\)-algebras.
	\end{definition}
	\begin{construction}[de Rham stacks from ring stacks]\label{con:de_Rham_from_ring_stacks}
		Consider the category of derived \(\Tc\)-prestacks. Then one views the \(\Tc_0\JT, \Tc, \Com\)-de Rham stacks of \(\AA^1_\Tc\) as a ring stack in the following natural way.
		\begin{align*}
			\left(\AA^1_\Tc\right)^{\dR_{\Tc_0\JT}}:\Tc\Alg\to \Tc\Alg,\quad R\mapsto R^{\red_{\Tc_0\JT}};\\
			\left(\AA^1_\Tc\right)^{\dR_{\Tc}}:\Tc\Alg\to \Tc\Alg,\quad R\mapsto R^{\red_\Tc};\\
			\left(\AA^1_\Tc\right)^{\dR_{\Com}}:\Tc\Alg\to \Tc\Alg,\quad R\mapsto R^{\red_\Com}.
		\end{align*}
		Now for an arbitrary prestack \(\X\) one defines its \(\Tc_0\JT,\Tc, \Com\)-de Rham stacks as the composite functor of the following form.
		\[
			\X^{\dR_{\Tc}}:=\X\circ \left(\AA^1_\Tc\right)^{\dR_{\Tc_0\JT, \Tc, \Com}}.
		\]
		More generally given a morphism of Fermat theories \(\Tc\to \Tc'\) one defines the \(\Tc\)-de Rham stack of a \(\Tc'\)-stack \(\Xc\) as follows.
		\[
			\X^{\dR_{\Tc}}:=\X\circ \left(\oblv^{\Tc'}_\Tc\AA^1_{\Tc'}\right)^{\dR_{\Tc}}.
		\]
		Here \(\left(\oblv^{\Tc'}_\Tc\AA^1_{\Tc'}\right)^{\dR_{\Tc}}\) remains a \(\Tc'\)-stack since for a \(\Tc'\)-algebra \(R\) the reduction \((\oblv^{\Tc'}_\Tc R)^{\red_{\Tc}}\) is a \(\Tc'\)-algebra.
	\end{construction}
	\begin{remark}
		We observe that if the ring \(\Tc_0\) is ordered one shows identically to the proof of \Cref{st:CI-inf_radical=RJ-radical} that the \(\Tc_0\JT\)-de Rham prestack of a derived \(\Tc\)-prestack is equivalent to the \(\Tc\)-de Rham prestack.
	\end{remark}
	The following slightly high-brow name for a basic principle observed both in \(\CI\) and analytic contexts is suggested by the work of Scholze on the geometrization of the real local Langlands correspondence \cite[Chapter II]{Scholze_Langlands}.
	\begin{principle}[{Riemann--Hilbert correspondence}]
		Assume that for finitely presented derived \(\Tc\)-algebras there is an equivalence of \(\Tc_0\JT\) and \(\Tc\)-radicals. Then there is an equivalence of de Rham stacks as follows.
		\[
			\Xc^{\dR_\Tc}\simeq \Xc^{\dR_{\Tc_0\JT}}
		\]
	\end{principle}
	\subsubsection{Fermatic base change}
	\begin{theorem}\label{st:fermatic_base_change}
		Let \(\Tc\to \Tc'\) be a morphism of Fermat theories. Then there is the following equivalence for any derived \(\Tc\)-stack \(\X\):
		\[
			\free^{\Tc'}_\Tc\left(\X^{\dR_{\Tc}}\right)\simeq\left(\free^{\Tc'}_\Tc \X\right)^{\dR_{\Tc}}.
		\]
		Moreover, these equivalent stacks are in general not equivalent to \[\left(\free^{\Tc'}_\Tc \X\right)^{\dR_{\Tc'}}.\]
	\end{theorem}
	\begin{proof}
		By \Cref{con:de_Rham_from_ring_stacks}, it is enough to verify the statement for de Rham stacks of \(\AA^1\). We have the following chain of equivalences for the first case.
		\begin{gather*}
			\free^{\Tc'}_\Tc\left(\left(\AA^1_\Tc\right)^{\dR_\Tc}\right)(R)\simeq \left(\AA^1_\Tc\right)^{\dR_{\Tc}}(\oblv^{\Tc'}_\Tc R)\simeq \left(\oblv^{\Tc'}_\Tc R\right)^{\red_\Tc}.
		\end{gather*}
		For the second case, we have a similar chain of equivalences.
		\begin{gather*}
			\left(\left(\free^{\Tc'}_{\Tc}\AA^1_\Tc\right)\right)^{\dR_{\Tc'}}(R)\simeq \left(\free^{\Tc'}_{\Tc}\AA^1_\Tc\right)(R^{\red_{\Tc'}})\simeq \oblv^{\Tc'}_\Tc\left(R^{\red_{\Tc'}}\right).
		\end{gather*}
		In general \(\oblv^{\Tc'}_\Tc\left(R^{\red_{\Tc'}}\right)\) is quite different from \(\left(\oblv^{\Tc'}_\Tc R\right)^{\red_\Tc}\) as we saw, for example, with \(\Tc=\Com\) and \(\Tc'=\CI\).
	\end{proof}
	\begin{remark}
		The conceptual reason why \Cref{st:fermatic_base_change} holds is that the de Rham stack of a derived \(\Tc\)-stack is defined via the ring stack \(\AA^1\). That, in turn, is \emph{defined over} \(\Com\) and hence over any other Fermat theory.
	\end{remark}
	\begin{definition}\label{def:point_preserving_morphism}
		A morphism of Fermat theories \(\Tc\to \Tc'\) is called \emph{point preserving} if the notion of \(\Tc\JT\)-reducedeness and \(\Tc'\JT\)-reducedeness coincide. The following identity holds. 
		\[
			\left(\oblv^{\Tc'}_\Tc R\right)^{\red_{\Tc_0\JT}}=\oblv^{\Tc'}_\Tc \left(R^{\red_{\Tc_0'\JT}}\right)
		\]
	\end{definition}
	\begin{remark}
		Equivalently, the condition of \Cref{def:point_preserving_morphism} can be stated as follows. In the underlying derived commutative algebra of a derived \(\Tc'\)-algebra the \(\Tc\JT\)-radical and \(\Tc'\JT\)-radical coincide. The following identity holds.
		\[
			\sqrt[\Tc_0\JT]{0}=\sqrt[\Tc_0'\JT]{0}
		\]
	\end{remark}
	\begin{example}
		There are several immediate examples of such morphisms. 
		\begin{enumerate}
			\item The inclusion of polynomials over \(\R\) into \(\CI\)-functions.
			\[
				\Com_\R\to \CI
			\]
			\item The inclusion of real analytic functions into \(\CI\)-functions.
			\[
				\mc{C}^\omega\to \CI
			\]
			\item The inclusion of polynomials over \(\C\) into holomorphic functions.
			\[
				\Com_\C\to \O
			\]
		\end{enumerate}
		
	\end{example}
	\begin{theorem}\label{st:TJ-de_Rham}
		Given a point preserving morphism of Fermat theories \(\Tc\to \Tc'\), there is an equivalence of de Rham stacks as follows.
		\[
			\free^{\Tc'}_\Tc \left(\Xc^{\dR_{\Tc_0\JT}}\right)=\left(\free^{\Tc'}_\Tc \Xc\right)^{\dR_{\Tc_0'\JT}}
		\]
	\end{theorem}
	The proof is similar to the proof of \Cref{st:fermatic_base_change}.
	\begin{proof}
		Once again, we only need to prove this statement for the de Rham stack of \(\AA^1\). We have the following chain of equivalences.
		\begin{gather*}
			\free^{\Tc'}_\Tc\left(\left(\AA^1_\Tc\right)^{\dR_{\Tc_0\JT}}\right)(R)\simeq \left(\AA^1_\Tc\right)^{\dR_{\Tc_0\JT}}(\oblv^{\Tc'}_\Tc R)\simeq \left(\oblv^{\Tc'}_\Tc R\right)^{\red_{\Tc_0\JT}}.
		\end{gather*}
		Similarly, we have a chain of equivalences.
	\begin{gather*}
			\left(\free^{\Tc'}_\Tc\left(\AA^1_\Tc\right)\right)^{\dR_{\Tc_0'\JT}}(R)\simeq \left(\free^{\Tc'}_\Tc\left(\AA^1_\Tc\right)\right)(R^{\red_{\Tc_0'\JT}})\simeq \oblv^{\Tc'}_\Tc \left(R^{\red_{\Tc_0'\JT}}\right).
		\end{gather*}
		It remains to use the point preserving property of the morphism \(\Tc\to \Tc'\) to show that the two expressions are equivalent.
	\end{proof}
	An immediate corollary of this result is the following.
	\begin{corollary}
		Given a point preserving morphism of Fermat theories \(\Tc\to \Tc'\). The cohomology of the \(\Tc'\) shape can be calculated using the \(\Tc\JT\)-de Rham cohomology. That is, for any derived stack \(\Xc\) over \(\Tc\), there is the following equivalence.
		\[
			\Map\left(\Xc^{\dR_{\Tc_0\JT}},\oblv^{\Tc'}_{\Tc}\AA^1_{\Tc'}\right)=\Map\left(\Xc^{\dR_{\Tc_0'\JT}}, \AA^1_{\Tc'}\right).
		\]	
	\end{corollary}
	\begin{proof}
		We have the following equivalence given by \Cref{st:TJ-de_Rham} and adjunction.
		\[
			\Map\left(\free^{\Tc'}_\Tc \left(\Xc^{\dR_{\Tc_0\JT}}\right),\AA^1_{\Tc'}\right)=\Map\left(\Xc^{\dR_{\Tc_0\JT}},\oblv^{\Tc'}_{\Tc}\AA^1_{\Tc'}\right)=\Map\left(\Xc^{\dR_{\Tc_0'\JT}}, \AA^1_{\Tc'}\right).
		\]
		It remains to apply \Cref{thm:de_Rham_TJ_vs_locale_coh} to the right-hand side of the above equivalence.
	\end{proof}
	\subsubsection{Further generalizations}
	One can consider also a version of the theory where instead of the morphism \(\Com\to \Tc\) given in \Cref{eq:can_morphism_from_com} one considers an arbitrary pair of theories \(\Tc\to \Tc'\) and the corresponding comparison morphism on the de Rham stacks. The ease with which base change for Fermat theories allows one to change the ``geometry of definition'' for a given stack gives a compelling argument for why it should be possible to think of theorems like GAGA or the Riemann--Hilbert correspondence in terms of morphisms of theories. 
	
	\section{H.-J. Reiffen's example of analytic non-contractibility}\label{Appendix:Reiffen}
	This appendix is concerned with the following question: given a zero set of a smooth function, is the de Rham cohomology of its structure sheaf isomorphic to the singular cohomology of its underlying topological space? Of course, \Cref{ex:de_Rham_counterexample} shows this is not true for sufficiently ``bad'' functions. However, one might expect that if the function is analytic or even a polynomial, then the classical de Rham cohomology should be isomorphic to the singular cohomology. The answer, however, is still negative.
	Below, we show that the counterexample of H.-J. Reiffen in \cite{Reiffen1967} to the holomorphic Poincaré lemma for singular complex hypersurfaces also provides a counterexample to the classical de Rham theorem once restricted to the real points.
	\subsection{The (non-)exact sequence of differential forms}
	The results in this section are due to H.-J. Reiffen \cite[\S 1]{Reiffen1967}, we merely observe that they hold in the \(\CI\)-setting as well. 
	\begin{construction}\label{con:sheaves_K_Reiffen}
		Assume, that \(X\) is a subset of \(\R^n\) containing \(0\) defined by smooth equations. Denote by \(\I\) the ideal defining \(X\) in the local ring \(\CI(\R^n)_0\). Denote by \(\Omega_0^\bt\) the classical de Rham complex of \(\R^n\) at the origin.
		Denote by \(\K^k\) the submodule of the module of \(k\)-forms consisting of the forms admitting the following presentation.
		\[
		\omega=\sum_{i=1}^s f_i \alpha_a + \sum_{j=1}^r dg_j \wedge \beta_j,\quad f_i,g_j\in \I,\; \alpha_i\in \Omega^k_0,\; \beta_j\in \Omega^{k-1}_0.
		\]
		As a convention, we put \(\K^0=\I\). 
	\end{construction}
	\begin{remark}\label{rem:inclusion}
		Observe that the collection \(\K^\bullet\) forms a differential ideal in the classical de Rham complex \(\Omega_0^{\bullet}\). We also note that the following identity holds \(d(\I\Omega^k_0)=d(\K^k)\). 
	\end{remark}
	\begin{lemma}\label{lem:description_of_naive_de_Rham_Reiffen}
		In the notation of \Cref{con:sheaves_K_Reiffen}
		The stalk of the classical de Rham complex \(\Omega^\bt_{X,0}\) of \(\CI(\R^n)_0/\I\) at the origin is given by the quotient complex \(\Omega_0^{\bullet}/\K^{\bullet}\) restricted to \(X\). 
	\end{lemma}
	\begin{proof}
		By definition of the classical de Rham complex, the \(k\)-th term of the complex is given by the exterior power of the module \(\Omega^1_0/\I\). The differential is given by the exterior derivative. The result follows from the definition of the module \(\K^k\).
	\end{proof}
	\begin{lemma}[{\cite[Korollar 1]{Reiffen1967}}]\label{lem:exactness_Reiffen}
		The complex \(\Omega^\bt_{X,0}\) is exact in the \(k\)-th term if and only if the \(k\)-th term of the complex \(\K^\bt\) is exact.
	\end{lemma}
	\begin{proof}
		This follows from \Cref{lem:description_of_naive_de_Rham_Reiffen}. Indeed, there is a long exact sequence in cohomology coming from a short exact sequence of complexes.
		\[
		0 \to \K^\bt\to \Omega_0^{\bt}\to \Omega^\bt_{X,0} \to 0
		\]
		It remains, to note that by the standard Poincaré lemma the complex \(\Omega_0^{\bt}\) is exact.
	\end{proof}
	\begin{proposition}[{\cite[\S 3]{Reiffen1967}}]\label{prop:Reiffen_eq_crit}
		The complex \(\Omega^\bt_{X,0}\) is acyclic in degree \(n-1\) if and only if there is the following inclusion of modules \(\I\Omega^n_0\sse d(\I\Omega^{n-1}_0)\).  
	\end{proposition}
	For reader's convenience we provide a proof of the above proposition due to Reiffen.
	\begin{proof}
		We present the proof in several steps for clarity. The basic point is that because we only consider forms of degrees \(n,n-1,n-2\), we can relate \(\I\Omega^\bt_0\) and \(\K^\bt\) much easier than in general degrees.
		\begin{enumerate}
			\item By \Cref{lem:exactness_Reiffen} the complex \(\Omega_{X,0}^n\) is exact in degree \(n-1\) if and only if there is an inclusion \(\K^n\sse d(\K^{n-1})\). 
			\item\label{reiffen_step_2} We claim that the inclusion \(\K^n\sse d(\K^{n-1})\) holds if and only if the inclusion \(\I\Omega^n_0\sse d(\K^{n-1})\) holds. 
			In the forward direction, this is implied by the inclusion
			\(\I\Omega_0^n\sse \K^n.\) 
			The other direction is also clear since we have the following inclusion.
			\[
				\K^n = \I \Omega^n_0 + \left(d\I\right) \wedge \Omega^{n-1}_0\sse d\K^{n-1}.
			\]
			\item\label{reiffen_step_4} We claim that the inclusion \(\I\Omega^n_0\sse d(\K^{n-1})\) holds if and only if the inclusion \({d\I\wedge \Omega^{n-1}_0\sse d(\I\Omega^{n-1}_0)}\) holds. We show that the forward direction holds using the following calculation.
			\begin{itemize}
				\item If \(\omega\in d\I\wedge \Omega^{n-1}_0\) then \(\omega\) can be written as
				\[
					\omega = \sum_{j=1}^r dg_j \wedge \beta_j, \quad g_j \in \I, \quad \beta_j \in \Omega_0^{n-1}
				\]
				\item Observe that there is the following identity implying that \(\omega\in d\K^{n-1}\).
				\[
					\sum_{j=1}^r dg_j \wedge \beta_j = d\left( \sum_{j=1}^r g_j \wedge \beta_j \right) - \sum_{j=1}^r g_j \, d\beta_j\in d\K^{n-1}.
				\]
			\end{itemize}
			In the opposite direction, we use the following calculation.
			\begin{itemize}
				\item Consider a form \(\omega\in \I\cdot \Omega^n_0\). Then \(\omega\) can be written as follows.
				\[
					\omega = \sum_{j=1}^s f_j \alpha_j, \quad f_j \in \I, \quad \alpha_j \in \Omega_0^n
				\]
				\item Note that by exactness of the stalk of the de Rham complex in \(\R^n\) we can present each \(\alpha_j\) as \(\alpha_j = d\gamma_j\) since all of them are top-degree and hence closed. Using this observation, we obtain the following identity. 
				\[
					\sum_{j=1}^s f_j \alpha_j = d\left( \sum_{j=1}^s f_j \gamma_j \right) - \sum_{j=1}^s df_j \wedge \gamma_j\in d\K^{n-1}.
				\]
				Here we use in an essential way the fact that we consider the \textbf{top-degree forms}.
			\end{itemize}
			\item We claim that the inculsion \(\I\Omega^n_0\sse d(\K^{n-1})\) holds if and only if the inclusion \(\I\Omega^n_0\sse d(\I\Omega^{n-1}_0)\) holds. The forward direction is evident. In the opposite direction we use step \ref{reiffen_step_4} to see that there is an inclusion \(d\I\wedge \Omega^{n-1}_0\sse d(\I\Omega^{n-1}_0).\) It remains to use step \ref{reiffen_step_2} and the following inclusion.
			\[
				\K^n = \I \Omega^n_0 + \left(d\I\right) \wedge \Omega^{n-1}_0\sse d\K^{n-1}.
			\] 
		\end{enumerate}
		Thus we showed that the inclusion \(\I\Omega^n_0\sse d(\I\Omega^{n-1}_0)\) holds if and only if the inclusion \(\I\Omega^n_0\sse d(\K^{n-1})\) holds. This, by the first step, is equivalent to the exactness of the complex \(\Omega_{X,0}^n\) in degree \(n-1\).
	\end{proof}
	The following result is the crucial criterion for the naive Poincaré lemma to hold or fail on a \(\CI\)-hypersurface (i.e. a zero locus of a smooth function).  
	\begin{corollary}[{\cite[Korollar 4]{Reiffen1967}}]\label{cor:Reiffen_crit}
		Let $f \in \CI(\R^n)$ ($n \geq 2$) be a smooth function such that \(f(0)=0.\) Then the classical Poincaré lemma holds at the origin in degree \(n-1\)  if and only if for any \(g\in \CI(\R^n)_0\) there exist smooth functions \(h_1,\dots,h_n\) such that
		\begin{equation}\label{eq:Reiffen_hypersurface_crit}
			f\cdot g = \sum_{i=1}^n \frac{\partial (f\cdot h_i)}{\partial x_i}.
		\end{equation}
	\end{corollary}
	\begin{proof}
		This is immediate from \Cref{prop:Reiffen_eq_crit} after fixing a generator of \(\Omega^n_0\). 
	\end{proof}
	\subsection{The example}
	We consider the following polynomial smooth function
	\[
		x^4+y^5+y^4x:\R^2\to \R.
	\]
	\begin{theorem}\label{thm:Reiffen_counterexample}
		The classical de Rham cohomology of the zero set \(X\) of the function \(f\) is not isomorphic to its singular cohomology.
	\end{theorem}
	To prove the above result we need the following proposition.
	\begin{proposition}
		The Poincaré lemma at the origin fails in degree \(1\) for the zero locus of the function~\({x^4+y^5+y^4x}\).
	\end{proposition}
	To see this, we essentially use Reiffen's argument, observing that the restriction on the analyticity of solutions is irrelevant.
	\begin{proof}
		We will prove the statement by showing that \Cref{eq:Reiffen_hypersurface_crit} has no solution for the function \(f=x^4+y^5+y^4x\) and \(g=1\). Consider Taylor expansions of \(h_1\) and \(h_2\) at the origin.
		\[
			h_1(x,y)=\sum_{i,j\le 1} A_{ij}x^iy^j + \vp(x,y), \quad h_2(x,y)=\sum_{i,j\le 1} B_{ij}x^iy^j+\psi(x,y).
		\]
		Here \(\vp\) and \(\psi\) are some functions which vanish at the origin with order at least \(2\). Now we compute the right-hand side of \Cref{eq:Reiffen_hypersurface_crit} using the Taylor expansions. 
		\begin{align*}
			x^4+y^5+y^4x = \frac{\pd}{\pd x}\left(A_{00} (x^4+y^5+y^4x)\right) + \frac{\pd}{\pd y}\left(B_{00} (x^4+y^5+y^4x)\right) + \\
			\frac{\pd}{\pd x}\left(\left(A_{10}x+A_{01}y\right) (x^4+y^5+y^4x)\right) + \frac{\pd}{\pd y}\left(\left(B_{10}x+B_{01} y\right) (x^4+y^5+y^4x)\right) + \ldots
		\end{align*}
		Below, we analyze the low-degree terms of this equation.
		\begin{itemize}
			\item[\ul{Term \(x^3\)}:] We have the following equation for the homogeneous term \(x^3\).
			\[
				0=4A_{00}x^3
			\]
			Hence \(A_{00}=0\). 
			\item[\ul{Term \(y^4\)}:] We have the following equation for the homogeneous term \(y^4\).
			\[
				0=A_{00}y^4+5B_{00}y^4.
			\]
			Hence \(B_{00}=0\).
			\item[\ul{Term \(x^4\)}:] We have the following equation for the homogeneous term \(x^4\).
			\[
				x^4=(A_{10}x\cdot 4x^3+A_{10}x^4)+B_{01}x^4.
			\]
			Hence \(5 A_{10}+B_{01}=1\).
			\item[\ul{Term \(y^5\)}:] We have the following equation for the homogeneous term \(y^5\).
			\[
				y^5=A_{10} y^5+B_{01}y^5+B_{01}y\cdot 5y^4.
			\]
			Hence \(A_{10}+6B_{01}=1\).
			\item[\ul{Term \(y^4x\)}:] We have the following equation for the homogeneous term \(y^4x\).
			\[
				y^4x=A_{10}\cd y^4x+A_{10}x\cdot y^4+B_{01} \cd y^4x+B_{01}y\cdot 4y^3x
			\]
			Hence \(2A_{10}+5B_{01}=1\).
		\end{itemize}
		As a result, the low-degree terms of \Cref{eq:Reiffen_hypersurface_crit} for \(g=1\) give us the following system of equations.
		\begin{align*}
			\begin{cases}
			5\cd A_{10}+B_{01}=1,\\
			A_{10}+6B_{01}=1,\\
			2A_{10}+5B_{01}=1.
			\end{cases}
		\end{align*}
		This system of equations has no solutions (since the system is overdetermined which could be verified by calculating the determinant of the extended matrix of this system). This concludes the proof. Note that here we didn't use the fact that the functions \(h_1\) and \(h_2\) are analytic, as was the case for Reiffen, only that their Taylor expansions must satisfy certain conditions.
	\end{proof}
	\begin{remark}
		We focused here on a single polynomial, however in their original work Reiffen shows using the same argument local homological non-contractability for the following family of polynomials:
		\[
			x^q+y^p+y^{p-1}x=0,\quad q\ge 4,\;p\ge q+1.
		\]
	\end{remark}
	\begin{proof}[\ul{Proof of \Cref{thm:Reiffen_counterexample}}]
		To see that there is no isomorphism between the classical de Rham cohomology and the singular cohomology we observe that the stalk of the classical de Rham cohomology at the origin is non-trivial. Observe that the classical de Rham theorem holds for the non-singular points of \(X\). We have the following diagram of inclusions.
		\[
		\{0\} \xhookrightarrow{i} X \xhookleftarrow{j} X\setminus \{0\}. 
		\]
		Observe that we have the following recollement fiber sequences.
		\[\begin{tikzcd}
			{j_!j^! \ul{\R}} && {\ul{\R}} && {i_*i^*\ul{\R}} \\
			{j_!j^!\Omega^\bt_X} && {\Omega^\bt_X} && { i_*i^*\Omega^\bt_X}
			\arrow[from=1-1, to=1-3]
			\arrow[from=1-1, to=2-1]
			\arrow[from=1-3, to=1-5]
			\arrow[from=1-3, to=2-3]
			\arrow[from=1-5, to=2-5]
			\arrow[from=2-1, to=2-3]
			\arrow[from=2-3, to=2-5]
		\end{tikzcd}\]
		The result is now immediate from the long exact sequence of cohomology.
	\end{proof}
	\subsubsection{}
	Note that by \Cref{thm:de_Rham_for_Lojasiewicz} the derived de Rham cohomology of the zero set of any polynomial is isomorphic to the singular cohomology of its underlying topological space. Hence, the above result shows that the classical de Rham cohomology is not isomorphic to the derived de Rham cohomology and the singular cohomology. In fact, we also show that the classical de Rham cohomology is strictly larger than singular cohomology.
	
\end{document}